\documentclass[12pt, twoside]{article}
\usepackage{amssymb}
\usepackage{amsmath}
\usepackage{cases}
\usepackage{txfonts}
\usepackage{amsfonts}
\usepackage{amsmath,amssymb}
\usepackage{multicol}
\usepackage{color}
\usepackage{graphicx}
\usepackage{graphics}

\def\R{\mathbb{R}}

\def\T{\mathbb{T}}
\def\epsilon{\varepsilon}
\newcommand{\be}{\begin{equation}}
\newcommand{\ee}{\end{equation}}
\newcommand{\baa}{\begin{array}}
\newcommand{\eaa}{\end{array}}
\newcommand{\ba}{\begin{eqnarray}}
\newcommand{\ea}{\end{eqnarray}}

\newtheorem{theorem}{Theorem}[section]
\newtheorem{lemma}[theorem]{Lemma}
\newtheorem{corollary}[theorem]{Corollary}

\newtheorem{definition}[theorem]{Definition}
\newtheorem{remark}[theorem]{Remark}

\newtheorem{claim}[theorem]{Claim}
\numberwithin{equation}{section}
\newenvironment{proof}[1][Proof]{\noindent\textbf{#1.} }{\hfill $\Box$}
\allowdisplaybreaks

\makeatletter
\begin{document}
\date{}
\title{\bf{Curved fronts of bistable reaction-diffusion equations in spatially periodic media \thanks{Guo and Liu are partially supported by by NSF grant 1826801. Li is supported by the National Natural Science
Foundation of China under grants 11731005 and 11671180. Wang is supported by the National Natural Science
Foundation of China under grant 11371179.}}}
\author{Hongjun Guo  \textsuperscript{a}, Wan-Tong Li \textsuperscript{b}, Rongsong Liu \textsuperscript{a} and Zhi-Cheng Wang \textsuperscript{b}\\
\\
\footnotesize{\textsuperscript{a} Department of Mathematics and Statistics, University of Wyoming, Laramie, USA}\\
\footnotesize{\textsuperscript{b} School of Mathematics and Statistics, Lanzhou University,
Lanzhou, China}}
\maketitle

\begin{abstract} In this paper, curved fronts are constructed for spatially periodic bistable reaction-diffusion equations under the a priori assumption that there exist pulsating fronts in every direction. Some sufficient and some necessary conditions of the existence of curved fronts are given. Furthermore, the curved front is proved to be unique and stable. Finally, a curved front with varying interfaces is also constructed.  Despite the effect of the spatial heterogeneity, the result shows the existence of curved fronts for spatially periodic bistable reaction-diffusion equations which is known for the homogeneous case. 
\vskip 0.1cm
\noindent\textit{Keywords. Pulsating front; Curved front; Spatially periodic reaction-diffusion equation; Uniqueness; Stability.}

\end{abstract}


\section{Introduction}
\noindent
In this paper, we consider spatially periodic reaction-diffusion equations of the type
\begin{equation}\label{eq1.1}
u_t-\Delta u=f(x,y,u),\quad\,(t,x,y)\in\R\times\R^2,
\end{equation}
where $u_t=\frac{\partial u}{\partial t}$ and $\Delta=\partial_{xx}+\partial_{yy}$ denotes the Laplace operator with respect to the space variables $(x,y)\in\R^2$. The reaction term $f(x,,y,u)$ is assumed to be periodic in $(x,y)$ and bistable in $u$. More precisely,  we assume throughout this paper that
\begin{itemize}
  \item[(F1)] $f(x,y,u)$ is continuous, of class $C^{\alpha}$ in $(x,y)$ uniformly in $u\in [0,1]$ with $\alpha\in (0,1)$, of class $C^2$ in $u$ uniformly in $(x,y)\in\R^2$ with $f_{u}(x,y,u)$ and $f_{uu}(x,y,u)$ being Lipschitz continuous in $u\in\R$;
  \item[(F2)] $f(x,y,u)$ is $L$-periodic with respect to $(x,y)$ where $L=(L_1,L_2)\in\R^2$, that is, $f(x+k_1L_1,y+k_2L_2,u)=f(x,y,u)$ for any $k_1$, $k_2\in\mathbb{Z}$;
  \item[(F3)] for every $(x,y)\in\R^2$, $0$ and $1$ are stable zeroes of $f(x,y,\cdot)$, that is, 
$$f(x,y,0)=f(x,y,1)=0,$$
and there exist $\lambda>0$ and $\sigma\in (0,1/2)$ such that
$$-f_u(x,y,u)\ge \lambda \hbox{ for all $(x,y,u)\in \R^2\times [0,\sigma]$ and $(x,y,u)\in \R^2\times [1-\sigma,1]$}.$$
\end{itemize}
A typical example of $f(x,y,u)$ is the cubic nonlinearity
$$f(x,y,u)=u(1-u)(u-\theta_{x,y}),$$
where $\theta_{x,y}\in (0,1)$ is a $L$-periodic function. The equation \eqref{eq1.1} is a special generalization of the famous Allen-Cahn equation \cite{AC}. For mathematical convenience, we extend $f(x,y,u)$ out of the interval $u\in[0,1]$ such that
\be\label{lambda}
-f_u(x,y,u)\ge \lambda \hbox{ for all $(x,y,u)\in \R^2\times (-\infty,\sigma]$ and $(x,y,u)\in \R^2\times [1-\sigma,+\infty)$}.
\ee
Then, $f(x,y,u)$ is globally Lipschitz continuous in $u\in\R$.

Before going further on, we first recall some well-known results in the homogeneous case, that is, 
\be\label{homo-eq}
u_t-\Delta u=f(u),\ (t,x)\in\R\times\R^N,
\ee
where $f$ is of bistable type, that is, $f(0)=f(1)=f(\theta)$, $f<0$ on $(0,\theta)$ and $f>0$ on $(\theta,1)$, for some $\theta\in(0,1)$. For one-dimensional space, it follows from celebrated results due to Fife and McLeod \cite{FM} that \eqref{homo-eq} admits a unique (up to shifts) traveling front $\phi(x-c_f t)$ satisfying
$$0<\phi<1,\ \phi(-\infty)=1 \hbox{ and } \phi(+\infty)=0.$$
Moreover, the speed $c_f$ has the sign of $\int_0^1 f(u)du$ and the front is globally and exponentially stable. A trivial extension of the traveling front to higher dimensional spaces is the planar front $\phi(x\cdot e- c_f t)$ where $e\in\mathbb{S}^{N-1}$ denotes the propagation direction. Notice that every level set of a planar front is a plane. Except planar fronts, more types of fronts are also known to exist in high dimensional spaces, such as $V$-shaped fronts, conical shaped fronts and pyramidal fronts, see Hamel et. al. \cite{HMR1}, Ninomiya and Taniguchi \cite{NT1} and Taniguchi  \cite{T1,T2}. All these fronts are transition fronts connecting $0$ and $1$ defined by Hamel \cite{H}. For above fronts, their interfaces between $0$ and $1$ can be given by their level sets and different shapes of interfaces actually show some structures of the solutions. One can roughly imagine a global appearance of such solutions in the framework of transition fronts by noticing that the solutions are approaching to $1$ and $0$ on one side and the other of the interfaces, respectively.

As far as a spatially periodic bistable reaction-diffusion equation considered, the situation is more complicated than the homogenous case. Because of the effect of hetereogeneities, there may even not exist transition fronts connecting states $0$ and $1$, see Zlato{\v{s}} \cite{Z3}. However, what we are concerned in this paper is the existence of curved fronts when there exist some fronts in every direction, that is, pulsating fronts. We now introduce the notion of pulsating front by referring to \cite{BH0,SKT,X1,X2,X3}.
\begin{definition}\label{PF}
Denote a periodic cell by $\mathbb{T}^2=[0,L_1]\times [0,L_2]$. A pair $(U_e,c_e)$ with $U_e:\R\times\mathbb{T}^2\rightarrow \R$ and $c_e\in\R$ is said to be a pulsating front of \eqref{eq1.1} with effective speed $c_e$ in the direction $e\in\mathbb{S}$ connecting $0$ and $1$ if the two following conditions are satisfied:
\begin{itemize}
\item[(i)] For every $\xi\in\R$, the profile $U_e(\xi,x,y)$ is $L$-periodic in $(x,y)$ and satisfies
$$\lim_{\xi\rightarrow +\infty} U_e(\xi,x,y)=0,\ \lim_{\xi\rightarrow -\infty} U_e(\xi,x,y)=1,\,\text{ uniformly for $(x,y)\in\mathbb{T}^2$}.$$
\item[(ii)] The map $u(t,x,y):=U_e((x,y)\cdot e-c _e t,x,y)$ is an entire (classical) solution of the parabolic equation \eqref{eq1.1}.
\end{itemize}
\end{definition}

We now recall some existence results of pulsating fronts for the general reaction-diffusion equation in spatially periodic media
\be\label{1.6}
u_t=\sum_i (a(x)u_{x_i})_{x_i} +\sum_i b_i(x)u_{x_i}+f(x,u), \ t\in\R,\ x\in\R^N.
\ee
For one dimensional case of \eqref{1.6} when $f(x,u)=g(x)f(u)$, Nolen and Ryzhik \cite{NR4} proved the existence of pulsating fronts with nonzero speed by provided with some restrictions for $g$ and $f$. Moreover, Ducrot, Giletti and Matano \cite{DGM} also got some existence results of pulsating fronts with a positive speed, if the solutions of \eqref{1.6} with some compactly supported initial conditions can converge locally uniformly to $1$ as $t\rightarrow +\infty$. Still for one-dimensional case, Ding, Hamel and Zhao \cite{DHZ} applied the implicit function theorem and abstract results of Fang and Zhao \cite{FZ} to get the existence of pulsating fronts for small period and large period. For higher dimensions, when the diffusivity matrix $a$ is close to identity and $f$ is independent of $x$, the existence of pulsating fronts is obtained by Xin \cite{X1,X2,X3} through refined perturbation arguments. Ducrot \cite{D} also got some existence results of fronts connecting $0$ and $1$ in every direction for slowly varying medium and rapidly varying medium (that is, $d<<1$ and $d>>1$ respectively when the reaction term is $f(dx,u)$), in which the fronts are either moving pulsating waves or standing transition waves. Although such existence results are known, there may not exist pulsating fronts in general. Zlato\v{s} \cite{Z3} constructed a periodic pure bistable reaction such that there is no pulsating fronts of \eqref{eq1.1}. We also refer to \cite{DHZ,X4,XZ} for some nonexistence results.

In this work, we aim to construct curved fronts by using some pulsating fronts with nonzero speeds. Therefore, we need to assume a priori that
\begin{itemize}
\item[(H1)] $\int_{\mathbb{T}^2\times[0,1]} f(x,y,u)dxdydu\neq 0$,

\item[(H2)] for every unit vector $e\in\R^2$, the equation \eqref{eq1.1} admits a pulsating front $U_e((x,y)\cdot e-c _e t,x,y)$ with $c_e\neq 0$.
\end{itemize}
From the results of Ducrot \cite{D} and Guo \cite{G}, one knows that if (H1), (H2) hold, the propagation speed $c_e$ of the pulsating front in every direction has the sign of $\int_{\mathbb{T}^2\times[0,1]} f(x,y,u)dxdydu$. We assume without loss of generality that
\be\label{int-f}
\int_{\mathbb{T}^2\times[0,1]} f(x,y,u)dxdydu>0,
\ee
which implies $c_e>0$ for all $e\in\mathbb{S}$. Otherwise, one can replace $u$, $f$, $U_e(\xi,x,y)$ by $\tilde{u}=1-u$, $g(x,y,u)=-f(x,y,1-u)$, $\tilde{U}_e(\xi,x,y)=1-U_e(-\xi,x,y)$ and consider the new pulsating front~$\tilde{U}_e$ with speed $-c_e$. From \cite{BH2} and \cite{G}, the speed $c_e$ and the profile $U_e$ of the pulsating front are unique up to shifts in time for any direction $e$. We fix the pulsating front in every direction $e$ by 
$$U_e(0,0,0)=\frac{1}{2}.$$
From \cite{G}, we also know that $\partial_{\xi} U_e<0$, the family $\{c_e\}_{e\in\mathbb{S}}$ is uniformly bounded with respect to $e$ and the minimum and maximum of $c_e$ can be reached with the following inequality
$$0<\min_{e\in\mathbb{S}} c_e\le \max_{e\in\mathbb{S}} c_e<+\infty.$$

In the whole paper, we always assume that (F1)-(F3), (H1)-(H2) and \eqref{int-f} hold and we do not repeat it in the sequel. We now focus on construction of curved fronts by some pulsating fronts. To the best of our knowledge, few results of the existence of curved fronts are known for bistable reaction-diffusion in spatially periodic media. However, one can refer to \cite{E,EHH} for the existence of curved fronts of monostable and combustion reaction-diffusion equations with a periodic shear flow and refer to \cite{BW} for a space-time periodic monostable reaction-advection-diffusion equation. Although the pulsating front $U_e((x,y)\cdot e-c_e t,x,y)$ is not exactly planar, every level set is still bounded with a plane. Thus, the pulsating front is also called almost-planar in the framework of transition fronts, see \cite{H}. We try to apply the ideas of Ninomiya and Taniguchi \cite{NT1} which they used for homogeneous bistable case, to construct the curved fronts. But, since the profiles $U_e$ and speeds $c_e$ of pulsating fronts are different in general with respect to the direction $e$, we have to update their ideas. 

We then claim our results. Let $\alpha\in(0,\pi)$. Then, by Assumption~(H2), there exists a pulsating front in the direction $(\cos\alpha,\sin\alpha)$, that is,
$$U_\alpha(x\cos\alpha+y\sin\alpha-c_{\alpha} t,x,y).$$
For any $\alpha$, $\beta\in(0,\pi)$, define
\be\label{U-}
U^-_{\alpha\beta}(t,x,y):=\max\{U_\alpha(x\cos\alpha+y\sin\alpha-c_{\alpha} t,x,y),U_\beta(x\cos\beta+y\sin\beta-c_{\beta} t,x,y)\},
\ee
which is a subsolution of \eqref{eq1.1}.
Our first result shows the existence of a curved front which converges to pulsating fronts along its asymptotic lines under some conditions on angles $\alpha$ and $\beta$. The curved front is actually a transition front connecting $0$ and $1$ whose interfaces can be chosen as a V-shaped curve.

\begin{theorem}\label{th1}
For any $\theta\in (0,\pi)$, let $g(\theta)=c_{\theta}/\sin\theta$.
For any $0<\alpha<\beta<\pi$ such that
\be\label{angle}
\frac{c_{\alpha}}{\sin\alpha}=\frac{c_{\beta}}{\sin\beta}:=c_{\alpha\beta}>\frac{c_{\theta}}{\sin\theta} \hbox{ for any $\theta\in (\alpha,\beta)$},\ g'(\alpha)<0 \hbox{ and } g'(\beta)>0,
\ee
there exists an entire solution $V(t,x,y)$ of \eqref{eq1.1} such that $V_t(t,x,y)>0$ for all $(t,x,y)\in\R\times\R^2$ and
\be\label{Vf}
\lim_{R\rightarrow +\infty} \sup_{x^2+(y-c_{\alpha\beta} t)^2>R^2} \Big| V(t,x,y)-U^-_{\alpha\beta}(t,x,y)\Big|=0.
\ee
\end{theorem}

\begin{remark}
It seems that the conditions $g'(\alpha)<0$ and $g'(\beta)>0$ can not be removed by our methods. These conditions are actually true for homogeneous unbalanced bistable case with the reaction term having positive integration from $0$ to $1$ ($\alpha$ has to be smaller than $\pi/2$ in this case by symmetry and $\beta=\pi-\alpha$), but false for homogeneous balanced bistable case. Moreover, the V-shaped front exists in homogeneous unbalanced bistable case, but does not exist in homogeneous balanced bistable case, see \cite{HM}. Nevertheless, for the balanced case, there exist some fronts whose level sets have an exponential shape for 2-dimensional space and a paraboloidal shape for $N$-dimensional space with $N\ge 3$, see \cite{CGHNR,T3,T4}.
\end{remark}

\begin{remark}
One can easily check that the curved front $V(t,x,y)$ in Theorem~\ref{th1} is a transition front connecting $0$ and $1$ (see \cite{H} for the definition) with sets
$$\Gamma_t:=\{x\le 0, y\in\R; x\cos\alpha+y\sin\alpha-c_{\alpha}t\} \cup\{x> 0, y\in\R;x\cos\beta+y\sin\beta-c_{\beta}t\},$$
$$\Omega_t^{+}:=\{x\le 0, y\in\R; x\cos\alpha+y\sin\alpha-c_{\alpha}t< 0\}\cup \{x> 0, y\in\R; x\cos\beta+y\sin\beta-c_{\beta}t< 0\},$$
and 
$$\Omega_t^{-}:=\{x\le 0, y\in\R; x\cos\alpha+y\sin\alpha-c_{\alpha}t> 0\}\cup \{x> 0, y\in\R; x\cos\beta+y\sin\beta-c_{\beta}t> 0\}.$$
Notice that for any fixed $t$, $\Gamma_t$ is a connected polyline since $c_{\alpha}/\sin\alpha=c_{\beta}/\sin\beta$ and the shape of $\Gamma_t$ is invariant with respect to $t$. Moreover, by the definition of the global mean speed \cite{H}, the curved front $V(t,x,y)$ has a global mean speed equal to $\min\{c_{\alpha},c_{\beta}\}$, in the sense that
$$\frac{d(\Gamma_t,\Gamma_s)}{|t-s|}\rightarrow \min\{c_{\alpha},c_{\beta}\}, \hbox{ as $|t-s|\rightarrow +\infty$}.$$
Here, the distance $d(A,B)$ between two subsets $A$ and $B$ of $\R^2$, is defined by the smallest geodesic distance between pairs of points in $A$ and $B$. Other definition of the distance $\tilde{d}$ like
$$\tilde{d}(A,B)=\min\Big(\sup\{d(x,B); x\in A\}, \sup\{d(y,A);y\in B\}\Big),$$
could be used. Then, there holds that $d(A,B)\le \tilde{d}(A,B)$ and the global mean speed is equal to $\max\{c_{\alpha},c_{\beta}\}$, in the sense that
$$\frac{\tilde{d}(\Gamma_t,\Gamma_s)}{|t-s|}\rightarrow \max\{c_{\alpha},c_{\beta}\}, \hbox{ as $|t-s|\rightarrow +\infty$}.$$
This is different with the homogeneous case, in which the global mean speeds under these two definitions are the same, see \cite{H} and see \cite{GHS} for the underlying domains being exterior domains and domains with multiple branches.
\end{remark}

We then show that the condition \eqref{angle} is not empty, that is, it is satisfied when $\alpha$ close to $0$ and $\beta$ close to $\pi$, see Figure~1.

\begin{figure}\centering
 \includegraphics[scale=0.25]{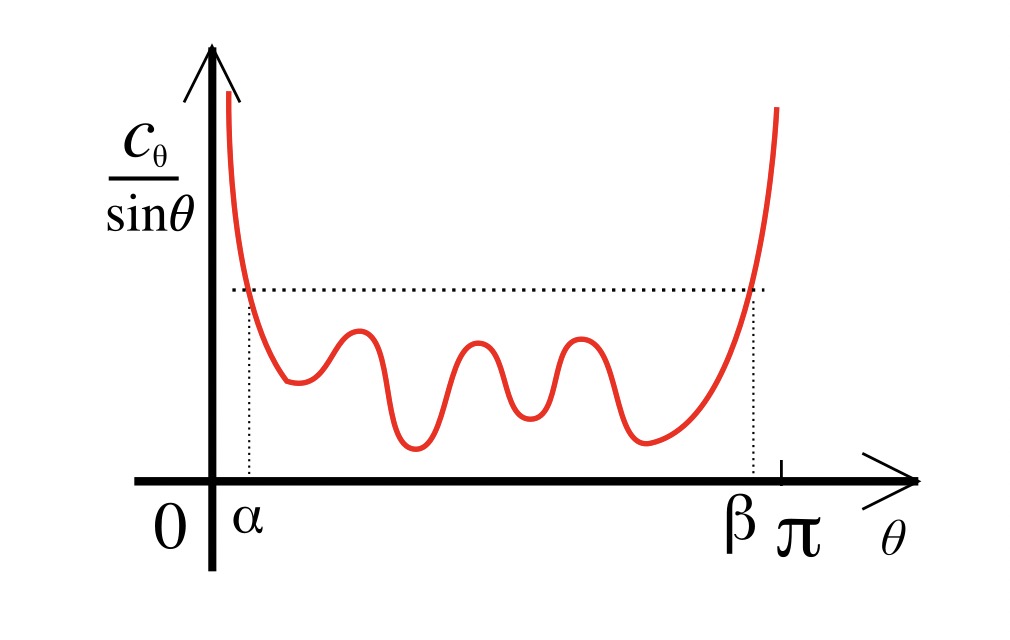}
 \caption{An example of $\alpha$ and $\beta$ satisfying \eqref{angle}.}
\end{figure}

\begin{corollary}\label{cor1}
There exist $0<\alpha_1<\beta_1<\pi$ such that for any $\alpha\in(0,\alpha_1)$, there is $\beta\in (\beta_1,\pi)$ such that \eqref{angle} holds for such $\alpha$, $\beta$ and there exists an entire solution $V(t,x,y)$ of \eqref{eq1.1} satisfying \eqref{Vf}.
\end{corollary}

Indeed, one can rotate the coordinate such that $y$-axis points to any direction. Therefore,  Corollary~\ref{cor1} implies that for any two pulsating fronts whose propagation directions are close to reversed with each other, one can use them to construct a curved front.

\begin{corollary}\label{cor2}
There exist $0<\rho<1$ such that for any directions $e_1$, $e_2$ with $-1<e_1\cdot e_2<-1+\rho$, there exist a direction $e_0$ such that
\be\label{ce1e2}
\frac{c_{e_1}}{\sqrt{1-(e_1\cdot e_0)^2}}=\frac{c_{e_2}}{\sqrt{1-(e_2\cdot e_0)^2}}:=c_{e_1e_2}
\ee
and there is an entire solution $V(t,x,y)$ of \eqref{eq1.1} satisfying 
\be\label{Ve1e2}
\lim_{R\rightarrow +\infty} \sup_{((x,y)-c_{e_1e_2}t e_0)^2>R^2} \Big| V(t,x,y)-U^-_{e_1e_2}(t,x,y)\Big|=0,
\ee
where
$$U^-_{e_1e_2}(t,x,y):=\max\{U_{e_1}((x,y)\cdot e_1-c_{e_1} t,x,y),U_{e_2}((x,y)\cdot e_2-c_{e_2} t,x,y)\}.$$
\end{corollary}

By Theorem~\ref{th1}, one knows that \eqref{angle} is a sufficient condition for the existence of $V(t,x,y)$ satisfying \eqref{Vf}. However, we can not show that \eqref{angle} is necessary, but can show that \eqref{angle} without $g'(\alpha)<0$ and $g'(\beta)>0$ is necessary. 

\begin{theorem}\label{th2}
If there are two angles $\alpha$ and $\beta$ of $(0,\pi)$ and a constant $c_{\alpha\beta}>0$ such that there exists an entire solution $V(t,x,y)$ of \eqref{eq1.1} satisfying \eqref{Vf}, then it holds
$$c_{\alpha\beta}=\frac{c_{\alpha}}{\sin\alpha}=\frac{c_{\beta}}{\sin\beta}>\frac{c_{\theta}}{\sin\theta} \hbox{ for any $\theta\in (\alpha,\beta)$}.$$
\end{theorem}

Now, we show the uniqueness and the stability of the curved front $V(t,x,y)$.

\begin{theorem}\label{th3}
For any fixed $0<\alpha<\beta<\pi$ satisfying 
$$\frac{c_{\alpha}}{\sin\alpha}=\frac{c_{\beta}}{\sin\beta}:=c_{\alpha\beta},$$
the entire solution $V(t,x,y)$ of \eqref{eq1.1} satisfying \eqref{Vf} is unique, that is, if there is an entire solution $V^*(t,x,y)$ satisfying \eqref{Vf}, then $V^*(t,x,y)\equiv V(t,x,y)$.
\end{theorem}

\begin{theorem}\label{th4}
Let $\alpha$ and $\beta$ be fixed angles satisfying \eqref{angle} and $V(t,x,y)$ be the entire solution  of \eqref{eq1.1} satisfying \eqref{Vf}.
Let $0\le u_0(x,y)\le 1$ be an initial value satisfying
\be\label{eq-u0}
\lim_{R\rightarrow +\infty} \sup_{x^2+y^2>R^2} \Big| u_0(x,y)-U_{\alpha\beta}^-(0,x,y)\Big|=0.
\ee
Then, the solution $u(t,x,y)$ of \eqref{eq1.1} for $t>0$ with $u(0,x,y)=u_0(x,y)$ satisfies
$$\lim_{t\rightarrow +\infty} \left\|u(t,x,y)-V(t,x,y)\right\|_{L^{\infty}(\R^2)}=0.$$ 
\end{theorem}

Next, we construct a transition front connecting $0$ and $1$ with varying interfaces. Such a kind of transition front is known in homogeneous case by \cite{H}, in which  the solution is orthogonal symmetric with respect to $y$-axis and behaves as three planar fronts as $t\rightarrow -\infty$. However, in our case, this transition front can not be symmetric in general.

\begin{theorem}\label{th5}
Let $\alpha$ and $\beta$ be fixed angles satisfying \eqref{angle} and let  $V_{\alpha\beta}(t,x,y)$ be the entire solution of \eqref{eq1.1} satisfying \eqref{Vf}. Denote $e_\alpha=(\cos\alpha,\sin\alpha)$ and $e_\beta=(\cos\beta,\sin\beta)$. Assume that there exist another angle $\theta\in(\alpha,\beta)$ and a direction $e_{\theta}=(\cos\theta,\sin\theta)$ such that 
\begin{itemize}
\item[(i)] for $e_{\alpha}$ and $e_{\theta}$, there is a direction $e_{\alpha\theta}$ such that \eqref{ce1e2} holds for $e_1=e_{\alpha}$, $e_2=e_{\theta}$ and $e_0=e_{\alpha\theta}$, it holds $h'(\alpha)<0$ where $h(s)=c_{s}/(e_{\alpha\theta}\cdot (\cos s,\sin s))$ for $0<s<\theta$ and there is an entire solution $V_{\alpha\theta}(t,x,y)$ satisfying \eqref{Ve1e2}.

\item[(ii)] for $e_{\beta}$ and $e_{\theta}$, there is a direction $e_{\beta\theta}$ such that \eqref{ce1e2} holds for $e_1=e_{\beta}$, $e_2=e_{\theta}$ and $e_0=e_{\beta\theta}$, it holds $h'(\beta)>0$ where $h(s)=c_{s}/(e_{\beta\theta}\cdot (\cos s,\sin s))$ for $\theta<s<\pi$ and $e_0=e_{\alpha\theta}$ and there is an entire solution $V_{\beta\theta}(t,x,y)$ satisfying \eqref{Ve1e2}.
\end{itemize}
Then, there exists an entire solution $u(t,x,y)$ of \eqref{eq1.1} such that
\begin{eqnarray}\label{th1.8-1}
u(t,x,y)\rightarrow \left\{\begin{aligned}
V_{\alpha\theta}(t,x,y), &\hbox{ uniformly in the half plane $\{(x,y)\in\R^2; x<0\}$},\\
V_{\beta\theta}(t,x,y)\}, &\hbox{ uniformly in the half plane $\{(x,y)\in\R^2; x>0\}$},
\end{aligned}
\right.
\quad\hbox{as $t\rightarrow -\infty$}.
\end{eqnarray}
and
\be\label{th1.8-2}
u(t,x)\rightarrow V_{\alpha\beta}(t,x,y), \hbox{ as $t\rightarrow +\infty$ uniformly in $\R^2$}.
\ee
\end{theorem}

The convergence in above theorem is in the sense of $L^{\infty}$ norm.

\begin{remark}
From the proof of Theorem~\ref{th5}, one can easily check that the entire solution $u(t,x,y)$ is a transition front connecting $0$ and $1$ with the interfaces
\begin{align*}
\Gamma_t:=&\Big\{x\le \frac{c_{\alpha}\sin\theta -c_{\theta}\sin\alpha }{\sin(\theta-\alpha)}t, y\in\R; x\cos\alpha+y\sin\alpha-c_{\alpha}t=0\Big\} \cup\Big\{\frac{c_{\alpha}\sin\theta -c_{\theta}\sin\alpha }{\sin(\theta-\alpha)}t<x\\
&\le \frac{c_{\beta}\sin\theta -c_{\theta}\sin\beta }{\sin(\theta-\beta)}t, y\in\R;x\cos\theta+y\sin\theta-c_{\theta}t=0\Big\}\cup\Big\{x> \frac{c_{\beta}\sin\theta -c_{\theta}\sin\beta }{\sin(\theta-\beta)}t, y\in\R;\\
&x\cos\beta+y\sin\beta-c_{\beta}t=0\Big\}, \quad \hbox{ for $t\le 0$},
\end{align*}
and
$$\Gamma_t:=\{x\le 0, y\in\R; x\cos\alpha+y\sin\alpha-c_{\alpha}t=0\} \cup\{x> 0, y\in\R;x\cos\beta+y\sin\beta-c_{\beta}t=0\},  \hbox{ for $t >0$},$$
see Figure~2.
\end{remark}

\begin{figure}\centering
 \includegraphics[scale=0.2]{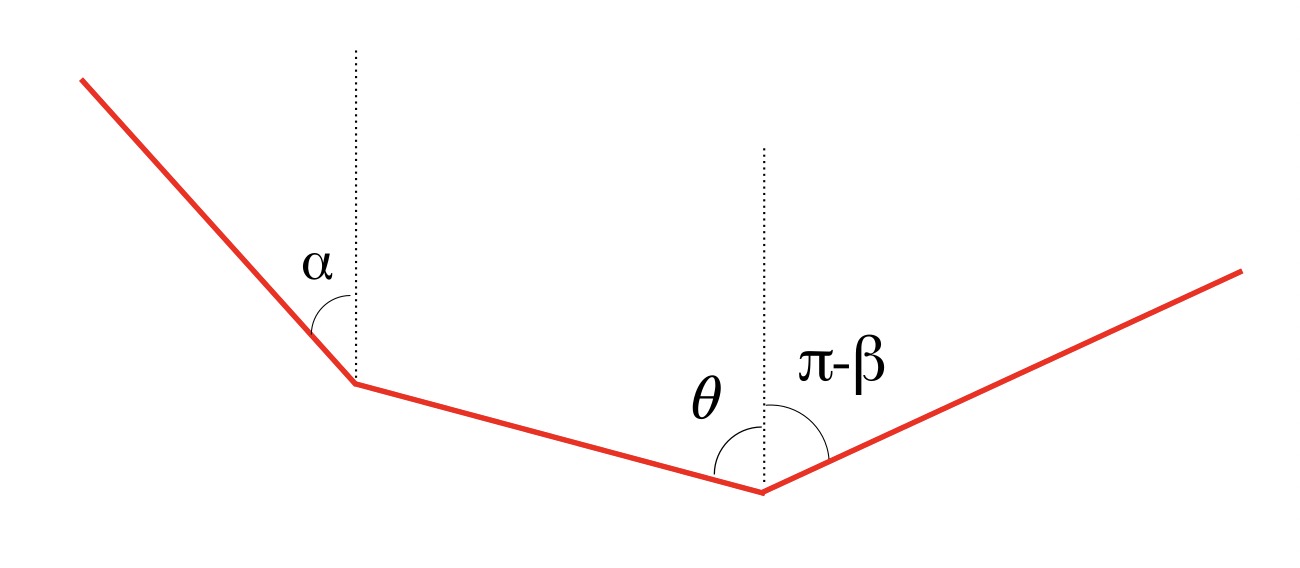}\includegraphics[scale=0.2]{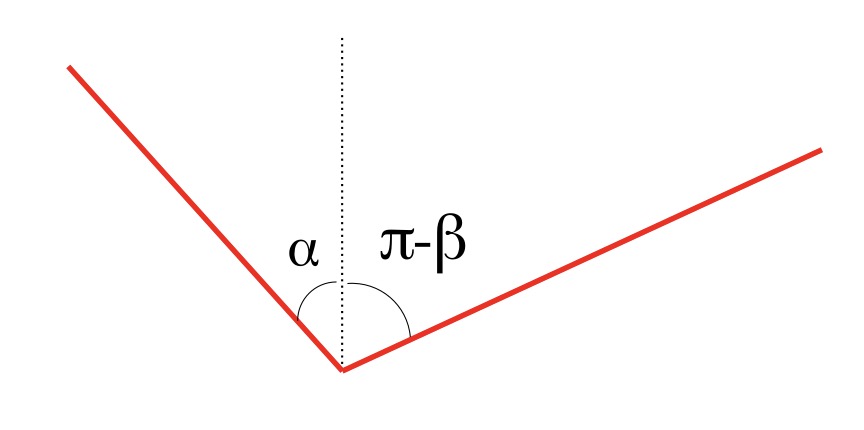}
 \caption{Left: interface when $t<<-1$; Right: interface when $t>>1$.}
\end{figure}

Finally, we give an example showing that Theorem~\ref{th5} is not empty.

\begin{corollary}\label{cor3}
Assume that $e_*$ is the direction such that the family of speeds $\{c_{e}\}_{e\in\mathbb{S}}$ reaches its minimal, that is, $c_{e_*}=\min_{e\in\mathbb{S}}\{c_{e}\}$. Then, there exist $e_1$ and $e_2$ close to $e_*$ such that \eqref{ce1e2} holds for $e_0=e_*$ and there is an entire solution $V_{e_1e_2}(t,x,y)$ of \eqref{eq1.1} satisfying \eqref{Ve1e2}. Moreover, there exist a direction $e_3$ close to $-e_*$ and a direction $e_{**}$ such that there is an entire solution $u(t,x,y)$ of \eqref{eq1.1} such that
\begin{eqnarray*}
u(t,x,y)\rightarrow \left\{\begin{array}{lll}
V_{e_1e_2}(t,x,y), &&\hbox{ uniformly in the half plane $\{(x,y)\in\R^2; (x,y)\cdot e_{**}<0\}$},\\
V_{e_2e_3}(t,x,y)\}, &&\hbox{ uniformly in the half plane $\{(x,y)\in\R^2; (x,y)\cdot e_{**}>0\}$},
\end{array}
\right.
\end{eqnarray*}
as $t\rightarrow -\infty$ 
and
$$u(t,x)\rightarrow V_{e_1e_3}(t,x,y), \hbox{ as $t\rightarrow +\infty$ uniformly in $\R^2$}.$$
\end{corollary}

We organize this paper as following. In Section~2, we first prove the existence of the curved front, that is, Theorem~\ref{th1}. Then, we give some examples showing that Theorem~\ref{th1} is not empty. We also show a necessary condition for the existence of the curved front in this section. Section~3 is devoted to the proof of the uniqueness and stability of the curved front in Theorem~\ref{th1}. In Section~4, we construct a curved front with varying interfaces and give an example.


\section{Existence of curved fronts}
This section is devoted to the construction of a curved front satisfying Theorem~\ref{th1}. We will need some properties of the pulsating front, especially the differentiability of the profile $U_e$ and the speed $c_e$ with respect to the direction $e$.

\subsection{Preliminaries}
We will use the hyperbolic function $\text{sech}(x)$ frequently in the sequel. Thus, we recall some known properties of it which can be checked easily.

\begin{lemma}\label{lemma-sech}
It holds that
$$|\text{\rm sech}'(x)|,\ |\text{\rm sech}''(x)|\le \text{\rm sech}(x), \hbox{ for $x\in\R$},$$
and there is a positive constant $p$ such that
$$\text{\rm sech}'(x)>0 \hbox{ for $x\le -p$},\ \text{\rm sech}'(x)<0 \hbox{ for $x\ge p$ and } \text{\rm sech}''(x)>0 \hbox{ for $|x|\ge p$}.$$
\end{lemma}

Then, we need a smooth V-shaped curve with $y=-x\cot\alpha$ and $y=-x\cot\beta$ being its asymptotic lines.

\begin{lemma}\label{lemma-psi}
For any $0<\alpha<\beta<\pi$, there is a smooth function $\psi(x)$ for $x\in\R$ with $y=-x\cot\alpha$ and $y=-x\cot\beta$ being its asymptotic lines and there are positive constants $k_1$, $k_2$ and $K_1$ such that
\begin{eqnarray}\label{K_1}
\left\{\begin{array}{lll}
\psi''(x)>0, &&\hbox{ for all $x\in\R$}\\
-\cot\alpha<\psi'(x)<-\cot\beta, &&\hbox{ for all $x\in\R$}\\
k_1\text{\rm sech}(x)\le \psi'(x)+\cot\alpha\le K_1 \text{\rm sech}(x), &&\hbox{ for $x<0$},\\
k_2\text{\rm sech}(x)\le -\cot\beta-\psi'(x)\le K_1 \text{\rm sech}(x), &&\hbox{ for $x\ge 0$},\\
\max(|\psi''(x)|,|\psi'''(x)|)\le K_1 \text{\rm sech}(x), &&\hbox{ for all $x\in\R$}.
\end{array}
\right.
\end{eqnarray}
\end{lemma}

\begin{proof}
Let $0<\alpha<\beta<\pi$. Since $\alpha<\beta$, there are two positive constants $a$, $b$ and a smooth function $\varphi(x)$ such that
\begin{eqnarray*}
\varphi(x)=\left\{\begin{array}{lll}
-x\cot \alpha, && x\le -a\\
-x\cot \beta, && x\ge b.
\end{array}
\right.
\hbox{ and } \varphi''(x)>0 \hbox{ for $-a<x<b$}.
\end{eqnarray*}
An example of such a function is that one can take an incircle of the straight lines $y=-x\cot \alpha$ and $y=-x\cot \beta$ with tangent points $(-a,a\cot\alpha)$ and $(b,-b\cot\beta)$ and $\varphi(x)$ is made of the line $y=-x\cot \alpha$ for $x\le -a$, the arc of the incircle between $-a$ and $b$, and the line $y=-x\cot \beta$ for $x\ge b$. One can mollify $\varphi(x)$ at $(-a,a\cot\alpha)$ and $(b,-b\cot\beta)$ such that $\varphi(x)\in C^{\infty}(\R)$, see Figure~3.
\begin{figure}\label{fig1}\centering
 \includegraphics[scale=0.2]{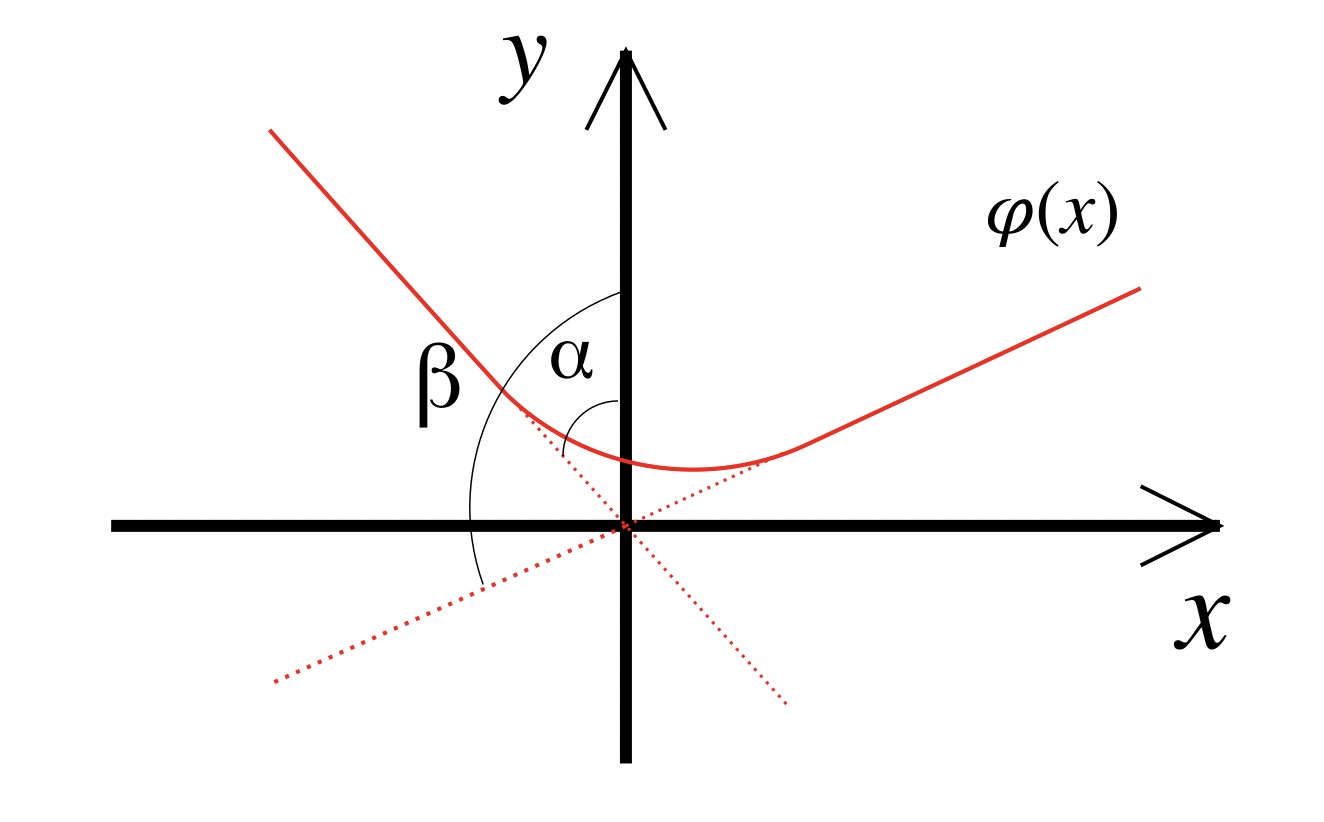}
 \caption{The function $\varphi(x)$.}
\end{figure}
Define a smooth function $\psi(x)$ as following
$$\psi(x):=\varphi(x)+\rho \text{sech}(x),$$
where $\rho>0$ is a constant. Since $\text{sech}''(x)$ is bounded and by Lemma~\ref{lemma-sech}, one can make $\rho$ small enough and $a$, $b$ sufficiently large such that
$$\psi''(x)>0 \hbox{ for all $x\in\R$}.$$
Moreover, one can easily check that $\psi(x)$ satisfies all properties in \eqref{K_1}. This completes the proof.
\end{proof}
\vskip 0.3cm

We now recall some properties of the pulsating front $U_e((x,y)\cdot e-c_e t,x,y)$. One can substitute the form $U_e((x,y)\cdot e-c_e t,x,y)$ into \eqref{eq1.1} and get that $(U_e(\xi,x,y),c_e)$ satisfies the semi-linear elliptic degenerate equation
\begin{align}\label{Ue}
c_e \partial_{\xi} U_e +\partial_{\xi\xi} U_e +2\nabla_{x,y} \partial_{\xi} U_e\cdot e+\Delta_{x,y} U_e +f(x,y,U_e)=0,\,\text{ for all $(\xi,x,y)\in\R\times\mathbb{T}^2$}.
\end{align}

From \cite[Lemma~2.1]{G}, we have the following lemma.

\begin{lemma}\label{ASY}
For any pulsating front $(U_e(\xi,x,y),c_e)$ with $c_e> 0$, there exist $\mu_1>0$, $\mu_2>0$, $C_1>0$ and $C_2>0$ independent of $e$ such that
\begin{align*}
0<U_e(\xi,x,y)\le C_1 e^{-\mu_1\xi}\ \ \text{for}\ \ \xi>0,\ (x,y)\in\mathbb{T}^2\\
0<1-U_e(\xi,x,y)\le C_2 e^{\mu_2\xi}\ \ \text{for}\ \ \xi\le 0,\ (x,y)\in\mathbb{T}^2.
\end{align*}
\end{lemma}

Then, by standard parabolic estimates applied to $u(t,x,y)=U_e((x,y)\cdot e-c_e t,x,y)$, one can get that $|\nabla
_{x,y}u_t|,\ |u_{tt}|,\ |u_t|\le C u(t+1,x,y)$ for some constant $C>0$ and $(t,x,y)\in\R\times\R^2$. Notice that $u_t(t,x,y)=-c_e \partial_{\xi} U_e((x,y)\cdot e-c_e t,x,y)$ with $c_e>0$. Then, by Lemma~\ref{ASY}, we have the following lemma.

\begin{lemma}\label{ASY2}
For any pulsating front $(U_e(\xi,x,y),c_e)$ with $c_e> 0$, there exist $\mu_3>0$ and $C_3>0$  independent of $e$ such that
\begin{align*}
|\partial_{\xi} U_e(\xi,x,y)|,\ |\partial_{\xi\xi} U_e(\xi,x,y)|,\ |\nabla_{x,y}\partial_{\xi} U_e(\xi,x,y)|\le C_3 e^{-\mu_3 |\xi|}\ \ \text{for}\ \ \xi\in\R,\ (x,y)\in\mathbb{T}^2.
\end{align*}
\end{lemma}

We also need the following properties.

\begin{lemma}\label{lemma-ut'}
For any $C>0$, there is $0<\delta<1/2$ independent of $e$ such that 
\be\label{leule}
\delta\le U_e(\xi,x,y)\le 1-\delta, \hbox{ for $-C\le \xi\le C$ and $(x,y)
\in\mathbb{T}^2$},
\ee
and there is $r>0$ independent of $e$ such that 
\be\label{Uer}
-\partial_{\xi}U_e(\xi,x,y)\ge r \hbox{ for for $-C\le \xi\le C$ and $(x,y)
\in\mathbb{T}^2$}.
\ee
\end{lemma}

\begin{proof}
Let $u(t,x,y)=U_e((x,y)\cdot e-c_e t,x,y)$. One can easily check that $u(t,x,y)$ is a transition front connecting $0$ and $1$ with set $\{(t,x,y)\in\R\times\R^2; (x,y)\cdot e-c_e t=0\}$ being its interfaces. Then, by \cite[Theorem~1.2]{BH2}, one immediately has that there is $0<\delta<1/2$ such that 
$$\delta\le u(t,x,y)\le 1-\delta, \hbox{ for $-C\le (x,y)\cdot e-c_e t\le C$.}$$
By continuity of $U_e$ with respect to $e$, one has that $\delta$ can be independent of $e$.

The following proof for \eqref{Uer} can be simplified for the pulsating front $U_e$. However, we do it in a general way in purpose that such idea can be used to prove that the curved front which we construct later has similar properties. Notice that $u_t(t,x,y)>0$ satisfies
$$
(u_t)_t-\Delta u_t -f_u(x,y,u)u_t=0, \hbox{ for $(t,x,y)\in\R\times\R^2$}.
$$
Assume that there is a sequence $\{(t_n,x_n,y_n)\}_{n\in\mathbb{N}}$ of $\R\times\R^2$ such that $-C\le (x_n,y_ n)\cdot e-c_e t_n\le C$ and $u_t(t_n,x_n,y_n)\rightarrow 0$ as $n\rightarrow +\infty$. Since $f(x,y,u)$ is periodic in $(x,y)$, there is $(x',y')\in \R^2$ such that $f(x+x_n,y+y_n,u)\rightarrow f(x+x',y+y',u)$ as $n\rightarrow +\infty$.
Let $u_n(t,x,y)=u(t+t_n,x+x_n,y+y_n)$ and $v_n(t,x,y)=u_t(t+t_n,x+x_n,y+y_n)$. By standard parabolic estimates, $u_n(t,x,y)$ converges to a solution $u_{\infty}(t,x,y)$ of 
$$u_t-\Delta u -f(x+x',y+y',u_{\infty})=0, \hbox{ for $(t,x,y)\in\R\times\R^2$},$$
and $v_n(t,x,y)$ converges to a solution $v_{\infty}(t,x,y)$ of
$$
v_t-\Delta v -f_u(x+x',y+y',u_{\infty})v=0, \hbox{ for $(t,x,y)\in\R\times\R^2$}.
$$
Moreover, $v_{\infty}(t,x,y)$ satisfies $v_{\infty}(t,x,y)\ge 0$ and $v_{\infty}(0,0,0)=0$. By the maximum principle, $v_{\infty}(t,x,y)\equiv 0$. Since $U_e(\xi,x,y)\rightarrow 1$ as $\xi\rightarrow -\infty$, there is $R>0$ large enough such that 
$$u(t,x,y)\ge 1-\sigma \hbox{ for $(t,x,y)\in\R\times\R^2$ such that $(x,y)\cdot e -c_e t\le -R$}$$
where $\sigma$ is defined in (F3). Take $(x_*,y_*)\in \R^2$ such that $(x_*,y_*)\cdot e<-R-C$.  Then, $v_{\infty}(t,x,y)\equiv 0$ implies that $u_t(t+t_n,x+x_*+x_n,y+y_*+y_n)\rightarrow 0$ as $n\rightarrow +\infty$ locally uniformly in $\R\times\R^2$. Notice that $(x_*+x_n,y_*+y_n)\cdot e -c_e t_n\le -R$ and hence, $u(t_n,x_*+x_n,y_*+y_n)\ge 1-\sigma$. Also notice that $1$ is the only equilibrium of \eqref{eq1.1} over $1-\sigma$ from (F3) and \eqref{lambda}. It further implies that $u(t+t_n,x+x_*+x_n,y+y_*+y_n)\rightarrow 1$ locally uniformly in $\R\times\R^2$. Since $(x_*,y_*)$ is fixed and $-C\le (x_n,y_n)\cdot e -c_e t_n\le C$, it reaches a contradiction with \eqref{leule}. This completes the proof. 
\end{proof}
\vskip 0.3cm

It follows from \cite[Theorem~1.5]{G} that $U_e$ and $c_e$ are differentiable with respect to $e$. Remember that $U_e$ are normalized by $U_e(0,0,0)=1/2$ for all $e\in\mathbb{S}$. For any $b\in \R^2\setminus\{0\}$, define
\be\label{Ubcb}
U_b=U_{\frac{b}{|b|}} \hbox{ and } c_b=c_{\frac{b}{|b|}}.
\ee

\begin{lemma}
Let $U_b$ and $c_b$ be defined in \eqref{Ubcb}. Then, $U_b$ and $c_b$ are doubly continuously Fr\'{e}chet differentiable at any $b\in \R^N\setminus\{0\}$.
\end{lemma}

Let us denote the Fr\'{e}chet derivatives up to second order of $U_e$ and $c_e$ with respect to $e$ by $U'_e$, $U''_e$, $c'_e$ and $c''_e$. From \cite{G}, one knows that $U_e$ is continuous with respect to $e$ in $L^{\infty}$ space. Since $U_e$ is uniformly bounded for $e\in\mathbb{S}$, the Fr\'{e}chet derivatives are all bounded in the sense that
$$\|U'_e\|=\sup_{h\in\R^N} \frac{\|U'_e\cdot h\|_{L^{\infty}(\R\times\mathbb{T}^N)}}{|h|}<+\infty,\ \|U''_e\|=\sup_{(h,\rho)\in\R^N\times\R^N} \frac{\|(U''_e \cdot h) \cdot \rho\|_{L^{\infty}(\R\times\mathbb{T}^N)}}{|h||\rho|}<+\infty,$$
and
$$\|c'_e\|=\sup_{h\in\R^N} \frac{|c'_e\cdot h|}{|h|}<+\infty,\ \|c''_e\|=\sup_{(h,\rho)\in\R^N\times\R^N} \frac{|(c''_e\cdot h)\cdot \rho|}{|h||\rho|}<+\infty.$$
We also know from \cite{G} that for any $h\in\R^2$, $\rho\in\R^2$, $U'_e\cdot h$ and $(U''_e\cdot h)\cdot \rho$ are differentiable with respect to $\xi$, $x$ and $y$ up to second order and these derivatives are bounded too. We then need the following properties of $U_e'$.

\begin{lemma}\label{lemma-Ue'}
For any $e\in\mathbb{S}$, there exist $\mu_4>0$ and $C_4>0$ independent of $e$ such that
$$|(U'_e\cdot h) (\xi,x,y)|,\ |(\partial_{\xi}U'_e\cdot h) (\xi,x,y)|\le C_4 e^{-\mu_4|\xi|} |h|, \hbox{ for any $h\in\R^2$, $\xi\in\R$ and $(x,y)\in\mathbb{T}^2$}.$$
\end{lemma}

\begin{proof}
Take a smooth nonincreasing function $p(\xi)$ such that
$$p(\xi)=1 \hbox{ for $\xi\le 0$ and } p(\xi)=e^{-r \xi} \hbox{ for $\xi\ge b$},$$
for some positive constants $r$ and $b$. Here, one can make $r$ and $b$ to be small and large enough respectively such that 
\be\label{r}
r\le \min\{\mu_1,\mu_2\},
\ee
and
\be\label{eq-p'}
c_e\left|\frac{p'(\xi)}{p(\xi)}\right|+\left|\frac{p''(\xi)}{p(\xi)}\right|\le \frac{\lambda}{2} \hbox{ for all $\xi\in\R$ and $e\in\mathbb{S}$},
\ee
where $\lambda>0$ is defined in (F3).

For every direction $e$, we define a function $V_e(\xi,x,y)$ by
$$V_e(\xi,x,y):=p^{-1}(\xi)U_e(\xi,x,y), \hbox{ for $\xi\in\R$ and $(x,y)\in\mathbb{T}^2$}.$$
By Lemma~\ref{ASY} and \eqref{r}, one has 
$$V_e(-\infty,x,y)=1 \hbox{ and } V_e(+\infty,x,y)=0, \hbox{ uniformly for $(x,y)\in\mathbb{T}^{2}$ and $e\in\mathbb{S}$},$$
and $V(\xi,x,y)\in L^2(\R\times\mathbb{T}^2)$. Since $U_e(\xi,x,y)$ satisfies \eqref{Ue}, one can get that $V_e(\xi,x,y)$ satisfies
\begin{align*}
c_e\partial_{\xi}V+\partial_{\xi\xi}V_e +2\nabla_{x,y}\partial_{\xi}V_e\cdot e &+\Delta_{x,y} V_e +\frac{2p'}{p} \partial_{\xi} V_e + \frac{2p'}{p} \nabla_{x,y} V_e\cdot e\\
&+\frac{1}{p} f(x,y,p V_e) + \Big(c_e\frac{p'}{p}+\frac{p''}{p}\Big)V_e=0, \hbox{ for $(\xi,x,y)\in\R\times\mathbb{T}^2$}. 
\end{align*}
From (F3) and \eqref{eq-p'}, there is $C>0$ such that
\begin{eqnarray}\label{fv}
\left\{\begin{array}{lll}
\frac{1}{p} f(x,y,p V_e) + \Big(c_e\frac{p'}{p}+\frac{p''}{p}\Big)V_e\le -\frac{\lambda}{2} V_e, &&\hbox{ for $(x,y)\in\mathbb{T}^2$ and $\xi\ge C$},\\
\frac{1}{p} f(x,y,p V_e) + \Big(c_e\frac{p'}{p}+\frac{p''}{p}\Big)V_e\ge \frac{\lambda}{2} (1-V_e), &&\hbox{ for $(x,y)\in\mathbb{T}^2$ and $\xi\le -C$},
\end{array}
\right.
\end{eqnarray}

For any $e\in\mathbb{S}$, define a linear operator
$$M_e(v):=c_e\partial_{\xi} v+\partial_{\xi\xi} v +2\nabla_{x,y} \partial_{\xi} v\cdot e +\Delta_{x,y} v +\frac{2p'}{p} \partial_{\xi}v + \frac{2p'}{p} \nabla_{x,y} v\cdot e-\beta v,$$
where $\beta>0$ is a fixed real number and
$$v\in D:=\{v\in H^1(\R\times\T^N);\ \partial_{\xi\xi}v+2\nabla_y\partial_{\xi}v\cdot e+\Delta_y v                                                                          \in L^2(\R\times\T^N)\},$$
see \cite{G} for definitions of $L^2(\R\times\mathbb{T}^2)$, $H^1(\R\times\mathbb{T}^2)$ and their norms. The space $D$ is endowed with the norm $\|v\|_{D}=\|v\|_{H^1(\R\times\mathbb{T}^N)}+\|\partial_{\xi\xi}v+2\nabla_y\partial_{\xi}v\cdot e+\Delta_y v\|_{L^2(\R\times\T^N)}$. Then, by the similar proofs of Lemma 3.1, Lemma 3.2 and Lemma 3.3 in \cite{DHZ} and Lemma~2.7 in \cite{G}, one knows that $M_e$ satisfies all the properties in Lemma~2.7 of \cite{G}, such as invertibility and boundedness. For any $e\in\mathbb{S}$, we then define
\begin{align*}
H_e(v):=c_e\partial_{\xi} v+\partial_{\xi\xi} v +2\nabla_{x,y} \partial_{\xi} v\cdot e +\Delta_{x,y} v &+\frac{2p'}{p} \partial_{\xi}v + \frac{2p'}{p} \nabla_{x,y} v\cdot e\\
&+f_u(y,pV_e)v+\Big(c_e\frac{p'}{p}+\frac{p''}{p}\Big) v,\quad v\in D.
\end{align*}
Notice that $H_e(v)=\widetilde{H}_e(pv)/p$ with $0<p(\xi)\le 1$ where
\begin{align*}
\widetilde{H}_e(v):=c_e\partial_{\xi} v+\partial_{\xi\xi} v +2\nabla_y \partial_{\xi} v\cdot e +\Delta_y v +f_u(y,U_e)v,\quad v\in D.
\end{align*}
By Lemma~2.9 in \cite{G}, one knows that the operator $\widetilde{H}_e$ and its adjoint operator $\widetilde{H}^*_e$ have algebraically simple eigenvalue $0$ and the kernel of $\widetilde{H}_e$ is generated by $\partial_{\xi} U_e$. Therefore, the operator $H_e$ and its adjoint operator $H^*_e$ also have algebraically simple eigenvalue $0$ and the kernel of $H_e$ is generated by $p^{-1}\partial_{\xi} U_e$. Moreover, the property that the range of $H_e$ is closed in $L^2(\R)\times\mathbb{T}^2$ can be proved in the same line of the proof of \cite[Lemma~4.1]{DHZ} by using \eqref{fv}. 

Now, for any $e\in\mathbb{S}$, $v\in H^2(\R\times\mathbb{T}^2)$, $\vartheta\in\R$ and $\eta\in\R^2$, define
\begin{align*}
K_e(v,\vartheta,\eta)=\vartheta\partial_{\xi} (V_e+v)&+2\nabla_y\partial_{\xi} (V_e+v)\cdot \eta +\frac{2p'}{p} \nabla_{x,y}(V_e+v)\cdot \eta\\
&+\frac{1}{p} f(y,p(V_e+v))-\frac{1}{p} f(y,V_e)+\Big(c_e\frac{p'}{p}+\frac{p''}{p}+\beta\Big) v,
\end{align*}
and
$$
G_e(v,\vartheta,\eta):=\left(v+M_e^{-1}(K_e(v,\vartheta,\eta)),\int_{\R^+\times\T^N} \left[(V_e(\xi,y)+v(\xi,y))^2-U_e^2(\xi,y)\right]dyd\xi\right).
$$
By following the proof of \cite[Lemma~2.10]{G}, one can get that for every $e\in\mathbb{S}$, the function $G_e:\ H^2(\R\times\T^N)\times\R\times\R^N\rightarrow D\times\R$ is continuous and it is continuously Fr\'{e}chet differentiable with respect to $(v,\vartheta)$ and doubly continuously Fr\'{e}chet differentiable with respect to $\eta$. For any $e\in\mathbb{S}^{N-1}$ and $(\tilde{v},\tilde{\vartheta})\in D\times\R$, define
\begin{align*}
Q_e(\tilde{v},\tilde{\vartheta})=\left(\tilde{v}+M_e^{-1}(\tilde{\vartheta}\partial_{\xi}V_e +f_u(y,U_e)\tilde{v} +\Big(c_e\frac{p'}{p}+\frac{p''}{p}+\beta\Big)\tilde{v}),
2\int_{\R^+\times\T^N} V_e(\xi,y)\tilde{v}(\xi,y)dyd\xi\right),
\end{align*}
which has the same form as $\partial_{(v,\vartheta)} G_e(0,0,0)$. By the properties of $H_e$ and the same line of the proofs of \cite[Lemma~3.3]{DHZ} and \cite[Lemma~2.11]{G}, one can get that $Q_e$ satisfies all properties in \cite[Lemma~2.11]{G}, such as invertibility and boundedness. 

As soon as we have all these properties of these operators, we can follow the same proof of \cite[Theorem~1.5]{G} to get that $V_b(\xi,x,y)=p^{-1}(\xi) U_b(\xi,x,y)$ is doubly Fr\'{e}chet differentiable at any $b\in\R^2\setminus\{0\}$. Moreover, $\|V_e'\|$ is bounded for any $e\in\mathbb{S}$.

Thus, by the definition of Fr\'{e}chet differentiation, we have
$$(U'_e\cdot h)(\cdot,\cdot,\cdot)=p(\xi) (V'_e\cdot h)(\cdot,\cdot,\cdot), \hbox{ for any $e\in\mathbb{S}$ and $h\in\R^2$}.$$
Therefore, there exist a positive constant $C_4$ such that
\be\label{U'e-asy}
|(U'_e\cdot h)(\xi,x,y)|\le p(\xi)\|V'_e\| |h|\le C_4 e^{-r\xi} |h| \hbox{ for $\xi\ge 0$, $(x,y)\in\mathbb{T}^2$ and $h\in\R^2$}.
\ee
By applying similar arguments to the other side, that is, $\xi<0$, one can also get that there are positive constants $C_5$ and $\mu_5$ such that
\be\label{U'e-asy2}
|(U'_e\cdot h)(\xi,x,y)|\le C_5 e^{\mu_5\xi} |h| \hbox{ for $\xi< 0$, $(x,y)\in\mathbb{T}^2$ and $h\in\R^2$}.
\ee

Lastly, we differentiate \eqref{Ue} at $e$ on $h\in\R^2$ and get that
\begin{align*}
&(c'_e\cdot h)\partial_{\xi} U_e +c_e\partial_{\xi} (U'_e\cdot h) +\partial_{\xi\xi} (U'_e\cdot h) +2\nabla_y\partial_{\xi} U_e\cdot (h-(e\cdot h)e)\nonumber\\
&+2\nabla_{x,y}\partial_{\xi} (U'_e\cdot h)\cdot e+\Delta_{x,y} (U'_e\cdot h) +f_u(x,y,U_e)(U'_e\cdot h)=0.
\end{align*}
By changing variables $\xi=(x,y)\cdot e-c_e t$, one has that $u(t,x):=(U'_e\cdot h)((x,y)\cdot e-c_e t,x,y)$ satisfies a parabolic equation
$$u_t-\Delta u=f_u(x,y, U_e)u+(c'_e\cdot h)\partial_{\xi} U_e +2\nabla_{x,y}\partial_{xi} U_e \cdot (h-(e\cdot h)e).$$
By parabolic estimates, Lemma~\ref{ASY2} and \eqref{U'e-asy}-\eqref{U'e-asy2}, one can get that there are positive constants $C_6$ and $\mu_6$ such that 
$$|u_t(t,x,y)|\le C_6 e^{-\mu_6 |(x,y)\cdot e-c_e t|}|h|,$$
that is,
$$|(\partial_{\xi}U'_e\cdot h) (\xi,x,y)|\le C_6 e^{-\mu_6|\xi|}|h| \hbox{ for any $h\in\R^2$, $\xi\in\R$ and $(x,y)\in\mathbb{T}^2$}.$$
This completes the proof.
\end{proof}

\subsection{Proof of Theorem~\ref{th1}}

Take any two angles $\alpha$, $\beta$ of $(0,\pi)$ such that \eqref{angle} holds. Let $\psi(x)$ be a smooth function satisfying Lemma~\ref{lemma-psi} for $\alpha$ and $\beta$. Take a constant $\varrho$ to be determined later.  For every point $(x,y)$ on the curve $y=\psi(\varrho x)/\varrho$, there is a unit normal
\be\label{e(x)}
e(x)=(e_1(x),e_2(x))=\Big(-\frac{\psi'(\varrho x)}{\sqrt{\psi'^2(\varrho x)+1}},\frac{1}{\sqrt{\psi'^2(\varrho x)+1}}\Big).
\ee
By Lemma~\ref{lemma-psi}, every component of $e(x)$ is differentiable with respect to $x$ and
$$e(x)\rightarrow (\cos\alpha,\sin\alpha) \hbox{ as $x\rightarrow -\infty$ and } e(x)\rightarrow (\cos\beta,\sin\beta) \hbox{ as $x\rightarrow +\infty$}.$$
Its derivatives can be denoted by 
$$e'(x)=(e_1'(x),e_2'(x))=\Big(-\frac{\varrho\psi''(\varrho x)}{(\psi'^2(\varrho x)+1)^{\frac{3}{2}}},-\frac{\varrho\psi'(\varrho x)\psi''(\varrho x)}{(\psi'^2(\varrho x)+1)^{\frac{3}{2}}}\Big),$$
and
\begin{align*}
e''(x)=(e_1''(x),e_2''(x))=&\Big(-\frac{\varrho^2\psi'''(\varrho x)}{(\psi'^2(\varrho x)+1)^{\frac{3}{2}}}+\frac{3\varrho^2\psi'(\varrho x)\psi''^2(\varrho x)}{(\psi'^2(\varrho x)+1)^{\frac{5}{2}}},\\
&-\frac{\varrho^2\psi''^2(\varrho x)}{(\psi'^2(\varrho x)+1)^{\frac{3}{2}}}-\frac{\varrho^2\psi'(\varrho x)\psi'''(\varrho x)}{(\psi'^2(\varrho x)+1)^{\frac{3}{2}}}+\frac{3\varrho^2\psi'^2(\varrho x)\psi''^2(\varrho x)}{(\psi'^2(\varrho x)+1)^{\frac{5}{2}}}\Big).
\end{align*}
Therefore, by Lemma~\ref{lemma-psi}, there exist $K_2>0$ and $K_3>0$ such that
\be\label{K2K3}
|e'(x)|\le \varrho K_2 \text{sech}(\gamma x) \hbox{ and } |e''(x)|\le \varrho^2 K_3 \text{sech}(\gamma x) \hbox{ for all $x\in\R$}.
\ee

Remember that $U^-_{\alpha\beta}(t,x,y)$ defined by \eqref{U-} is a subsolution of \eqref{eq1.1}. Now, take a positive constant $\varepsilon$ and we define
$$U^+(t,x,y)=U_{e(x)}(\xi(t,x,y),x,y)+\varepsilon \text{sech}(\varrho x),$$
where  
\be\label{xi}
\xi(t,x,y)=\frac{y-c_{\alpha\beta}t-\psi(\varrho x)/\varrho}{\sqrt{\psi'^2(\varrho x)+1}},
\ee 
and $c_{\alpha\beta}$ is defined by \eqref{angle}.
We prove that $U^+(t,x,y)$ is a supersolution of \eqref{eq1.1} for small $\varepsilon$ and $\varrho$.

\begin{lemma}\label{U+}
There exist $\varepsilon_0>0$ and $\varrho(\varepsilon_0)>0$ such that for any $0<\varepsilon\le \varepsilon_0$ and $0<\varrho\le \varrho(\varepsilon_0)$, the function $U^+(t,x,y)$ is a supersolution of \eqref{eq1.1} with $U^+_t>0$. Moreover, it satisfies
\be\label{U+cU-}
\lim_{R\rightarrow +\infty} \sup_{x^2+(y-c_{\alpha,\beta} t)^2>R^2} \Big| U^+(t,x,y)-U^-_{\alpha\beta}(t,x,y)\Big|\le 2\varepsilon,
\ee
and
\be\label{U+geU-}
U^+(t,x,y)\ge U^-_{\alpha\beta}(t,x,y), \hbox{ for all $t\in\R$ and $(x,y)\in\R^2$}.
\ee
\end{lemma}

\begin{proof}
We divide the proof into three steps. 

{\it Step 1: $U^+$ is a supersolution.}  We will pick $\varepsilon_0>0$ and $\varrho(\varepsilon)$ such that Lemma~\ref{U+} holds. Assume that
$$\varepsilon_0\le \frac{\sigma}{2},$$
where $\sigma>0$ is defined in (F3).
More restrictions on $\varepsilon_0$ will be given later. One can compute that
\begin{align*}
LU^+:=&U^+_t-\Delta_{x,y} U^+-f(x,y,U^+)\\
=& \partial_{\xi} U_{e(x)} \xi_t -\partial_{\xi\xi} U_{e(x)} (\xi_x^2+\xi_y^2) -2\nabla_{x,y}\partial_{\xi} U_{e(x)}\cdot (\xi_x,\xi_y) -\Delta_{x,y} U_{e(x)}-\partial_{\xi} U_{e(x)} \xi_{xx} \\
&-U''_{e(x)} \cdot e'(x)\cdot e'(x)-U'_{e(x)}\cdot e''(x)-2\partial_{\xi}U'_{e(x)}\cdot e'(x) \xi_x -2\partial_xU'_{e(x)}\cdot e'(x) \\
&-\varepsilon \varrho^2 \text{sech}''(\varrho x)-f(x,y,U^+),
\end{align*}
where $\partial_{\xi}U_{e(x)}$, $\partial_{\xi\xi} U_{e(x)}$, $\nabla_{x,y}\partial_{\xi} U_{e(x)}$, $\Delta_{x,y}U_{e(x)}$, $U''_{e(x)}\cdot e'(x)\cdot e'(x)$, $U'_{e(x)}\cdot e''(x)$, $\partial_{\xi}U'(e(x))\cdot e'(x)$, $\partial_x U'_{e(x)}\cdot e'(x)$ are taking values at $(\xi(t,x,y),x,y)$ and $U^+$, $\xi_t$, $\xi_x$, $\xi_y$ are taking values at $(t,x,y)$.
By \eqref{xi}, it follows from a direct computation that
\begin{eqnarray}\label{eq-xi}
\begin{aligned}
&\xi_t=-\frac{c_{\alpha\beta} }{\sqrt{\psi'^2(\varrho x)+1}},\\
&\xi_x=-\frac{\varrho\psi'(\varrho x)\psi''(\varrho x)}{(\psi'^2(\varrho x)+1)^{\frac{1}{2}}}\xi-\frac{\psi'(\varrho x)}{\sqrt{\psi'^2(\varrho x)+1}},\\
&\xi_y=\frac{1}{\sqrt{\psi'^2(\varrho x)+1}},\\
&\xi_{xx}=-\frac{\varrho^2\psi''^2(\varrho x)+\varrho^2\psi'(\varrho x)\psi''(\varrho x)}{\psi'^2(\varrho x)+1}\xi+\frac{3\varrho^2\psi'^2(\varrho x)\psi''^2(\varrho x)}{(\psi'^2(\varrho x)+1)^2}\xi +\frac{\varrho(\psi'^2(\varrho x)-1)\psi''(\varrho x)}{(\psi'^2(\varrho x)+1)^{\frac{3}{2}}},\\
&\xi^2_x+\xi^2_y-1=\left(\Big(\frac{\varrho\psi'(\varrho x)\psi''(\varrho x)}{(\psi'^2(\varrho x)+1)^{\frac{3}{2}}}\Big)^2\xi^2+2\Big(\frac{\varrho\psi'(\varrho x)\psi''(\varrho x)}{(\psi'^2(\varrho x)+1)^{\frac{3}{2}}}\Big)\frac{\psi'(\varrho x)}{\sqrt{\psi'^2(\varrho x)+1}}\xi\right).
\end{aligned}
\end{eqnarray}
By noticing that $\xi_y=e_2(x)$ and by \eqref{Ue}, one has
\begin{equation}\label{LU+}
\begin{aligned}
LU^+
=& (c_{e(x)}-\xi_t)\partial_{\xi} U_{e(x)}  -\partial_{\xi\xi} U_{e(x)} (\xi^2_x+\xi^2_y-1) -2\partial_x\partial_{\xi} U_{e(x)}(\xi_x-e_1(x))-\partial_{\xi} U_{e(x)} \xi_{xx} \\
&-U''_{e(x)} \cdot e'(x)\cdot e'(x) -U'_{e(x)}\cdot e''(x)-2\partial_{\xi}U'_{e(x)}\cdot e'(x) \xi_x -2\partial_xU'_{e(x)}\cdot e'(x) \\
&-\varepsilon \varrho^2 \text{sech}''(\varrho x)+f(x,y,U_{e(x)})-f(x,y,U^+),
\end{aligned}
\end{equation}
where $\partial_{\xi}U_{e(x)}$, $\partial_{\xi\xi} U_{e(x)}$, $\partial_{x}\partial_{\xi} U_{e(x)}$, $U''_{e(x)}\cdot e'(x)\cdot e'(x)$, $U'_{e(x)}\cdot e''(x)$, $\partial_{\xi}U'(e(x))\cdot e'(x)$, $\partial_x U'_{e(x)}\cdot e'(x)$, $U_{e(x)}$ are taking values at $(\xi(t,x,y),x,y)$ and $U^+$, $\xi_t$, $\xi_x$, $\xi_y$ are taking values at $(t,x,y)$. By Lemma~\ref{ASY2}, one has that $|\partial_{\xi\xi}U_{e(x)}\xi^2|$, $|\partial_{\xi\xi} U_{e(x)} \xi|$, $|\partial_{x}\partial_{\xi} U_{e(x)} \xi|$ and $|\partial_{\xi} U_{e(x)} \xi|$ are uniformly bounded for $\xi\in\R$, $(x,y)\in\R^2$. Then, by 
 Lemmas~\ref{lemma-psi} and \eqref{eq-xi}, there is $C_5>0$ such that
\be\label{C5}
|\partial_{\xi\xi} U_{e(x)} (\xi^2_x+\xi^2_y-1)| +2|\partial_x\partial_{\xi} U_{e(x)}(\xi_x-e_1(x))|+|\partial_{\xi} U_{e(x)} \xi_{xx}|\le C_5 \varrho\text{sech}(\varrho x).
\ee
Since $\|U'_{e}\|$, $\|U''_{e}\|$, $\|\partial_{\xi}U'_{e}\|$, $\|\partial_xU'_e\|$ are bounded and by Lemma~\ref{lemma-Ue'}, \eqref{K2K3}, there is $C_6>0$ such that
\be\label{C6}
|U''_{e(x)} \cdot e'(x)\cdot e'(x)|
+|U'_{e(x)}\cdot e''(x)|
+2|\partial_{\xi}U'_{e(x)}\cdot e'(x) \xi_x| +2|\partial_xU'_{e(x)}\cdot e'(x)|\le C_6\varrho\text{sech}(\varrho x).
\ee

We claim that
\begin{claim}\label{claim1}
There is $C_7>0$ such that 
\be\label{ineq-c}
\xi_t-c_{e(x)}=\frac{c_{\alpha\beta}}{\sqrt{\psi'^2(\varrho x)+1}}-c_{e(x)}\ge C_7\text{sech}(\varrho x)>0.
\ee
\end{claim}
We postpone the proof of this claim after the proof of this lemma.

Then, it follows from \eqref{LU+}, \eqref{C5}, \eqref{C6}, \eqref{ineq-c}, Lemma~\ref{lemma-sech} and $\partial_{\xi} U_e<0$ that
\be\label{LU+2}
\begin{aligned}
LU^+\ge & -\partial_{\xi} U_{e(x)} C_7\text{sech}(\varrho x) -(C_5+C_6)\varrho\text{sech}(\varrho x) -2\varepsilon \varrho^2 \text{sech}(\varrho x)\\
&+f(x,y,U_{e(x)})-f(x,y,U^+).
\end{aligned}
\ee
By Lemma~\ref{ASY}, there is $C>0$ such that 
\be\label{eq-C}
0<U_e(\xi,x,y)\le \frac{\sigma}{2} \hbox{ for $\xi\ge C$ and } 0<1-U_e(\xi,x,y)\le \frac{\sigma}{2} \hbox{ for $\xi\le -C$}.
\ee
uniformly for $(x,y)\in\mathbb{T}^2$ and $e\in\mathbb{S}$. Then, for $(t,x,y)\in\R\times\R^2$ such that $\xi(t,x,y)\ge C$ and $\xi(t,x,y)\le -C$ respectively, one has that $U^+(t,x,y)\le \sigma/2+\varepsilon\le \sigma$ and $U^+(t,x,y)\ge 1-\sigma/2$ respectively since $\varepsilon\le \varepsilon_0\le \sigma/2$ and hence, it follows from \eqref{lambda} that
\be\label{f1}
f(x,y,U_{e(x)})-f(x,y,U^+)\ge \lambda \varepsilon \text{sech}(\varrho x).
\ee
Since $\partial_{\xi} U_e<0$ and by \eqref{LU+2}, \eqref{f1}, one has that 
$$LU^+\ge \Big(-(C_5+C_6)\varrho-2\varepsilon \varrho^2+\lambda\varepsilon\Big)\text{sech}(\varrho x)\ge0,$$
by taking $0<\varrho\le \varrho(\varepsilon)$ where $\varrho(\varepsilon)>0$ is small enough such that
\be\label{de}
-(C_5+C_6)\varrho-2\varepsilon \varrho^2+\lambda\varepsilon>0, \hbox{ for all $0<\varrho\le \varrho(\varepsilon)$}.
\ee
Finally, for $(t,x,y)\in\R\times\R^2$ such that $-C\le \xi(t,x,y)\le C$, it follows from Lemma~\ref{lemma-ut'} that there is $k>0$ such that 
\be\label{eq-k}
-\partial_{\xi} U_e(\xi,x,y)\ge k \hbox{ for all $e\in\mathbb{S}$}.
\ee
Notice that 
\be\label{eq-M}
f(x,y,U_{e(x)})-f(x,y,U^+)\ge -M \varepsilon \text{sech}(\varrho x),
\ee
where $M:=\max_{(x,y,u)\in\mathbb{T}^2\times\R} |f_u(x,y,u)|$. Thus, it follows from \eqref{LU+2}, \eqref{de}, \eqref{eq-k} and \eqref{eq-M} that
$$LU^+\ge \Big(kC_7-(C_5+C_6)\varrho-2\varepsilon \varrho^2-M\varepsilon\Big)\text{sech}(\varrho x)\ge \Big(kC_7-(\lambda+M)\varepsilon\Big)\text{sech}(\varrho x)\ge 0,$$
by taking $\varepsilon_0=\min\{\sigma/2, kC_7/(\lambda+M)\}$ and $0<\varepsilon\le \varepsilon_0$.

Therefore, $LU^+\ge 0$ for all $t\in\R$ and $(x,y)\in\R^2$. By the comparison principle, $U^+(t,x,y)$ is a supersolution of \eqref{eq1.1}. The property $U^+_t>0$ comes from $\partial_{\xi}U_e<0$ and $c_{\alpha\beta}>0$.

{\it Step 2: the proof of \eqref{U+cU-}.} Since $e(x)\rightarrow (\cos\alpha,\sin\alpha)$ as $x\rightarrow -\infty$ and by the definition of $U'_{e}$, there is $R_1>0$ such that 
\begin{align}
|U_{e(x)}(\xi(t,x,y),x,y)-U_{\alpha}(\xi(t,x,y),x,y)|\le& \|U'_{\alpha}\| |e(x)-(\cos\alpha,\sin\alpha)|+o(|e(x)-(\cos\alpha,\sin\alpha)|)\nonumber\\
\le& \frac{\varepsilon}{2},\quad \hbox{ for $x\le -R_1$ and $t\in\R$, $y\in\R$}.\label{epsilon2-1}
\end{align}
Notice that $1/\sqrt{\psi'^2(\varrho x)+1}\rightarrow \sin\alpha$ as $x\rightarrow -\infty$ and $c_{\alpha\beta}\sin\alpha=c_{\alpha}$. Then, by Lemma~\ref{lemma-psi}, one has that
$$\xi(t,x,y)\rightarrow x\cos\alpha+y\sin\alpha-c_{\alpha}t, \hbox{ as $x\rightarrow -\infty$ for any $t\in\R$ and $y\in\R$}.$$
Thus, there is $R_2>0$ such that
$$|U_{\alpha}(\xi(t,x,y),x,y)-U_{\alpha}(x\cos\alpha+y\sin\alpha-c_{\alpha}t,x,y)|\le \frac{\varepsilon}{2}, \hbox{ for $x\le -R_2$ and $t\in\R$, $y\in\R$}.$$
Together with \eqref{epsilon2-1}, it follows that
\be\label{R2}
|U^+(t,x,y)-U_{\alpha}(x\cos\alpha+y\sin\alpha-c_{\alpha}t,x,y)|\le 2\varepsilon, \hbox{ for $x\le -\max\{R_1,R_2\}$ and $t\in\R$, $y\in\R$}.
\ee
Similarly, one can prove that there is $R_3>0$ such that
\be\label{R3}
|U^+(t,x,y)-U_{\beta}(x\cos\beta+y\sin\beta-c_{\beta}t,x,y)|\le 2\varepsilon, \hbox{ for $x\ge R_3$ and $t\in\R$, $y\in\R$}.
\ee

Now, for $-\max\{R_1,R_2\}\le x\le R_3$, we know that $\psi(\varrho x)$ and $\psi'(\varrho x)$ are bounded. Then, as $y-c_{\alpha\beta} t\rightarrow +\infty$, one has that 
$$\xi(t,x,y)\rightarrow +\infty \hbox{ and } x\cos\alpha+y\sin\alpha-c_{\alpha}t\rightarrow +\infty, \hbox{ for $-\max\{R_1,R_2\}\le x\le R_3$}.$$
Thus, there is $R_4>0$ such that 
$$0<U_{e(x)}(\xi(t,x,y),x,y)\le \frac{\varepsilon}{2},$$ 
and
$$0<U_{\alpha}(x\cos\alpha+y\sin\alpha-c_{\alpha}t,x,y),\ U_{\beta}(x\cos\beta+y\sin\beta-c_{\beta}t,x,y)\le \frac{\varepsilon}{2},$$ 
for $-\max\{R_1,R_2\}\le x\le R_3$ and $y-c_{\alpha\beta}t\ge R_4$. Hence, 
\be\label{R4}
|U^+(t,x,y)- U_{\alpha\beta}(t,x,y)|\le 2\varepsilon,
\ee
for $-\max\{R_1,R_2\}\le x\le R_3$ and $y-c_{\alpha\beta}t\ge R_4$. Similarly, since $U_{e(x)}(-\infty,x,y)=U_{\alpha}(-\infty,x,y)=1$ uniformly for $(x,y)\in\mathbb{T}^2$, there is $R_5$ such that
\be\label{R5}
|U^+(t,x,y)- U_{\alpha\beta}(t,x,y)|\le 2\varepsilon,
\ee
for $-\max\{R_1,R_2\}\le x\le R_3$ and $y-c_{\alpha\beta}t\le -R_5$.

On the other hand, for any fixed $r\in\R$ and any point $(t,x,y)\in\R\times\R^2$ such that $x\cos\alpha+y\sin\alpha-c_{\alpha}t=r$, one has that 
$$x\cos\beta+y\sin\beta-c_{\beta}t=x\frac{\sin(\alpha-\beta)}{\sin\alpha}+\frac{\sin\beta}{\sin\alpha} r\rightarrow +\infty, \hbox{ as $x\rightarrow -\infty$},$$
since $-\pi<\alpha-\beta<0$ and $c_{\alpha}/\sin\alpha=c_{\beta}\sin\beta$. It implies that $U_{\beta}(x\cos\beta+y\sin\beta-c_{\beta}t,x,y)\rightarrow 0$ as $x\rightarrow -\infty$ for $(t,x,y)\in\R\times\R^2$ such that $x\cos\alpha+y\sin\alpha-c_{\alpha}t=r$. While, $U_{\alpha}(r,x,y)>0$. Thus, there is $R_6>0$ such that
$$U^-_{\alpha\beta}(t,x,y)=U_{\alpha}(x\cos\alpha+y\sin\alpha-c_{\alpha}t), \hbox{ for $x\le -R_6$ and $t\in\R$, $y\in\R$}.$$
Similarly, there is $R_7>0$ such that
$$U^-_{\alpha\beta}(t,x,y)=U_{\beta}(x\cos\beta+y\sin\beta-c_{\beta}t), \hbox{ for $x\ge R_7$ and $t\in\R$, $y\in\R$}.$$
Then, by \eqref{R2}-\eqref{R5}, we have our conclusion \eqref{U+cU-}.

{\it Step 3: the proof of \eqref{U+geU-}.} We only have to prove that $U^+(t,x,y)\ge U_{\alpha}(x\cos\alpha+y\sin\alpha-c_{\alpha}t)$ and $U^+(t,x,y)\ge U_{\beta}(x\cos\beta+y\sin\beta-c_{\beta}t)$ for all $t\in\R$ and $(x,y)\in\R^2$.

Since $U_e(-\infty,x,y)=1$ and $U_e(+\infty,x,y)=0$ for any $(x,y)\in\mathbb{T}^2$ and $e\in\mathbb{S}$, there is $C>0$ such that
$$0<U_e(\xi,x,y)\le \sigma \hbox{ for $\xi\ge C$ and $(x,y)\in\mathbb{T}^2$},$$ 
and
$$1-\sigma\le U_e(\xi,x,y)<1 \hbox{ for $\xi\le -C$ and $(x,y)\in\mathbb{T}^2$},$$
where $\sigma$ is defined in (F3).
By \eqref{U+cU-} and letting $\varepsilon\le \sigma/4$, there is $R>0$ such that 
$$U^+(t,x,y)\le \sigma, \hbox{ for $(t,x,y)\in\Omega_R^+$ and } U^+(t,x,y)\ge 1-\sigma, \hbox{ for $(t,x,y)\in\Omega_R^-$},$$
where
\begin{align*}
\Omega_R^+:=\{(t,x,y)\in\R\times\R^2; x\le 0 \hbox{ and }& x\cos\alpha+y\sin\alpha-c_{\alpha}t\ge c_{\alpha}R\}\cup \{(t,x,y)\in\R\times\R^2;\\
 &x> 0 \hbox{ and } x\cos\beta+y\sin\beta-c_{\beta}t\ge c_{\beta}R\},
\end{align*}
and
\begin{align*}
\Omega_R^-:=\{(t,x,y)\in\R\times\R^2; x\le 0 \hbox{ and }& x\cos\alpha+y\sin\alpha-c_{\alpha}t\le -c_{\alpha}R\}\cup \{(t,x,y)\in\R\times\R^2;\\
 &x> 0 \hbox{ and } x\cos\beta+y\sin\beta-c_{\beta}t\le -c_{\beta} R\}.
\end{align*}
Notice that for any $t$, the boundaries of $\Omega^+_t$ and $\Omega^-_t$ are connected polylines since $c_{\alpha}/\sin\alpha=c_{\beta}/\sin\beta$.
By Lemma~\ref{lemma-ut'} and the definition of $U^+(t,x,y)$, there is $0<\sigma'\le \sigma$ such that
$$\sigma'\le U^+(t,x,y)\le 1-\sigma', \hbox{ for $(t,x,y)\in\R\times\R^2\setminus (\Omega_R^+\cup\Omega_R^-)$}.$$

For any $\tau\in\R$, let $u_{\tau}(t,x,y)=U_{\alpha}(x\cos\alpha+y\sin\alpha-c_{\alpha}t+\tau)$. Let 
\begin{align*}
\omega_{\tau}^+:=\{(t,x,y)\in\R\times\R^2;  x\cos\alpha+y\sin\alpha-c_{\alpha}t+\tau\ge C\},
\end{align*}
and
\begin{align*}
\omega_{\tau}^-:=\{(t,x,y)\in\R\times\R^2;  x\cos\alpha+y\sin\alpha-c_{\alpha}t+\tau\le -C\}.
\end{align*}
Notice that since $\alpha<\beta$, one has that
$$\{(t,x,y)\in\R\times\R^2; x\cos\alpha+y\sin\alpha-c_{\alpha}t\le -c_{\alpha}R\}\subset \Omega_R^-,$$
and
$$\Omega_R^+\subset \{(t,x,y)\in\R\times\R^2; x\cos\alpha+y\sin\alpha-c_{\alpha}t\ge c_{\alpha}R\}.$$
Thus,
\begin{align*}
\R\times\R^2\setminus (\omega^+_{\tau}\cup\omega^-_{\tau})\subset \Omega_{(C-\tau)/c_{\alpha}}^-\hbox{ and }
\R\times\R^2\setminus(\Omega_R^+\cup\Omega_R^-) \subset \omega_{C+c_{\alpha} R}^+.
\end{align*}
Then, by \eqref{U+cU-}, $U_e(-\infty,x,y)=1$ and $U_e(+\infty,x,y)=0$, there is $\tau_1\ge c_{\alpha}R+C$ large enough such that for any $\tau\ge \tau_1$,
$$U^+(t,x,y)\ge 1-\sigma'\ge u_{\tau}(t,x,y), \hbox{ for all $(t,x,y)\in\R\times\R^2\setminus (\omega^+_{\tau}\cup\omega^-_{\tau})$},$$
and
$$u_{\tau}(t,x,y)\le \sigma'\le U^+(t,x,y), \hbox{ for all $(t,x,y)\in\R\times\R^2\setminus (\Omega^+_{R}\cup\Omega^-_{R})$}.$$
Moreover, since $\tau\ge \tau_1\ge R+C$, one has that
$$U^+(t,x,y)\ge 1-\sigma\ge \sigma\ge u_{\tau}(t,x,y), \hbox{ for all $(t,x,y)\in \omega^+_{\tau}\cap\Omega_{R}^-$}.$$
Thus, it follows that
\be\label{tau-middle}
u_{\tau}(t,x,y)\le U^+(t,x,y), \hbox{ for any $\tau\ge \tau_1$ and all $(t,x,y)\in\R\times\R^2\setminus(\omega^-_{\tau}\cup\Omega^+_R)$}.
\ee

Also notice that 
$$u_{\tau}(t,x,y),\ U^+(t,x,y)\ge 1-\sigma \hbox{ in $\omega_{\tau}^-$ and } u_{\tau}(t,x,y),\ U^+(t,x,y)\le \sigma \hbox{ in $\Omega_R^+$},$$
and $f(x,y,u)$ is nonincreasing in $u\in (-\infty,\sigma]$ and $u\in [1-\sigma,+\infty)$ for any $(x,y)\in\mathbb{T}^2$ by \eqref{lambda}.
By following similar proof as the proof of \cite[Lemma~4.2]{BH2} which mainly applied the sliding method and the linear parabolic estimates, one can get that
$$U^+(t,x,y)\ge u_{\tau}(t,x,y), \hbox{ in $\omega_{\tau}^-$ and $\Omega_R^+$}.$$
Combined with \eqref{tau-middle}, one has that
$$U^+(t,x,y)\ge u_{\tau}(t,x,y), \hbox{ for any $\tau\ge \tau_1$ and all $(t,x,y)\in\R\times\R^2$}.$$

Now, we decrease $\tau$. Define
\begin{align*}
\tau_*=\inf\{\tau\in\R; U^+(t,x,y)\ge u_{\tau}(t,x,y) \hbox{ for all $(t,x,y)\in\R\times\R^2$}\}.
\end{align*}
From above arguments, one knows that $\tau_*<+\infty$.
Since $U^+(t,x,y)\rightarrow U_{\alpha}(x\cos\alpha+y\sin\alpha-c_{\alpha}t,x,y)$ as $x\rightarrow -\infty$, $U_{\alpha}(\xi,x,y)$ is decreasing in $\xi$ and by the definition of $u_{\tau}(t,x,y)$, one also knows that $\tau_*\ge 0$. Assume that $\tau_*>0$. If 
$$\inf\{U^+(t,x,y)-u_{\tau_*}(t,x,y); (t,x,y)\in \R\times\R^2\setminus (\omega^-_{\tau_*}\cup\Omega_R^+)\}>0, $$
then there is $\eta>0$ such that 
\begin{align*}
U^+(t,x,y)\ge u_{\tau_*-\eta}(t,x,y), &\hbox{ for $(t,x,y)\in \R\times\R^2\setminus (\omega^-_{\tau_*-\eta}\cup\Omega_R^+)$}.
\end{align*}
Then, one can apply the above arguments again and get that $U^+(t,x,y)\ge u_{\tau_*-\eta}(t,x,y)$ for all $(t,x,y)\in\R\times\R^2$ which contradicts the definition of $\tau_*$. Thus, 
$$\inf\{U^+(t,x,y)-u_{\tau_*}(t,x,y); (t,x,y)\in \R\times\R^2\setminus (\omega^-_{\tau_*}\cup\Omega_R^+)\}=0.$$
Since $\alpha<\beta$, there is a sequence $\{(t_n,x_n,y_n)\}_{n\in\mathbb{N}}$ in $\R\times\R^2\setminus (\omega^-_{\tau_*}\cup\Omega_R^+)$ such that 
$$-C-\tau_*\le x_n\cos\alpha+y_n\sin\alpha-c_{\alpha}t_n\le c_{\alpha} R,$$
and
$$U^+(t_n,x_n,y_n)-u_{\tau_*}(t_n,x_n,y_n)\rightarrow 0, \hbox{ as $n\rightarrow +\infty$}.$$
Then, there is $\xi_*\in\R$ such that 
$x_n\cos\alpha+y_n\sin\alpha-c_{\alpha}t_n\rightarrow \xi_*$ as $n\rightarrow +\infty$.
Since $U^+(t,x,y)\rightarrow U_{\alpha}(x\cos\alpha+y\sin\alpha-c_{\alpha}t,x,y)$ as $x\rightarrow -\infty$, $U^+(t,x,y)\rightarrow U_{\beta}(x\cos\beta+y\sin\beta-c_{\beta}t,x,y)$ as $x\rightarrow +\infty$ with $\alpha<\beta$ and $\tau_*>0$, one has that $x_n$ is bounded and there is $x_*\in\R$ such that $x_n\rightarrow x_*$ as $n\rightarrow +\infty$. Again by $U^+(t,x,y)\rightarrow U_{\alpha}(x\cos\alpha+y\sin\alpha-c_{\alpha}t,x,y)$ as $x\rightarrow -\infty$ and by \eqref{R2}, there is $R'>0$ such that 
\be\label{R'}
|U^+(t,x,y)- U_{\alpha}(x\cos\alpha+y\sin\alpha-c_{\alpha}t,x,y)|\le \varepsilon \hbox{ for $x\le -R'$ and $t\in\R$, $y\in\R$}.
\ee
Let $v(t,x,y)=U^+(t,x,y)-u_{\tau_*}(t,x,y)$. Then, $v(t,x,y)\ge 0$ and 
\be\label{vR'}
v(t,x,y)>0 \hbox{ for any $(t,x,y)\in\R\times\R^2$ such that $x\le -R$, $x\cos\alpha+y\sin\alpha-c_{\alpha}t=\xi_*$},
\ee
by \eqref{R'}, $\tau_*>0$ and taking $\varepsilon$ sufficiently small.
Since $U^+(t,x,y)$ is a supersolution and $u_{\tau_*}(t,x,y)$ is a solution of \eqref{eq1.1}, we have that $v(t,x,y)$ satisfies
$$v_t-\Delta v\ge -b(x,y)v, \hbox{ for $(t,x,y)\in\R\times\R^2$},$$
where $b(x,y)$ is bounded. Since $v(t_n,x_n,y_n)\rightarrow 0$ and by the linear parabolic estimates and $x_n$ is bounded, one gets that
$$v(t_n-1,x_n-R',y_n+\frac{R'\cos\alpha -c_{\alpha}}{\sin\alpha})\rightarrow 0 \hbox{ as $n\rightarrow +\infty$},$$
which contradicts \eqref{vR'}. Thus, $\tau_*=0$ and $U^+(t,x,y)\ge U_{\alpha}(x\cos\alpha+y\sin\alpha-c_{\alpha}t,x,y)$ for all $(t,x,y)\in\R\times\R^2$.

Similarly one can prove that $U^+(t,x,y)\ge U_{\beta}(x\cos\beta+y\sin\beta-c_{\beta}t,x,y)$ for all $(t,x,y)\in\R\times\R^2$. In conclusion, $U^+(t,x,y)\ge U^-_{\alpha\beta}(t,x,y)$ for all $(t,x,y)\in\R\times\R^2$.
\end{proof}
\vskip 0.3cm

\begin{proof}[Proof of Claim~\ref{claim1}]
Notice that 
$$-\frac{\psi'(\rho x)}{\sqrt{\psi'^2(\varrho x)+1}}=e_1(x) \hbox{ and } \frac{1}{\sqrt{\psi'^2(\varrho x)+1}}=e_2(x).$$
Let $\hat{\theta}(x)=\arccos e_1(x)$. By Lemma~\ref{lemma-psi}, one can get that $\alpha<\hat{\theta}(x)<\beta$ for all $x\in\R$ and $\hat{\theta}(-\infty)=\alpha$, $\hat{\theta}(+\infty)=\beta$. Then, $e(x)=(\cos\hat{\theta},\sin\hat{\theta})$ and
$$\frac{c_{e(x)}}{e_2(x)}=\frac{c_{\hat{\theta}}}{\sin\hat{\theta}}.$$
Thus,
$$\frac{c_{\alpha\beta}}{\sqrt{\psi'^2(\varrho x)+1}}-c_{e(x)}=\sin\hat{\theta}\left(c_{\alpha\beta}-\frac{c_{\hat{\theta}}}{\sin\hat{\theta}}\right).$$
Since $c_{\alpha\beta}>c_{\theta}/\sin\theta$ for any $\theta\in(\alpha,\beta)$ and $0<\min\{\sin\alpha,\sin\beta\}\le \sin\hat{\theta}\le 1$, one only has to prove that 
\be\label{c2}
c_{\alpha\beta}-\frac{c_{\hat{\theta}}}{\sin\hat{\theta}}\ge C_7\text{sech}(\varrho x), \hbox{ for some positive constant $C_7$ and when $|x|$ is large}.
\ee

We only consider when $x<0$ and similar arguments can be applied for $x>0$. Define
$$g(\theta)=\frac{c_{\theta}}{\sin\theta}, \hbox{ for $\theta\in(0,\pi)$}.$$
Obviously, $g(\theta)$ is a $C^2$ function since $c_{e}$ is doubly differentiable with respect to $e$. By \eqref{angle}, one has that $g'(\alpha)<0$. Since $\hat{\theta}(x)\rightarrow \alpha$ as $x\rightarrow -\infty$, it then follows that
\be\label{g'}
\frac{c_{\alpha}}{\sin\alpha}-\frac{c_{\hat{\theta}}}{\sin\hat{\theta}}=g'(\alpha)(\alpha-\hat{\theta}(x))+o(|\alpha-\hat{\theta}(x)|), \hbox{ for $x$ negative enough}.
\ee
Moreover, by \eqref{K_1}, one has that
\begin{align*}
\hat{\theta}(x)-\alpha=&\int_{-\infty}^{x} \hat{\theta}'(s)ds=- \int_{-\infty}^{x} \frac{e_1'(s)}{\sqrt{1-e_1^2(s)}}ds=\int_{\infty}^x \frac{\varrho\psi''(\varrho s)}{\psi'^2(\varrho s)+1} ds\\
\ge& \frac{1}{|\psi'|^2_{L^{\infty}}+1} (\psi'(\varrho x)+\cot\alpha)\ge \frac{k_1}{|\psi'|^2_{L^{\infty}}+1} \text{sech}(\varrho x).
\end{align*}
One then can conclude \eqref{c2} from \eqref{g'} for $x$ negative enough.
\end{proof}
\vskip 0.3cm

Now, we are ready to prove Theorem~\ref{th1}.

\begin{proof}[Proof of Theorem~\ref{th1}]
Let $u_n(t,x)$ be the solution of \eqref{eq1.1} for $t\ge -n$ with initial data
$$u_n(-n,x,y)=U^-_{\alpha\beta}(-n,x,y).$$
By Lemma~\ref{U+}, one can get from the comparison principle that
\be\label{U-U+}
U^-_{\alpha\beta}(t,x,y)\le u_n(t,x,y)\le U^+(t,x,y), \hbox{ for $t\ge -n$ and $(x,y)\in \R^2$}.
\ee
Since $U^-_{\alpha\beta}(t,x,y)$ is a subsolution, the sequence $u_n(t,x,y)$ is increasing in $n$. Letting $n\rightarrow +\infty$ and by parabolic estimates, the sequence $u_n(t,x,y)$ converges to an entire solution $V(t,x,y)$ of \eqref{eq1.1}. By \eqref{U-U+}, one has that
$$U^-_{\alpha\beta}(t,x,y)\le V(t,x,y)\le U^+(t,x,y), \hbox{ for $t\in\R$ and $(x,y)\in \R^2$}.$$
Then, it follows from Lemma~\ref{U+} that \eqref{Vf} holds. 

By $U_{\alpha\beta}^-(t,x,y)$ is increasing in $t$ and the maximum principle, one has that $(u_n)_t(t,x,y)>0$ for all $t\in(-n,+\infty)$ and $(x,y)\in\R^2$. By letting $n\rightarrow +\infty$ and the strong maximum principle, one concludes that $u_t(t,x,y)>0$ for all $t\in\R$ and $(x,y)\in\R^2$. This completes the proof.
\end{proof}

\subsection{Proofs of Corollaries~\ref{cor1}, \ref{cor2} and Theorem~\ref{th2}}

We then give some examples to show that Theorem~\ref{th1} is not empty, that is, Corollaries~\ref{cor1}, \ref{cor2}.

\begin{proof}[Proof of Corollary~\ref{cor1}]
Notice that $c_{\theta}$ and $c'_{\theta}$ are uniformly bounded for $\theta\in[0,\pi]$. Let $g(\theta):=c_{\theta}/\sin\theta$. Then,
$$g'(\theta)=\frac{c'_{\theta}\cdot (-\sin\theta,\cos\theta)}{\sin\theta}-\frac{c_{\theta}\cos\theta}{\sin^2\theta}.$$
Obviously, there are constants $0<\alpha_1<\beta_1<\pi$ such that $g'(\theta)<0$ for $\theta\in (0,\alpha_1)$ and $g'(\theta)>0$ $\theta\in(\beta_1,\pi)$ since $c'_e$ is bounded for any $e\in\mathbb{S}$ and $\sin\theta\rightarrow 0$ as $\theta\rightarrow 0$ or $\pi$. One can also notice that $g(\theta)\rightarrow +\infty$ as $\theta\rightarrow 0$ or $\theta\rightarrow \pi$. By continuity, one can take any $\alpha\in(0,\alpha_1)$ and there is $\beta\in(\beta_1,\pi)$ such that $g(\alpha)=g(\beta)$ and $g(\theta)<g(\alpha)=g(\beta)$ for all $\theta\in(\alpha,\beta)$.

Then, the conclusion of Corollary~\ref{cor1} follows from Theorem~\ref{th1}.
\end{proof}
\vskip 0.3cm

\begin{proof}[Proof of Corollary~\ref{cor2}]
Take two directions $e_1=(\cos\theta_1,\sin\theta_1)$ and $e_2=(\cos\theta_2,\sin\theta_2)$ where $\theta_1$, $\theta_2\in (0,2\pi)$. Assume without loss of generality that $\theta_2>\theta_1$. Rotate the coordinate by changing variables as 
\begin{eqnarray*}
\left\{\begin{array}{lll}
X=x\cos\theta+y\sin\theta,&&\\
Y=-x\sin\theta+y\cos\theta,&&
\end{array}
\right.
\end{eqnarray*}
where $\theta$ varies from $\theta_2-\pi/2$ to $\theta_1+\pi/2$. Let $u(t,x,y)$ be the solution of \eqref{eq1.1} and let $v(t,X,Y)=u(t,x,y)=u(t,X\cos\theta-Y\sin\theta,X\sin\theta+Y\cos\theta)$. Then, $v(t,X,Y)$ satisfies
\be\label{vXY}
v_t-\Delta_{X,Y} v=g(X,Y,u),\quad\,(t,X,Y)\in\R\times\R^2,
\ee
where $g(X,Y,v)=f(X\cos\theta-Y\sin\theta,X\sin\theta+Y\cos\theta,v)$. Notice that $g(X,Y,v)$ satisfies (F1)-(F3). 

Assume without loss of generality that $\theta_2-\theta_1<\pi$. Otherwise, if $\theta_2-\theta_1>\pi$, we can take $\theta$ varying from $\theta_2$ to $2\pi+\theta_1$. Then, under the new coordinate, directions $e_1$ and $e_2$ become $(\cos(\theta_1+\pi/2-\theta),\sin(\theta_1+\pi/2-\theta))$ and $(\cos(\theta_2+\pi/2-\theta),\sin(\theta_2+\pi/2-\theta))$ where $0<\theta_1+\pi/2-\theta<\theta_2+\pi/2-\theta<\pi$. Since $\sin\theta$ is increasing in $[0,\pi/2]$ and decreasing in $[\pi/2,\pi]$, one has that
$$\frac{c_{e_1}}{\sin(\theta_1+\pi/2-\theta)} \hbox{ is increasing from $\frac{c_{e_1}}{\sin(\theta_1-\theta_2+\pi)}$ to $+\infty$ as $\theta$ varies from $\theta_2-\pi/2$ to $\theta_1+\pi/2$},$$
and
$$\frac{c_{e_2}}{\sin(\theta_2+\pi/2-\theta)} \hbox{ is decreasing from $+\infty$ to $\frac{c_{e_2}}{\sin(\theta_2-\theta_1)}$ as $\theta$ varies from $\theta_2-\pi/2$ to $\theta_1+\pi/2$}.$$
By continuity and for any $0<\theta_2-\theta_1<\pi$, there is $\theta^*$ such that
$$\frac{c_{e_1}}{\sin(\theta_1+\pi/2-\theta^*)}=\frac{c_{e_2}}{\sin(\theta_2+\pi/2-\theta^*)}.$$
On the other hand, by the proof of Corollary~\ref{cor1}, there is $0<\alpha_1<\pi$ small enough such that for $0<\pi-(\theta_2-\theta_1)<\alpha_1$, it holds
$$\frac{c_{e_1}}{\sin(\theta_1+\pi/2-\theta^*)}=\frac{c_{e_2}}{\sin(\theta_2+\pi/2-\theta^*)}>\frac{c_{\theta}}{\sin(\theta-\theta^*)} \hbox{ for $\theta_1+\pi/2<\theta<\theta_1+\pi/2$}.$$
Then, it follows from Corollary~\ref{cor1} that there is an entire solution $V(t,X,Y)$ of \eqref{vXY} satisfying  \eqref{Ve1e2}. By taking $\rho=\cos(\pi-\alpha_1)-1$ and $e_0=(\cos\theta^*,\sin\theta^*)$, the conclusion of Corollary~\ref{cor2} immediately follows.
\end{proof}
\vskip 0.3cm

Now, we show that condition \eqref{angle} without $g'(\alpha)<0$ and $g'(\beta)>0$ is necessary for the existence of the curved front in Theorem~\ref{th1}.

\begin{proof}[Proof of Theorem\ref{th2}]
We first prove that 
\be\label{ca=cb}
\frac{c_{\alpha}}{\sin\alpha}=\frac{c_{\beta}}{\sin\beta}.
\ee
Assume by contradiction that $c_{\alpha}/\sin\alpha\neq c_{\beta}/\sin\beta$. Take a sequence $\{t_n\}_{n\in\mathbb{N}}$ such that $t_n\rightarrow +\infty$. Then, for the sequence 
$$(x_n,y_n)=\Big(\frac{(c_{\alpha}\sin\beta-c_{\beta}\sin\alpha)t_n}{\sin(\beta-\alpha)},\frac{(c_{\alpha}\cos\beta-c_{\beta}\cos\alpha)t_n}{\sin(\alpha-\beta)}\Big),$$
one has that $x_n^2+(y_n-c_{\alpha\beta}t_n)^2\rightarrow +\infty$ as $n\rightarrow +\infty$ for any $c_{\alpha\beta}\in\R$ since $c_{\alpha}/\sin\alpha\neq c_{\beta}/\sin\beta$. Notice that for any $n$, there are $k^1_n$, $k^2_n\in\mathbb{Z}$ and $x'_n$, $y'_n\in [0,L_2)$ such that $x_n=k^1_n L_1+x'_n$ and $y_n=k^2_n L_2+y'_n$. Moreover, up to extract subsequences of $x_n$ and $y_n$, there are $x'_*\in [0,L_1]$ and $y'_*\in [0,L_2]$ such that $x_n'\rightarrow x'_*$ and $y'_n\rightarrow y'_*$ as $n\rightarrow +\infty$. Since $f(x,y,\cdot)$ is $L$-periodic in $(x,y)$, one has $f(x+x_n,y+y_n,\cdot)\rightarrow f(x+x'_*,y+y'_*,\cdot)$ as $n\rightarrow +\infty$. Let $v_n(t,x,y)=V(t+t_n,x+x_n,y+y_n)$. By standard parabolic estimates, $v_n(t,x,y)$, up to extract of a subsequence, converges to a solution $v_{\infty}(t,x,y)$ of 
\be\label{eq-v1}
v_t-\Delta v=f(x+x'_*,y+y'_*,v), \hbox{ for $t\in\R$ and $(x,y)\in\R^2$}.
\ee
By definitions of $x_n$ and $y_n$, one can also have that
$$U^-_{\alpha\beta}(t+t_n,x+x_n,y+y_n)\rightarrow \hat{U}^-_{\alpha\beta}(t,x,y), \hbox{ as $n\rightarrow +\infty$ uniformly in $\R\times\R^2$},$$
where
$$\hat{U}^-_{\alpha\beta}(t,x,y):=\max\{U_\alpha(x\cos\alpha+y\sin\alpha-c_{\alpha} t,x+x'_*,y+y'_*),U_\beta(x\cos\beta+y\sin\beta-c_{\beta} t,x+x'_*,y+y'_*)\}.$$ 
Moreover, by \eqref{Vf} and $x_n^2+(y_n-c_{\alpha\beta} t_n)^2\rightarrow +\infty$ as $n\rightarrow +\infty$, one gets that
$$v_n(t,x,y)\rightarrow \hat{U}^-_{\alpha\beta}(t,x,y) \hbox{ as $n\rightarrow +\infty$ locally uniformly in $\R\times\R^2$}.$$
It implies that $v_{\infty}(t,x,y)=\hat{U}^-_{\alpha\beta}(t,x,y)$ which is impossible since $\hat{U}^-_{\alpha\beta}(t,x,y)$ is not a solution of \eqref{eq-v1}. Therefore, \eqref{ca=cb} holds.

Then, we prove that 
\be\label{cab=ca}
c_{\alpha\beta}=\frac{c_{\alpha}}{\sin\alpha}=\frac{c_{\beta}}{\sin\beta}.
\ee
Assume by contradiction that $c_{\alpha\beta}\neq c_{\alpha}/\sin\alpha$. Take a sequence $(t_n)_{n\in\mathbb{N}}$ such that $t_n=L_2 n\sin\alpha/c_{\alpha}  \rightarrow +\infty$ and consider the sequence
$$(x_n,y_n)=(0,\frac{c_{\alpha}}{\sin\alpha}t_n).$$
Notice that $x_n^2+(y_n-c_{\alpha\beta}t_n)^2\rightarrow +\infty$ as $n\rightarrow +\infty$ since $c_{\alpha\beta}\neq c_{\alpha}/\sin\alpha$, $t_nc_{\alpha}/\sin\alpha=nL_2$ and $U^-_{\alpha\beta}(t+t_n,x+x_n,y+y_n)=U^-_{\alpha\beta}(t,x,y)$. Then, one can make the similar arguments as above to get a contradiction. Thus, \eqref{cab=ca} holds.

At last, we prove that
$$\frac{c_{\theta}}{\sin\theta}<c_{\alpha\beta}=\frac{c_{\alpha}}{\sin\alpha}=\frac{c_{\beta}}{\sin\beta} \hbox{ for any $\theta\in(\alpha,\beta)$}.$$
Assume by contradiction that there is $\theta\in(\alpha,\beta)$ such that $c_{\theta}/\sin\theta\ge c_{\alpha\beta}$. Then, two cases may occur: (i) $c_{\theta}/\sin\theta> c_{\alpha\beta}$; (ii) $c_{\theta}/\sin\theta= c_{\alpha\beta}$.

For case (i), take $t=0$ and by \eqref{Vf}, for any $\varepsilon>0$, there is $R_{\varepsilon}>0$ such that 
\be\label{Re}
\sup_{|(x,y)|>R_{\varepsilon}} \Big| V(0,x,y)-U^-_{\alpha\beta}(0,x,y)\Big|\le \varepsilon.
\ee 
We claim that 

\begin{claim}\label{claim-th2}
There exist constants $\tau\in\R$ and $\delta>0$ such that
$$V(t,x,y)\ge U_{\theta}(x\cos\theta+y\sin\theta -c_{\theta} t+\tau,x,y)-\delta e^{-\delta t}, \hbox{ for $t\ge 0$ and $x\in\R^2$}.$$
\end{claim}
In order to not lengthen the proof, we postpone the proof of this claim after the proof of Theorem~\ref{th2}. Take a sequences $(t_n)_{n\in\mathbb{N}}$ such that $t_n\rightarrow +\infty$ as $n\rightarrow +\infty$ and $y_n=c_{\alpha\beta} t_n +R$ where $R$ is a constant. Then, since $U_e(+\infty,x,y)=0$ for all $e\in\mathbb{S}$ and $(x,y)\in\mathbb{T}^2$, one can take $R$ large enough such that 
\begin{align*}
U^-_{\alpha\beta}(t_n,0,y_n)=&\max\{U_{\alpha}(y_n\sin\alpha-c_{\alpha}t_n,0,y_n),U_{\beta}(y_n\sin\beta-c_{\beta}t_n,0,y_n)\}\\
=&\max\{U_{\alpha}(R\sin\alpha,0,y_n),U_{\beta}(R\sin\beta,0,y_n)\}\le \frac{1}{4}.
\end{align*}
By \eqref{Vf} and even if it means increasing $R$, one has that 
\be\label{vle}
V(t_n,0,y_n)\le U^-_{\alpha\beta}(t_n,0,y_n)+\frac{1}{4}\le \frac{1}{2} \hbox{ for all $n$}.
\ee
However, since $c_{\theta}/\sin\theta>c_{\alpha\beta}$ and hence,
$$y_n\sin\theta-c_{\theta} t_n=(c_{\alpha\beta}\sin\theta-c_{\theta})t_n+R\sin\theta\rightarrow -\infty, \hbox{ as $n\rightarrow +\infty$},$$
it follows from Claim~\ref{claim-th2} that 
$$V(t_n,0,y_n)\ge U_{\theta}(y_n\sin\theta-c_{\theta} t_n+\tau,0,y_n)\rightarrow 1 \hbox{ as $n\rightarrow +\infty$},$$
which contradicts \eqref{vle}. Case (i) is ruled out.

Now we consider case (ii). Since $U_e(-\infty,x,y)=1$ and $U_e(+\infty,x,y)=0$ for any $(x,y)\in\mathbb{T}^2$ and $e\in\mathbb{S}$, there is $C>0$ such that
$$0<U_e(\xi,x,y)\le \sigma \hbox{ for $\xi\ge C$ and $(x,y)\in\mathbb{T}^2$},$$ 
and
$$1-\sigma\le U_e(\xi,x,y)<1 \hbox{ for $\xi\le -C$ and $(x,y)\in\mathbb{T}^2$},$$
where $\sigma$ is defined in (F3).
By \eqref{Vf}, there is $R>0$ such that 
$$V(t,x,y)\le \sigma, \hbox{ for $(t,x,y)\in\Omega_R^+$ and } V(t,x,y)\ge 1-\sigma, \hbox{ for $(t,x,y)\in\Omega_R^-$},$$
where
\begin{align*}
\Omega_R^+:=\{(t,x,y)\in\R\times\R^2; x\le 0 \hbox{ and }& x\cos\alpha+y\sin\alpha-c_{\alpha}t\ge c_{\alpha}R\}\cup \{(t,x,y)\in\R\times\R^2;\\
 &x> 0 \hbox{ and } x\cos\beta+y\sin\beta-c_{\beta}t\ge c_{\beta}R\},
\end{align*}
and
\begin{align*}
\Omega_R^-:=\{(t,x,y)\in\R\times\R^2; x\le 0 \hbox{ and }& x\cos\alpha+y\sin\alpha-c_{\alpha}t\le -c_{\alpha}R\}\cup \{(t,x,y)\in\R\times\R^2;\\
 &x> 0 \hbox{ and } x\cos\beta+y\sin\beta-c_{\beta}t\le -c_{\beta}R\}.
\end{align*}
By a similar proof as of Lemma~\ref{lemma-ut'}, there is $0<\sigma'\le \sigma$ such that
$$\sigma'\le V(t,x,y)\le 1-\sigma', \hbox{ for $(t,x,y)\in\R\times\R^2\setminus (\Omega_R^+\cup\Omega_R^-)$}.$$

For any $\tau\in\R$, let $u_{\tau}(t,x,y)=U_{\theta}(x\cos\theta+y\sin\theta -c_{\theta} t+\tau,x,y)$.  Let 
\begin{align*}
\omega_{\tau}^+:=\{(t,x,y)\in\R\times\R^2;  x\cos\theta+y\sin\theta-c_{\theta}t+\tau\ge C\},
\end{align*}
and
\begin{align*}
\omega_{\tau}^-:=\{(t,x,y)\in\R\times\R^2;  x\cos\theta+y\sin\theta-c_{\theta}t+\tau\le -C\}.
\end{align*}
Since $\alpha<\theta<\beta$ and $c_{\theta}/\sin\theta=c_{\alpha}/\sin\alpha=c_{\beta}/\sin\beta$, one can easily check that 
\begin{align*}
\R\times\R^2\setminus (\omega^+_{\tau}\cup\omega^-_{\tau})\subset \Omega_{(C-\tau)/c_{\alpha}}^-\hbox{ and }
\R\times\R^2\setminus(\Omega_R^+\cup\Omega_R^-) \subset \omega_{C+c_{\alpha}R}^+.
\end{align*}
Then, by \eqref{Vf}, $U_e(-\infty,x,y)=1$ and $U_e(+\infty,x,y)=0$, there is $\tau_1\ge c_{\alpha}R+C$ large enough such that for any $\tau\ge \tau_1$,
$$V(t,x,y)\ge 1-\sigma'\ge u_{\tau}(t,x,y), \hbox{ for all $(t,x,y)\in\R\times\R^2\setminus (\omega^+_{\tau}\cup\omega^-_{\tau})$},$$
and
$$u_{\tau}(t,x,y)\le \sigma'\le V(t,x,y), \hbox{ for all $(t,x,y)\in\R\times\R^2\setminus (\Omega^+_{R}\cup\Omega^-_{R})$}.$$
Moreover, since $\tau\ge \tau_1\ge c_{\alpha}R+C$, one has that
$$V(t,x,y)\ge 1-\sigma\ge \sigma\ge u_{\tau}(t,x,y), \hbox{ for all $(t,x,y)\in \omega^+_{\tau}\cap\Omega_{R}^-$}.$$
Thus, it follows that
\be\label{tau-middle-2}
u_{\tau}(t,x,y)\le V(t,x,y), \hbox{ for any $\tau\ge \tau_1$ and all $(t,x,y)\in\R\times\R^2\setminus(\omega^-_{\tau}\cup\Omega^+_R)$}.
\ee
Also notice that 
$$u_{\tau}(t,x,y),\ V(t,x,y)\ge 1-\sigma \hbox{ in $\omega_{\tau}^-$ and } u_{\tau}(t,x,y),\ V(t,x,y)\le \sigma \hbox{ in $\Omega_R^+$},$$
and $f(x,y,u)$ is nonincreasing in $u\in (-\infty,\sigma]$ and $u\in [1-\sigma,+\infty)$ for any $(x,y)\in\mathbb{T}^2$ by \eqref{lambda}.
By following similar proof as the proof of \cite[Lemma~4.2]{BH2} which mainly applied the sliding method and the linear parabolic estimates, one can get that
$$V(t,x,y)\ge u_{\tau}(t,x,y), \hbox{ in $\omega_{\tau}^-$ and $\Omega_R^+$}.$$
Combined with \eqref{tau-middle-2}, one has that
$$V(t,x,y)\ge u_{\tau}(t,x,y), \hbox{ for any $\tau\ge \tau_1$ and all $(t,x,y)\in\R\times\R^2$}.$$

Let 
$$\tau_*=\inf\{\tau\in\R; u_{\tau}(t,x,y)\le V(t,x,y) \hbox{ for all $(t,x,y)\in\R\times\R^2$}\}.$$
By above arguments, one knows that $\tau_*<+\infty$. On the other hand, for any fixed $(t,x,y)$, $u_{\tau}(t,x,y)=U_{\theta}(x\cos\theta+y
\sin\theta-c_{\theta} t +\tau,x,y)\rightarrow 1$ as $\tau\rightarrow -\infty$ and $V(t,x,y)<1$ by the maximum principle. By the definition of $\tau_*$, one also has that $\tau_*>-\infty$. Thus, $|\tau_*|$ is bounded.
If 
$$\inf\{V(t,x,y)-u_{\tau_*}; (t,x,y)\in \R\times\R^2\setminus (\omega^-_{\tau_*}\cup\Omega_R^+)\}>0,$$
there is $\eta>0$ such that 
$$V(t,x,y)\ge u_{\tau_*-\eta}(t,x,y), \hbox{ for $(t,x,y)\in \R\times\R^2\setminus (\omega^-_{\tau_*}\cup\Omega_R^+)$}.$$
Then, one can follow above arguments again to get that
$$u_{\tau_*-\eta}(t,x,y)\le V(t,x,y), \hbox{ for $(t,x,y)\in\R\times\R^2$},$$
which contradicts the definition of $\tau_*$. Thus,
$$\inf\{V(t,x,y)-u_{\tau_*}; (t,x,y)\in \R\times\R^2\setminus (\omega^-_{\tau_*}\cup\Omega_R^+)\}=0.$$
Since $V(t,x,y)\ge \sigma'$ in $\R\times\R^2\setminus (\omega^-_{\tau_*}\cup\Omega_R^+)$ and $u_{\tau_*}(t,x,y)=U_{\theta}(x\cos\theta+y\sin\theta-c_{\theta}t+\tau_*,x,y)\rightarrow 0$ as $x\cos\theta+y\sin\theta-c_{\theta}t\rightarrow +\infty$,
there is $R_1>0$ and there is a sequence $\{(t_n,x_n,y_n)\}_{n\in\mathbb{N}}$ in $\R\times\R^2\setminus (\omega^-_{\tau_*}\cup\Omega_R^+)$ such that
\be\label{theta-tau}
-C-\tau_*\le x_n\cos\theta+y_n\sin\theta -c_{\theta} t_n\le R_1
\ee
and
\be\label{vrU}
V(t_n,x_n,y_n)-U_{\theta}(x_n\cos\theta+y_n\sin\theta -c_{\theta} t_n+\tau_*,x_n,y_n)\rightarrow 0, \hbox{ as $n\rightarrow +\infty$}.
\ee
Notice that $x_n$ is bounded. Otherwise, if $x_n\rightarrow -\infty$ as $n\rightarrow +\infty$, then it follows from \eqref{theta-tau} and $\theta<\alpha$ that
\begin{align*}
x_n\cos\alpha+y_n\sin\alpha-c_{\alpha} t_n&=x_n\cos\alpha+\sin\alpha\Big(y_n-\frac{c_{\alpha}}{\sin\alpha} t_n\Big)\\
&=x_n\cos\alpha+\sin\alpha\Big(y_n-\frac{c_{\theta}}{\sin\theta} t_n\Big)\\
&=\Big(\cos\alpha-\frac{\sin\alpha\cos\theta}{\sin\theta}\Big)x_n+\sin\alpha\sin\theta\Big(\cos\theta x_n+ \sin\theta y_n- c_{\theta} t_n\Big)\\
&\rightarrow +\infty,\quad\quad \hbox{ as $n\rightarrow +\infty$},
\end{align*}
and $x_n^2+(y_n-c_{\alpha\beta} t_n)^2\rightarrow +\infty$ as $n\rightarrow +\infty$. It implies that $V(t_n,x_n,y_n)\rightarrow U^-_{\alpha\beta}(t_n,x_n,y_n)\rightarrow 1$ as $n\rightarrow +\infty$ which contradicts  $u_{\tau_*}(t,x,y)\le 1-\sigma'$ in $\R\times\R^2\setminus \omega^-_{\tau_*}$ and \eqref{vrU}. Similarly, it is not possible that $x_n\rightarrow +\infty$ as $n\rightarrow +\infty$. Thus, there is $x_*\in\R$ such that $x_n\rightarrow x_*$ as $n\rightarrow +\infty$.
Let $w(t,x,y)=V(t,x,y)-u_{\tau_*}(t,x,y)$. Then, by \eqref{vrU}, $w(t_n,x_n,y_n)\rightarrow 0$ as $n\rightarrow +\infty$. Consider the point $(t_n-1,x_n-R',y_n-c_{\theta}/\sin\theta+R'\cos\theta/\sin\theta)$ for some constant $R'$. Notice that by \eqref{theta-tau},
$$(x_n-R')\cos\theta+(y_n-c_{\theta}/\sin\theta+R'\cos\theta/\sin\theta)\sin\theta - c_{\theta} (t_n-1)\in [-C-\tau_*,R_1],\hbox{ for any $n$},$$
and 
$$(x_n-R')\cos\alpha+(y_n-c_{\theta}/\sin\theta+R'\cos\theta/\sin\theta)\sin\alpha - c_{\alpha} (t_n-1)\rightarrow +\infty, \hbox{ as $R'\rightarrow +\infty$},$$
for any $n$.
 By taking $R'$ large enough, one can let 
$$V(t_n-1,x_n-R',y_n-c_{\theta}+R'\cos\theta/\sin\theta)> 1-\sigma', \hbox{ for any $n$}.$$
Then, by noticing that $(t_n-1,x_n-R,y_n-c_{\theta}+R\cos\theta/\sin\theta)$ satisfies \eqref{theta-tau} and hence $u_{\tau_*}(t_n-1,x_n-R,y_n-c_{\theta}+R\cos\theta/\sin\theta)\le 1-\sigma'$, one has that
\be\label{eq-wn}
w(t_n-1,x_n-R,y_n-c_{\theta}+R\cos\theta/\sin\theta)>0.
\ee
However, since $V(t,x,y)$ and $u_{\tau_*}(t,x,y)$ are solutions of \eqref{eq1.1}, we have that $w(t,x,y)$ satisfies
$$w_t-\Delta w\ge -b(x,y) w, \hbox{ for $(t,x,y)\in\R\times\R^2$},$$
where $b(x,y)$ is bounded. 
By the linear parabolic estimates, one can get that
$$w(t_n-1,x_n-R,y_n-c_{\theta}+R\cos\theta/\sin\theta)\rightarrow 0, \hbox{ as $n\rightarrow +\infty$},$$
which contradicts \eqref{eq-wn}. Therefore, case (ii) is ruled out.

In conclusion, $c_{\theta}/\sin\theta<c_{\alpha\beta}$ for any $\theta\in (\alpha,\beta)$.
\end{proof}
\vskip 0.3cm

We finish this section by proving Claim~\ref{claim-th2}.

\begin{proof}[Proof of Claim~\ref{claim-th2}]
Take $\delta>0$ such that 
$$\delta\le \min\left\{\frac{\sigma}{2},\lambda\right\},$$
where $\sigma$ and $\lambda$ are defined in (F3). Since $U_{\theta}(-\infty,x,y)=1$ and $U_{\theta}(+\infty,x,y)=0$ for $(x,y)\in\mathbb{T}^2$, there is $C>0$ such that 
\begin{eqnarray*}
\left\{\begin{array}{lll}
0<U_{\theta}(\xi,x,y)\le \delta, &&\hbox{ for $\xi\ge C$ and $(x,y)\in\mathbb{T}^2$}\\  
1-\delta\le U_{\theta}(\xi,x,y)< 1, &&\hbox{ for $\xi\le -C$ and $(x,y)\in\mathbb{T}^2$}.
\end{array}
\right.
\end{eqnarray*}
From Lemma~\ref{lemma-ut'}, there is $k>0$ such that $-\partial_{\xi}U_{\theta}(\xi,x,y)\ge k$ for $-C\le \xi\le C$ and $(x,y)\in\mathbb{T}^2$. 
Take $\omega>0$ such that
\be\label{omega}
k\omega\ge \delta +M,
\ee
where $M=\max_{(x,y,u)\in\mathbb{T}^2\times\R} |f_u(x,y,u)|$. It follows from \eqref{Re} and the definition of $U^-_{\alpha\beta}$ that there is $R_{\delta}>0$ such that
$$
V(0,x,y)\ge 1-\delta, \hbox{ for $(x,y)\in\Omega$}$$
where 
$$\Omega:=\{x\le 0, y\in\R; x\cos\alpha+y\sin\alpha\le -R_{\delta}\}\cup \{x\ge 0, y\in\R; x\cos\beta+y\sin\beta\le -R_{\delta}\}.$$

Define
$$v^-(t,x,y)=U_{\theta}(\xi^-(t,x,y),x,y)-\delta e^{-\delta t},$$
where
$$\xi^-(t,x,y)=x\cos\theta+y\sin\theta-c_{\theta} t-\omega e^{-\delta t}+\omega+\hat{R}_{\delta}+C,$$
and $\hat{R}_{\delta}=R_{\delta} \sin\theta \max\{1/\sin\alpha,1/\sin\beta\}$.
We prove that $v^-(t,x,y)$ is a subsolution of the problem satisfied by $V(t,x,y)$ for $t\ge 0$ and $(x,y)\in\R^2$.

Firstly, we check the initial data. Since $\alpha<\theta<\beta$, one has that
$$\{(x,y)\in\R^2; \xi^-(0,x,y)\le C\}\subset \Omega.$$
Then,
$$v^-(0,x,y)\le 1-\delta\le V(0,x,y), \hbox{ for $(x,y)\in\R^2$ such that $\xi(0,x,y)\le C$}.$$
For $(x,y)\in\R^2$ such that $\xi(0,x,y)\ge C$, one has that
$$v^-(0,x,y)\le \delta-\delta=0\le V(0,x,y).$$
Thus, $v^-(0,x,y)\le V(0,x,y)$ for all $(x,y)\in\R^2$.

We then check that
$$Nv:=v^-_t-\Delta v^- -f(x,y,v^-)\le 0, \hbox{ for $t\ge 0$ and $(x,y)\in\R^2$}.$$
By some computation and \eqref{Ue}, one has that
\be\label{c-nv}
Nv=\omega\delta e^{-\delta t} \partial_{\xi} U_{\theta}+\delta^2 e^{-\delta t} +f(x,y,U_{\theta})-f(x,y,v^-).
\ee
For $t\ge 0$ and $(x,y)\in\R^2$ such that $\xi(t,x,y)\ge C$, one has that $0<U_{\theta}(\xi(t,x,y),x,y)\le \delta$ and hence $v^-(t,x,y)\le 2\delta\le \sigma$. Thus, by \eqref{lambda}, it follows that
\be\label{c-f1}
f(x,y,U_{\theta})-f(x,y,v^-)\le -\lambda\delta e^{-\delta t}.
\ee
Since $\partial_{\xi} U_{\theta}<0$, it follows from \eqref{c-nv} and \eqref{c-f1} that
\begin{align*}
Nv\le \delta^2 e^{-\delta t} -\lambda\delta e^{-\delta t}\le 0.
\end{align*}
Similarly, one can prove that $Nv\le 0$ for $t\ge 0$ and $(x,y)\in\R^2$ such that $\xi(t,x,y)\le -C$.
Finally, for $t\ge 0$ and $(x,y)\in\R^2$ such that $-C\le \xi(t,x,y)\le C$, one has that $-\partial_{\xi}U_{\theta}(\xi(t,x,y),x,y)\ge k$ and
\be\label{c-f2}
f(x,y,U_{\theta})-f(x,y,v^-)\le M\delta e^{-\delta t},
\ee
where $M=\max_{(x,y,u)\in\mathbb{T}^2\times\R} |f_u(x,y,u)|$. Then, it follows from \eqref{omega}, \eqref{c-nv} and \eqref{c-f2} that
\begin{align*}
Nv\le -k\omega\delta e^{-\delta t} +\delta^2 e^{-\delta t} +M\delta e^{-\delta t}\le 0.
\end{align*}

By the comparison principle, one gets that
$$V(t,x,y)\ge v^-(t,x,y), \hbox{ for $t\ge 0$ and $x\in\R^2$}.$$
Then, the conclusion of Claim~\ref{claim-th2} follows immediately.
\end{proof}
\vskip 0.3cm

\section{Uniqueness and stability of the curved front}

This section is devoted to the proofs of uniqueness and stability of the curved front in Theorem~\ref{th1}, that is, Theorems~\ref{th3} and \ref{th4}.

\subsection{Proof of Theorem~\ref{th3}}

The idea of the proof of the uniqueness is inspired by Berestycki and Hamel \cite{BH2} who proved that for any two almost-planar fronts $u_1(t,x,y)$ and $u_2(t,x,y)$, there is $T\in\R$ such that either $u_1(t+T,x,y)>u_2(t,x,y)$ or $u_1(t+T,x,y)=u_2(t,x,y)$.
\vskip 0.3cm

\begin{proof}[Proof of Theorem~\ref{th3}]
Assume that there is another curved front $V^*(t,x,y)$ satisfying \eqref{Vf}. By \eqref{Vf}, there is $R>0$ large enough such that
$$
V(t,x,y),\ V^*(t,x,y)\le \sigma \hbox{ for $(t,x,y)\in\omega^+$ and } V(t,x,y),\ V^*(t,x,y)\ge 1-\sigma \hbox{ for $(t,x,y)\in\omega^-$},
$$
where $\sigma$ is defined in (F3),
\begin{align*}
\omega_t^+:=\{(t,x,y)\in\R\times\R^2; x\le 0 \hbox{ and }& x\cos\alpha+y\sin\alpha-c_{\alpha}t\ge c_{\alpha} R\}\cup \{(t,x,y)\in\R\times\R^2;\\
 &x> 0 \hbox{ and } x\cos\beta+y\sin\beta-c_{\beta}t\ge c_{\beta}R\},
\end{align*}
and
\begin{align*}
\omega_t^-:=\{(t,x,y)\in\R\times\R^2; x\le 0 \hbox{ and }& x\cos\alpha+y\sin\alpha-c_{\alpha}t\le -c_{\alpha} R\}\cup \{(t,x,y)\in\R\times\R^2;\\
 &x> 0 \hbox{ and } x\cos\beta+y\sin\beta-c_{\beta}t\le -c_{\beta}R\}.
\end{align*}
Since $c_{\alpha}/\sin\alpha=c_{\beta}/\sin\beta$, one knows that $\omega_t^+$ and $\omega_t^-$ are connected. 
By a similar proof as of Lemma~\ref{lemma-ut'}, there is $0<\sigma'\le \sigma$ such that
$$\sigma'\le V(t,x,y),\ V^*(t,x,y)\le 1-\sigma', \hbox{ in $\R\times\R^2\setminus (\omega_t^+\cup \omega_t^-)$}.$$
Then, by taking $\tau$ large enough, one has
\begin{align*}
V^*(t-\tau,x,y)\le \sigma' \le V(t,x,y), &\hbox{ for $(t,x,y)\in\R\times\R^2\setminus (\omega_t^+\cup \omega_t^-)$},
\end{align*}
and
\begin{align*}
V(t,x,y)\ge 1-\sigma'\ge V^*(t-\tau,x,y), &\hbox{ for $(t,x,y)\in\R\times\R^2\setminus (\omega_{t-\tau}^+\cup\omega_{t-\tau}^-)$},
\end{align*}
It means that
$$V^*(t-\tau,x,y)\le V(t,x,y), \hbox{ in $\overline{\R\times\R^2\setminus (\omega_{t-\tau}^+\cup \omega_t^-)}$}.$$
Since $f(x,y,u)$ is nonincreasing in $u\in (-\infty,\sigma]$ and $u\in [1-\sigma,+\infty)$ for $(x,y)\in\mathbb{T}^2$ and by the same line of the proof of \cite[Lemma~4.2]{BH2}, one can get that
$$V^*(t-\tau,x,y)\le V(t,x,y), \hbox{ for all $(t,x,y)\in\omega^-_{t-\tau}$ and $(t,x,y)\in\omega^+_{t}$}.$$
and hence,
$$V^*(t-\tau,x,y)\le V(t,x,y), \hbox{ for all $(t,x,y)\in\R\times\R^2$ and $\tau$ large enough}.$$
Now, we decrease $\tau$ and let
$$\tau_*=\inf\{\tau\in\R; V^*(t-\tau,x,y)\le V(t,x,y), \hbox{ for $(t,x,y)\in\R\times\R^2$}\}.$$
Since both $V(t,x,y)$ and $V^*(t,x,y)$ satisfy \eqref{Vf}, one knows that $\tau_*\ge 0$.
Assume that $\tau_*>0$. If 
$$\inf\{V(t,x,y)-V^*(t-\tau_*,x,y); (t,x,y)\in \R\times\R^2\setminus (\omega_{t-\tau_*}^+\cup \omega_t^-)\}>0,$$
then there is $\eta>0$ such that
$$V^*(t-(\tau_*-\eta),x,y)\le V(t,x,y), \hbox{ for $(t,x,y)\in \R\times\R^2\setminus (\omega_{t-\tau_*}^+\cup \omega_t^-)$}.$$
By applying above arguments again, one can get that
$$V^*(t-(\tau_*-\eta),x,y)\le V(t,x,y), \hbox{ for $(t,x,y)\in\R\times\R^2$},$$
which contradicts the definition of $\tau_*$. Thus, 
$$\inf\{V(t,x,y)-V^*(t-\tau_*,x,y); (t,x,y)\in \R\times\R^2\setminus (\omega_{t-\tau_*}^+\cup \omega_t^-)\}=0, $$
and there is a sequence $\{(t_n,x_n,y_n)\}_{n\in\mathbb{N}}$ such that 
$$V(t_n,x_n,y_n)-V^*(t_n-\tau_*,x_n,y_n)\rightarrow 0, \hbox{ as $n\rightarrow +\infty$}.$$
Then, by following similar arguments as Step~3 of the proof of Lemma~\ref{U+}, one can get a contradiction. Thus, $\tau_*=0$.

Therefore, 
$$V(t,x,y)\ge V^*(t,x,y), \hbox{ for all $(t,x,y)\in\R\times\R^2$}.$$
The same arguments can be applied by changing positions of $V(t,x,y)$ and $V^*(t,x,y)$, and then, we can get that
$$V^*(t,x,y)\ge V(t,x,y), \hbox{ for all $(t,x,y)\in\R\times\R^2$}.$$
In conclusion, $V^*(t,x,y)\equiv V(t,x,y)$.
\end{proof}

\subsection{Stability of the curved front}
Take any $0<\alpha<\beta<\pi$ such that Theorem~\ref{th1} holds. Since $g'(\alpha)<0$, one can take $\alpha_1\in(0,\alpha)$ such that 
$$\frac{c_{\theta}}{\sin\theta}>\frac{c_{\alpha}}{\sin\alpha}, \hbox{ for $\theta\in [\alpha_1,\alpha]$}.$$
Similar as Lemma~\ref{lemma-psi}, there is a smooth function $\varphi_1(x)$ with $y=-x\cot\alpha$ and $y=-x\cot \alpha_1$ being its asymptotic lines and there are positive constant $k_3$, $k_4$ and $K_4$ such that
\begin{eqnarray}\label{K_4}
\left\{\begin{array}{lll}
\varphi_1''(x)<0, &&\hbox{ for all $x\in\R$},\\
-\cot\alpha>\varphi_1'(x)>-\cot\alpha_1, &&\hbox{ for all $x\in\R$},\\
k_1\text{sech}(x)\le -\cot\alpha-\varphi_1'(x)\le K_4 \text{sech}(x), &&\hbox{ for $x<0$},\\
k_2\text{sech}(x)\le \varphi_1'(x)+\cot\alpha_1\le K_4 \text{sech}(x), &&\hbox{ for $x\ge 0$},\\
\max(|\varphi_1''(x)|,|\varphi_1'''(x)|)\le K_4 \text{sech}(x), &&\hbox{ for all $x\in\R$}.
\end{array}
\right.
\end{eqnarray}
Take a constant $\varrho$ which will be determined later.  For every point $(x,y)$ on the curve $y=\varphi_1(\varrho x)/\varrho$, there is a unit normal
$$e(x)=(e_1(x),e_2(x))=\Big(-\frac{\varphi_1'(\varrho x)}{\sqrt{\varphi_1'^2(\varrho x)+1}},\frac{1}{\sqrt{\varphi_1'^2(\varrho x)+1}}\Big).
$$

For $(x,y)\in\R^2$ and $t\in\R$, take a constant $\varepsilon$ and we define
$$U_{1}^-(t,x,y)=U_{e(x)}(\underline{\xi}(t,x,y),x,y)+\varepsilon \text{sech}(\varrho x),$$
where  
$$
\underline{\xi}(t,x,y)=\frac{y-c_{\alpha\beta}t-\varphi_1(\varrho x)/\varrho}{\sqrt{\varphi_1'^2(\varrho x)+1}}.
$$

\begin{lemma}\label{U1-}
There exist $\varepsilon_0$ and $\varrho(\varepsilon_0)$ such that for any $0<\varepsilon\le \varepsilon_0$ and $0<\varrho\le \varrho(\varepsilon_0)$, the function $U_1^-(t,x,y)$ is a subsolution of \eqref{eq1.1}. Moreover, it satisfies
\be\label{U1-leU-}
\lim_{R\rightarrow +\infty} \sup_{x<-R} \Big| U_1^-(t,x,y)-U_{\alpha}(x\cos\alpha+y\sin\alpha-c_{\alpha}t,x,y)\Big|\le 2\varepsilon,
\ee
and
\be\label{U1-le2e}
U_1^-(t,x,y)\le U_{\alpha}(x\cos\alpha+y\sin\alpha-c_{\alpha}t,x,y), \hbox{ for all $t\in\R$ and $(x,y)\in\R^2$}.
\ee
\end{lemma}

\begin{proof}
Assume that
$$\varepsilon_0\le \frac{\sigma}{2},$$
where $\sigma>0$ is defined in (F3).
More restrictions on $\varepsilon_0$ will be given later. It follows from similar computation as Step~1 of the proof of Lemma~\ref{U+} that
\begin{align*}
NU_1^-:=&(U^-_1)_t-\Delta_{x,y} U^-_1 -f(x,y,U_1^-)\\
=& \partial_{\xi} U_{e(x)} \underline{\xi}_t -\partial_{\xi\xi} U_{e(x)} (\underline{\xi}_x^2+\underline{\xi}_y^2) -2\nabla_{x,y}\partial_{\xi} U_{e(x)}\cdot (\underline{\xi}_x,\underline{\xi}_y) -\Delta_{x,y} U_{e(x)}-\partial_{\xi} U_{e(x)} \underline{\xi}_{xx} \\
&-U''_{e(x)} \cdot e'(x)\cdot e'(x)-U'_{e(x)}\cdot e''(x)-2\partial_{\xi}U'_{e(x)}\cdot e'(x) \underline{\xi}_x -2\partial_xU'_{e(x)}\cdot e'(x) \\
&-\varepsilon \varrho^2 \text{sech}''(\varrho x)-f(x,y,U_1^-)\\
=& (c_{e(x)}-\underline{\xi}_t)\partial_{\xi} U_{e(x)}  -\partial_{\xi\xi} U_{e(x)} (\underline{\xi}^2_x+\underline{\xi}^2_y-1)-2\partial_x\partial_{\xi} U_{e(x)}(\underline{\xi}_x-e_1(x)) -\partial_{\xi} U_{e(x)} \underline{\xi}_{xx}\\
& -U''_{e(x)} \cdot e'(x)\cdot e'(x)-U'_{e(x)}\cdot e''(x)-2\partial_{\xi}U'_{e(x)}\cdot e'(x) \underline{\xi}_x -2\partial_xU'_{e(x)}\cdot e'(x) \\
&-\varepsilon \varrho^2 \text{sech}''(\varrho x)+f(x,y,U_{e(x)})-f(x,y,U_1^-),
\end{align*}
where $U_{e(x)}$, $\partial_{\xi}U_{e(x)}$, $\partial_{\xi\xi} U_{e(x)}$, $\nabla_{x,y}\partial_{\xi} U_{e(x)}$, $\Delta_{x,y}U_{e(x)}$, $U''_{e(x)}\cdot e'(x)\cdot e'(x)$, $U'_{e(x)}\cdot e''(x)$, $\partial_{\xi}U'(e(x))\cdot e'(x)$, $\partial_x U'_{e(x)}\cdot e'(x)$ are taking values at $(\xi(t,x,y),x,y)$ and $U_1^-$, $\underline{\xi}_t$, $\underline{\xi}_x$, $\underline{\xi}_y$ are taking values at $(t,x,y)$.
Similar as \eqref{C5}, \eqref{C6} in the proof of Lemma~\ref{U+}, there are $C_5>0$ and $C_6>0$ such that
\be\label{U1-C5}
|\partial_{\xi\xi} U_{e(x)} (\xi^2_x+\xi^2_y-1)| +2|\partial_x\partial_{\xi} U_{e(x)}(\xi_x-e_1(x))|+|\partial_{\xi} U_{e(x)} \xi_{xx}|\le C_5 \varrho\text{sech}(\varrho x),
\ee
and
\be\label{U1-C6}
|U''_{e(x)} \cdot e'(x)\cdot e'(x)|+|U'_{e(x)}\cdot e''(x)|
+2|\partial_{\xi}U'_{e(x)}\cdot e'(x) \xi_x| +2|\partial_xU'_{e(x)}\cdot e'(x)|\le C_6\varrho\text{sech}(\varrho x).
\ee

By a similar proof as of Claim~\ref{claim1}, we can easily get that
$$C_{e(x)}-\underline{\xi}_t>0, \hbox{ for $x\in\R$}.$$
and there is $C_7>0$ such that 
$$c_{e(x)}-\underline{\xi}_t=c_{e(x)}-\frac{c_{\alpha\beta}}{\sqrt{\varphi_1'^2(\varrho x)+1}}\ge C_7\text{sech}(\varrho x)>0, \hbox{ for $x$ being negative enough}.$$
Since $\varphi_1'(x)(\varrho x)\rightarrow -\cot\alpha_1$, $e(x)\rightarrow (\cos\alpha_1,\sin\alpha_1)$ as $x\rightarrow +\infty$ and $c_{\alpha_1}/\sin\alpha_1>c_{\alpha}/\sin\alpha$, there is a constant $c>0$ such that
\be\label{eq-c}
c_{e(x)}-\underline{\xi}_t=c_{e(x)}-\frac{c_{\alpha\beta}}{\sqrt{\varphi_1'^2(\varrho x)+1}}\ge c>0, \hbox{ for all $x\ge 0$}.
\ee

For $x<0$, one can make similar arguments as in the proof of Lemma~\ref{U+} to get that $NU^-_1\le 0$.
For $x\ge 0$, one can get from \eqref{U1-C5}, \eqref{U1-C6}, \eqref{eq-c} and Lemma~\ref{lemma-sech} that
\begin{align}
NU^-_1\le & c\partial_{\xi} U_{e(x)}  +(C_5+C_6)\varrho\text{sech}(\varrho x) +2\varepsilon \varrho^2 \text{sech}(\varrho x)+f(x,y,U_{e(x)})-f(x,y,U_1^-).\label{LU1-2}
\end{align}
For $(t,x,y)\in\R\times\R^2$ such that $\underline{\xi}(t,x,y)\ge C$ and $\underline{\xi}(t,x,y)\le -C$ where $C$ is defined by \eqref{eq-C}, it follows from (F3) and $\varepsilon\le \varepsilon_0\le \sigma/2$ that
$$f(x,y,U_{e(x)})-f(x,y,U^-_1)\le -\lambda \varepsilon \text{sech}(\varrho x).$$
Since $\partial_{\xi} U_e<0$, one has that 
$$NU^-_1\le \Big((C_5+C_6)\varrho+2\varepsilon \varrho^2-\lambda\varepsilon\Big)\text{sech}(\varrho x)\le0,$$
by taking $\varrho(\varepsilon)>0$ small enough such that
\be\label{de2}
(C_5+C_6)\varrho+2\varepsilon \varrho^2-\lambda\varepsilon<0,
\ee
and $0<\varrho\le \varrho(\varepsilon)$.
Finally, for $(t,x,y)\in\R\times\R^2$ such that $-C\le \xi(t,x,y)\le C$, there is $k>0$ such that 
$$-\partial_{\xi} U_e(\xi,x,y)\ge k \hbox{ for all $e\in\mathbb{S}$}.$$
Notice that 
$$f(x,y,U_{e(x)})-f(x,y,U^+)\le M \varepsilon \text{sech}(\varrho x),$$
where $M:=\max_{(x,y,u)\in\mathbb{T}^2\times\R} |f_u(x,y,u)|$. Thus, it follows from \eqref{LU1-2} and \eqref{de2} that
$$NU^-_1\le -kc+\Big((C_5+C_6)\varrho+2\varepsilon \varrho^2+M\varepsilon\Big)\text{sech}(\varrho x)\le -kc +(\lambda+M)\varepsilon\text{sech}(\varrho x)\le 0,$$
by taking $\varepsilon_0=\min\{\sigma/2, k c/(\lambda+M)\}$ and $0<\varepsilon\le \varepsilon_0$.

By similar arguments as in Step 2 of the proof of Lemma~\ref{U+}, one gets that \eqref{U1-leU-} holds. The inequality \eqref{U1-le2e} can be gotten by comparing $U_1^-(t,x,y)$ with $U_{\alpha}(x\cos\alpha+y\sin\alpha-c_{\alpha}t,x,y)$ through similar arguments as in Step 3 of the proof of Lemma~\ref{U+}. This completes the proof.
\end{proof}
\vskip 0.3cm

Similarly, since $g'(\beta)>0$, one can take $\beta_1\in(\beta,\pi)$ such that
$$\frac{c_{\beta}}{\sin\beta}<\frac{c_{\theta}}{\sin\theta}, \hbox{ for all $\theta\in (\beta,\beta_1]$}.$$
Similar as Lemma~\ref{lemma-psi}, there is a smooth function $\varphi_2(x)$ with $y=-x\cot\beta$ and $y=-x\cot \beta_1$ being its asymptotic lines and there are positive constant $k_5$, $k_6$ and $K_5$ such that
\begin{eqnarray}\label{K_5}
\left\{\begin{array}{lll}
\varphi_2''(x)<0,&&\hbox{ for all $x\in\R$},\\
-\cot\beta_1<\psi'_2(x)>-\cot\beta,&&\hbox{ for all $x\in\R$},\\ 
k_5\text{sech}(x)\le -\cot\beta_1-\varphi_2'(x)\le K_5 \text{sech}(x), &&\hbox{ for $x<0$},\\
k_6\text{sech}(x)\le \varphi_2'(x)+\cot\beta\le K_5 \text{sech}(x), &&\hbox{ for $x\ge 0$},\\
\max(|\varphi_2''(x)|,|\varphi_2'''(x)|)\le K_5 \text{sech}(x), &&\hbox{ for all $x\in\R$}.
\end{array}
\right.
\end{eqnarray}
Take a constant $\varrho$ which will be determined later.  For every point $(x,y)$ on the curve $y=\varphi_2(\varrho x)/\varrho$, there is a unit normal
$$e(x)=(e_1(x),e_2(x))=\Big(-\frac{\varphi_2'(\varrho x)}{\sqrt{\varphi_2'^2(\varrho x)+1}},\frac{1}{\sqrt{\varphi_2'^2(\varrho x)+1}}\Big).
$$

For $(x,y)\in\R^2$ and $t\in\R$, take a constant $\varepsilon$ and we define
$$U_{2}^-(t,x,y)=U_{e(x)}(\underline{\xi}(t,x,y),x,y)+\varepsilon \text{sech}(\varrho x),$$
where  
$$
\underline{\xi}(t,x,y)=\frac{y-c_{\alpha\beta}t-\varphi_2(\varrho x)/\varrho}{\sqrt{\varphi_2'^2(\varrho x)+1}}.
$$
Similar as Lemma~\ref{U1-}, we can prove the following lemma.

\begin{lemma}\label{U2-}
There exist $\varepsilon_0$ and $\varrho(\varepsilon_0)$ such that for any $0<\varepsilon\le \varepsilon_0$ and $0<\varrho\le \varrho(\varepsilon_0)$, the function $U_2^-(t,x,y)$ is a subsolution of \eqref{eq1.1}. Moreover, it satisfies
$$\lim_{R\rightarrow +\infty} \sup_{x>R} \Big| U_2^-(t,x,y)-U_{\beta}(x\cos\beta+y\sin\beta-c_{\beta}t,x,y)\Big|\le 2\varepsilon,$$
and
$$U_2^-(t,x,y)\le U_{\beta}(x\cos\beta+y\sin\beta-c_{\beta}t,x,y), \hbox{ for all $t\in\R$ and $(x,y)\in\R^2$}.$$
\end{lemma}

Then, we need the following sub and supersolutions for the Cauchy problems of \eqref{eq1.1}.

\begin{lemma}\label{lemma-subsup}
For any function $u(t,x,y)\in C^{1,2}(\R\times\R^2)$, if it is a subsolution of \eqref{eq1.1} for $(t,x,y)\in\R\times\R^2$ with $u_t>0$ and for any $0<\sigma_1<1/2$ there is a positive constant $k$ such that
\be\label{ut-k}
u_t\ge k, \hbox{ for $(t,x,y)\in\R\times\R^2$ such that $\sigma_1\le u(t,x,y)\le 1-\sigma_1$},
\ee
then for any $0<\delta<\sigma/2$ where $\sigma$ is defined in (F3), there exist positive constants $\omega$ and $\lambda$ such that
$$\underline{u}(t,x,y)=u(t+\omega\delta e^{-\lambda t}-\omega\delta,x,y)-\delta e^{-\lambda t},$$
is a subsolution of \eqref{eq1.1} for $t\ge 0$ and $(x,y)\in\R^2$. Similarly, if $u(t,x,y)$ is a smooth supersolution satisfying \eqref{ut-k}, then for any $0<\delta<\sigma/2$, there exist positive constants $\omega$ and $\lambda$ such that
$$\overline{u}(t,x,y)=u(t-\omega \delta e^{-\lambda t}+\omega\delta,x,y)+\delta e^{-\lambda t}$$
is a supersolution of \eqref{eq1.1} for $t\ge 0$ and $(x,y)\in\R^2$.
\end{lemma}

\begin{proof}
We only prove for the subsolution. Similar arguments can be applied for the supersolution. Take any $0<\delta<\sigma/2$ where $\sigma$ is defined in (F3). Let $k>0$ such that $u_t\ge k$ for $(t,x,y)\in\R\times\R^2$ such that $\sigma\le u\le 1-\sigma/2$.
Take $\omega>0$ such that
$$k\omega\ge \frac{\lambda +M}{\lambda},$$
where $\lambda$ is defined in (F3) and $M:=\max_{(x,y,u)\in\mathbb{T}^2\times\R} |f_u(x,y,u)|$.

We then check that 
$$N\underline{u}:=\underline{u}_t-\Delta_{x,y}\underline{u}-f(x,y,\underline{u})\le 0, \hbox{ for $t> 0$ and $(x,y)\in\R^2$}.$$
By computation, one can get that
\begin{align*}
N\underline{u}=-\omega\delta \lambda e^{-\lambda t} u_t+\delta \lambda e^{-\lambda t}+f(x,y,u(t+\omega \delta e^{-\lambda t}-\omega \delta,x,y))-f(x,y,\underline{u}).
\end{align*}
For $t> 0$ and $(x,y)\in\R^2$ such that $1-\sigma/2 \le u(t+\omega\delta e^{-\lambda t},x,y)\le 1$ and $0\le u(t+\omega \delta e^{-\lambda t},x,y)\le \sigma/2$ respectively, one has that $ \underline{u}(t,x,y)\ge 1-\sigma$ and $\underline{u}(t,x,y)\le \sigma$ respectively. Then, by \eqref{lambda}, it follows that
$$f(x,y,u(t+\omega \delta e^{-\lambda t}-\omega \delta,x,y))-f(x,y,\underline{u})\le -\lambda \delta e^{-\lambda t}.$$
Thus, by $u_t>0$, we have
$$N\underline{u}\le \delta\lambda e^{-\lambda t} -\lambda \delta e^{-\lambda t}\le 0.$$
For $t> 0$ and $(x,y)\in\R^2$ such that $\delta \le u(t+\omega \delta e^{-\lambda t},x,y)\le 1-\sigma/2$, one has that
$$N\underline{u}\le -k\omega\delta\lambda e^{-\lambda t} +\delta \lambda e^{-\lambda t}+M \delta e^{-\lambda t}\le 0,$$
by the definition of $\omega$.

This completes the proof.
\end{proof}
\vskip 0.3cm

Now, we are ready to prove the stability of the curved front of Theorem~\ref{th1}.

\begin{proof}[Proof of Theorem~\ref{th4}]
Let
$$\underline{U}(t,x,y)=\max\{U_1^-(t+\omega \delta e^{-\lambda t}-\omega \delta,x,y)-\delta e^{-\lambda t},U_2^-(t+\omega \delta e^{-\lambda t}-\omega \delta,x,y)-\delta e^{-\lambda t}\},$$
and
$$\overline{U}(t,x,y)=U^+(t-\omega\delta e^{-\lambda t}+\omega \delta,x,y)+\delta e^{-\lambda t},$$
where $\omega$, $\delta$ and $\lambda$ are defined in Lemma~\ref{lemma-subsup}.

By \eqref{eq-u0}, there is $R_{\delta}>0$ such that
$$U_{\alpha\beta}^-(0,x,y)-\delta\le u_0(x,y)\le U^-_{\alpha\beta}(0,x,y)+\delta,\hbox{ for $x^2+y^2>R_{\delta}^2$}.$$
Then, it follows from Lemma~\ref{U+}, Lemma~\ref{U1-} and Lemma~\ref{U2-} that
$$\underline{U}(0,x,y)\le u_0(x,y)\le \overline{U}(0,x,y), \hbox{ for $x^2+y^2>R_{\delta}^2$}.$$
By the definition of $\psi(x)$, one has that $\psi(0)>0$. Thus, 
$$\xi(0,x,y)=\frac{y-\psi(\varrho x)/\varrho}{\sqrt{\psi'^2(\varrho x)}+1}\rightarrow -\infty \hbox{ as $\varrho\rightarrow 0$ for $x^2+y^2\le R_{\delta}^2$},$$
which implies $U_{e(x)}(\xi(0,x,y),x,y)\rightarrow 1$ as $\varrho\rightarrow 0$ for $x^2+y^2\le R_{\delta}^2$. 
By taking $\varrho$ small enough, one can make that
$$u_0(x,y)\le 1\le U^+(0,x,y)+\delta=\overline{U}(0,x,y), \hbox{ for $x^2+y^2\le R_{\delta}^2$}.$$
Similarly, since $\varphi_1(0)<0$ and $\varphi_2(0)<0$, one can make that
$$u_0(x,y)\ge 0\ge \underline{U}(0,x,y), \hbox{ for $x^2+y^2\le R_{\delta}^2$},$$
by taking $\varrho$ small enough. Notice that by taking $\varrho$ small enough, $\delta$ can be arbitrary small.
Thus, we have that
$$\underline{U}(0,x,y)\le u_0(x,y)\le \overline{U}(0,x,y), \hbox{ for all $(x,y)\in\R^2$}.$$
On the other hand, by Lemma~\ref{lemma-ut'}, one knows that $U_1^-(t,x,y)$, $U^-_2(t,x,y)$ and $U^+(t,x,y)$ satisfy \eqref{ut-k}. 
By Lemma~\ref{lemma-subsup} and the comparison principle, one can get that
$$\underline{U}(t,x,y)\le u(t,x,y)\le \overline{U}(t,x,y), \hbox{ for $t\ge 0$ and $(x,y)\in\R^2$}.$$
Take a sequence $t_n=L_2n/c_{\alpha\beta}$ where $L_2$ is the period of $y$. Then, $t_n\rightarrow +\infty$ as $n\rightarrow +\infty$. By parabolic estimates, the sequence $u_n(t,x,y):=u(t+t_n,x,y+L_2 n)$ converges, locally uniformly in $\R\times\R^2$, to a solution $u_{\infty}(t,x,y)$ of \eqref{eq1.1}. Since $U_1^-(t+t_n,x,y+L_2 n)=U_1^-(t,x,y)$, $U_2^-(t+t_n,x,y+L_2 n)=U_2^-(t,x,y)$ and $U^+(t+t_n,x,y+L_2 n)=U^+(t,x,y)$, one has that
\be\label{U1unU+}
\begin{aligned}
\max\{U_1^-(t+\omega \delta e^{-\lambda (t+t_n)}-&\omega \delta,x,y)-\delta e^{-\lambda (t+t_n)},U_2^-(t+\omega \delta e^{-\lambda (t+t_n)}-\omega \delta,x,y)-\delta e^{-\lambda (t+t_n)}\}\\
&\le u_n(t,x,y)\le U^+(t-\omega\delta e^{-\lambda (t+t_n)}+\omega \delta,x,y)+\delta e^{-\lambda (t+t_n)},
\end{aligned}
\ee
and by passing $n\rightarrow +\infty$,
$u_{\infty}(t,x,y)$ satisfies
$$\max\{U_1^-(t-\omega\delta,x,y),U_2^-(t-\omega\delta,x,y)\}\le u_{\infty}(t,x,y)\le U^+(t+\omega\delta,x,y), \hbox{ for $(t,x,y)\in\R\times\R^2$}.$$
Let $U^-_{12}(t,x,y)=\max\{U_1^-(t,x,y), U^-_2(t,x,y)\}$. Then, by Lemmas \ref{U1-}, \ref{U2-} and similar arguments as Step~2 of the proof of Lemma~\ref{U+}, one can get that
\be\label{U12ab}
\lim_{R\rightarrow +\infty} \sup_{x^2+(y-c_{\alpha,\beta} t)^2>R^2} \Big| U^-_{12}(t,x,y)-U^-_{\alpha\beta}(t,x,y)\Big|\le 2\varepsilon.
\ee
By Lemma~\ref{U+}, \eqref{U12ab} and $\delta$, $\varepsilon$ can be taken arbitrary small, we have that
$$
\lim_{R\rightarrow +\infty} \sup_{x^2+(y-c_{\alpha\beta} t)^2>R^2} \Big| u_{\infty}(t,x,y)-U^-_{\alpha\beta}(t,x,y)\Big|=0,$$
By uniqueness of the curved front, we have $u_{\infty}(t,x,y)\equiv V(t,x,y)$. Thus, for any $\varepsilon>0$, it follows from \eqref{Vf}, \eqref{U1unU+}, \eqref{U12ab}, Lemma~\ref{U+} and taking $\delta$ small enough that  there is $t_0>0$ large enough such that 
$$|u(t_0,x,y)-V(t_0,x,y)|\le \varepsilon, \hbox{ for all $(x,y)\in\R^2$}.$$
On the other hand, by $V_t>0$ and a similar proof as of Lemma~\ref{lemma-ut'}, one knows that $V(t,x,y)$ satisfies \eqref{ut-k}. 
By Lemma~\ref{lemma-subsup} again and the comparison principle, one gets that
$$V(t_0 +t+\omega\varepsilon e^{-\lambda t}-\omega\varepsilon,x,y)-\varepsilon e^{-\lambda t}\le u(t_0+t,x,y)\le V(t_0-\omega\varepsilon e^{-\lambda t}+\omega\varepsilon,x,y)+\varepsilon e^{-\lambda t},$$
for $t\ge 0$ and $(x,y)\in\R^2$. Then, since $\varepsilon$ can be arbitrary small, one finally has that
$$u(t,x,y)\rightarrow V(t,x,y), \hbox{ as $t\rightarrow +\infty$ uniformly in $\R\times\R^2$}.$$
This completes the proof.
\end{proof}

\section{A curved front with varying interfaces}
In this section, we construct a curved front with varying interfaces. It behaves as three pulsating fronts as $t\rightarrow -\infty$ and as two pulsating fronts as $t\rightarrow +\infty$. We can not apply the idea of Hamel \cite{H} by considering a Neumann boundary problem in the half plane $x<0$ since our problem is not orthogonal symmetric with respect to $y$-axis in general. 

Let $\alpha$, $\beta$ and $\theta$ satisfy Theorem~\ref{th5}. We will need the following properties.

\begin{lemma}\label{lemma-c1c2}
It holds that
$$c_{\alpha\theta}e_{\alpha\theta}=\left(\frac{c_{\alpha}\sin\theta-c_{\theta}\sin\alpha}{\sin(\theta-\alpha)},\frac{c_{\alpha}\cos\theta-c_{\theta}\cos\alpha}{\sin(\alpha-\theta)}\right):=(c_1,c_2),$$
and
$$c_{\beta\theta}e_{\beta\theta}=\left(\frac{c_{\beta}\sin\theta-c_{\theta}\sin\beta}{\sin(\theta-\beta)},\frac{c_{\beta}\cos\theta-c_{\theta}\cos\beta}{\sin(\beta-\theta)}\right):=(\hat{c}_1,\hat{c}_2),$$
with $c_1>0$ and $\hat{c}_1<0$. Moreover, 
$$\frac{c_{\alpha}}{e_{\alpha\theta}\cdot (\cos\alpha,\sin\alpha)}=\frac{c_{\theta}}{e_{\alpha\theta}\cdot (\cos\theta,\sin\theta)}=c_{\alpha\theta}>\frac{c_{\theta_1}}{e_{\alpha\theta}\cdot (\cos\theta_1,\sin\theta_1)},\hbox{ for any $\theta_1\in(\alpha,\theta)$}, $$
and
$$\frac{c_{\beta}}{e_{\beta\theta}\cdot (\cos\beta,\sin\beta)}=\frac{c_{\theta}}{e_{\beta\theta}\cdot (\cos\theta,\sin\theta)}=c_{\beta\theta}>\frac{c_{\theta_2}}{e_{\beta\theta}\cdot (\cos\theta_2,\sin\theta_2)},\hbox{ for any $\theta_2\in(\theta,\beta)$}.$$
\end{lemma}

\begin{proof}
Assume by contradiction that $c_{\alpha\theta}e_{\alpha\theta}\neq (c_1,c_2)$. Take a sequence $\{t_n\}_{n\in\mathbb{N}}$ such that $t_n\rightarrow +\infty$. Then, for the sequence 
$$(x_n,y_n)=(c_1,c_2)t_n,$$
one has that $((x_n,y_n)-c_{\alpha\theta}e_{\alpha\theta}t_n)^2\rightarrow +\infty$ as $n\rightarrow +\infty$ since $c_{\alpha\theta}e_{\alpha\theta}\neq (c_1,c_2)$. Notice that for any $n$, there are $k^1_n$, $k^2_n\in\mathbb{Z}$ and $x'_n$, $y'_n\in [0,L_2)$ such that $x_n=k^1_n L_1+x'_n$ and $y_n=k^2_n L_2+y'_n$. Moreover, up to extract subsequences of $x_n$ and $y_n$, there are $x'_*\in [0,L_1]$ and $y'_*\in [0,L_2]$ such that $x_n'\rightarrow x'_*$ and $y'_n\rightarrow y'_*$ as $n\rightarrow +\infty$. Since $f(x,y,\cdot)$ is $L$-periodic in $(x,y)$, one has $f(x+x_n,y+y_n,\cdot)\rightarrow f(x+x'_*,y+y'_*,\cdot)$ as $n\rightarrow +\infty$.  Let $v_n(t,x,y)=V_{\alpha\theta}(t+t_n,x+x_n,y+y_n)$. By standard parabolic estimates, $v_n(t,x,y)$, up to extract of a subsequence, converges to a solution $v_{\infty}(t,x,y)$ of 
\be\label{eq-v2}
v_t-\Delta v=f(x+x'_*,y+y'_*,v), \hbox{ for $t\in\R$ and $(x,y)\in\R^2$}.
\ee
By definitions of $x_n$, $y_n$, $c_1$ and $c_2$, one can also have that
$$U^-_{\alpha\theta}(t+t_n,x+x_n,y+y_n)\rightarrow \hat{U}^-_{\alpha\theta}(t,x,y), \hbox{ as $n\rightarrow +\infty$ uniformly in $\R\times\R^2$},$$
where
$$\hat{U}^-_{\alpha\theta}(t,x,y):=\max\{U_\alpha(x\cos\alpha+y\sin\alpha-c_{\alpha} t,x+x'_*,y+y'_*),U_\theta(x\cos\theta+y\sin\theta-c_{\theta} t,x+x'_*,y+y'_*)\}.$$ 
Moreover, since $V_{\alpha\theta}(t,x,y)$ satisfies
$$\lim_{R\rightarrow +\infty} \sup_{((x,y)-c_{e_1e_2}t e_{\alpha\theta})^2>R^2} \Big| V_{\alpha\theta}(t,x,y)-U^-_{\alpha\theta}(t,x,y)\Big|=0,$$
one then gets that
$$v_n(t,x,y)\rightarrow \hat{U}^-_{\alpha\theta}(t,x,y) \hbox{ as $n\rightarrow +\infty$ locally uniformly in $\R^2$}.$$
In implies that $v_{\infty}(t,x,y)=\hat{U}^-_{\alpha\theta}(t,x,y)$ which is impossible since $\hat{U}^-_{\alpha\theta}(t,x,y)$ is not a solution of \eqref{eq-v2}. Thus, $c_{\alpha\theta}e_{\alpha\theta}= (c_1,c_2)$. Similarly, one can prove that $c_{\beta\theta}e_{\beta\theta}=(\hat{c}_1,\hat{c}_2)$.

The signs of $c_1$ and $\hat{c}_1$ can be easily gotten from the facts $\alpha<\theta<\beta$ and $c_{\alpha}/\sin\alpha=c_{\beta}/\sin\beta>c_{\theta}/\sin\theta$.

By rotating the coordinate such that $e_{\alpha\theta}$ being the $y$-axis, then the speed of the pulsating front $U_{\theta_1}(x\cos\theta_1+y\sin\theta_1-c_{\theta_1}t,x,y)$ in direction $e_{\alpha\theta}$ can be denoted by 
$$\frac{c_{\theta_1}}{e_{\alpha\theta}\cdot (\cos\theta_1,\sin\theta_1)}.$$ By Theorem~\ref{th2}, one has that
$$\frac{c_{\alpha}}{e_{\alpha\theta}\cdot (\cos\alpha,\sin\alpha)}=\frac{c_{\theta}}{e_{\alpha\theta}\cdot (\cos\theta,\sin\theta)}=c_{\alpha\theta}
\hbox{ and }
\frac{c_{\theta_1}}{e_{\alpha\theta}\cdot (\cos\theta_1,\sin\theta_1)}<c_{\alpha\theta}, \hbox{ for any $\theta_1\in(\alpha,\theta)$}.$$
This completes the proof.
\end{proof}
\vskip 0.3cm

Let $\varphi_1(x)$ be a smooth function such that there exist $a_1<0<b_1$ such that
$$\varphi_1(x)=-x\cot\alpha, \hbox{ for $x\le a_1$}, \varphi_1(x)=-x\cot\theta, \hbox{ for $x\ge b_1$ and }  \varphi_1''(x)> 0 \hbox{ for $x\in (a_1,b_1)$}.$$
Let $\varphi_2(x)$ be a smooth function such that there exist $a_2<0<b_2$ such that
$$\varphi_2(x)=-x\cot\theta, \hbox{ for $x\le a_2$}, \varphi_2(x)=-x\cot\beta, \hbox{ for $x\ge b_2$ and } \varphi_2''(x)> 0 \hbox{ for $x\in(a_2,b_2)$}.$$
Let 
$$\psi_1(t,x)=\varphi_1(x-c_1 t) +\rho\text{sech}(x-c_1 t) +\rho\text{sech}(x-\hat{c}_1 t),$$
and
$$\psi_2(t,x)=\varphi_2(x-\hat{c}_1 t) +\rho\text{sech}(x-c_1 t) +\rho\text{sech}(x-\hat{c}_1 t).$$
By $c_1>0$, $\hat{c}_1<0$ and making $|a_1|$, $|a_2|$, $b_1$, $b_2$ large enough and $\rho$ small enough, one can let $(\psi_1)_{xx}>0$ for $t$ negative enough and $x\le (c_1+\hat{c}_1)t/2$ and $(\psi_2)_{xx}>0$ for $t$ negative enough and $x\ge (c_1+\hat{c}_1)t/2$. 
Let 
\begin{eqnarray}
\psi(t,x)=
\left\{\begin{array}{lll}
\psi_1(t,x), &&\hbox{ for $x\le (c_1+\hat{c}_1)t/2$},\\
\psi_2(t,x), &&\hbox{ for $x>(c_1+\hat{c}_1)t/2$}.
\end{array}
\right.
\end{eqnarray}

Take a constant $\varrho$ to be determined. For every point on the curve $y=\psi_1(\varrho t,\varrho x)$, there is a unit normal
$$e(t,x)=(e_1(t,x),e_2(t,x))=\left(-\frac{(\psi_1)_x(\varrho t,\varrho x)}{\sqrt{(\psi_1)_x^2(\varrho t,\varrho x)+1}},\frac{1}{\sqrt{(\psi_1)_x^2(\varrho t,\varrho x)+1}}\right).$$
For every point on the curve $y=\psi_2(\varrho t,\varrho x)$, there is a unit normal
$$\eta(t,x)=\left(-\frac{(\psi_2)_x(\varrho t,\varrho x)}{\sqrt{(\psi_2)_x^2(\varrho t,\varrho x)+1}},\frac{1}{\sqrt{(\psi_2)_x^2(\varrho t,\varrho x)+1}}\right).$$

Take $\varepsilon>0$ to be determined. For $t\in\R$ and $(x,y)\in\R^2$, define
\begin{eqnarray*}
\widetilde{U}^+(t,x,y):=\left\{\begin{aligned}
U_{e(t,x)}\left(\xi_1(t,x,y),x,y\right)+\varepsilon \text{sech}(\varrho(x-c_1 t))+\varepsilon \text{sech}(\varrho(x-\hat{c}_1 t)), &\hbox{ for $x\le \frac{c_1+\hat{c}_1}{2} t$},\\
U_{\eta(t,x)}\left(\xi_2(t,x,y),x,y\right)+\varepsilon \text{sech}(\varrho(x-c_1 t))+\varepsilon \text{sech}(\varrho(x-\hat{c}_1 t)), &\hbox{ for $x> \frac{c_1+\hat{c}_1}{2} t$},
\end{aligned}
\right.
\end{eqnarray*}
where 
$$\xi_1(t,x,y):=\frac{y-c_2 t-\psi_1(\varrho t,\varrho x)/\varrho}{\sqrt{(\psi_1)_x^2(\varrho t,\varrho x)+1}} \hbox{ and } \xi_2(t,x,y):=\frac{y-\hat{c}_2 t-\psi_2(\varrho t,\varrho x)/\varrho}{\sqrt{(\psi_2)_x^2(\varrho t,\varrho x)+1}}.$$
By the definition of $\psi_1$, $\psi_2$, $c_1$, $c_2$, $\hat{c}_1$ and $\hat{c}_2$, one can easily check that around $x=(c_1+\hat{c}_1)t/2$, 
$$U_{e(t,x)}\left(\xi_1(t,x,y),x,y\right)=U_{\theta}(x\cos\theta+y\sin\theta-c_{\theta}t,x,y)=U_{\eta(t,x)}\left(\xi_2(t,x,y),x,y\right),$$
for $t$ negative enough.
Thus, $\widetilde{U}^+(t,x,y)$ is smooth for $t$ negative enough and $(x,y)\in\R^2$.

\begin{lemma}\label{lemma4.2}
There exist $\varepsilon_0$ and $\varrho(\varepsilon_0)$ such that for any $0<\varepsilon\le \varepsilon_0$ and $0<\varrho\le \varrho(\varepsilon_0)$, the function $\widetilde{U}^+(t,x,y)$ is a supersolution of \eqref{eq1.1} for $t$ negative enough. Moreover, it satisfies
\be\label{lemma4.2-1}
\lim_{R\rightarrow +\infty}\sup_{x\le 0, ((x,y)-c_{\alpha\theta}e_{\alpha\theta} t)^2>R^2} \left|\widetilde{U}^+(t,x,y)-U^-_{\alpha\theta}(t,x,y)\right|\le 2\varepsilon,
\ee
\be\label{lemma4.2-2}
\lim_{R\rightarrow +\infty}\sup_{x>0, ((x,y)-c_{\beta\theta}e_{\beta\theta} t)^2>R^2} \left|\widetilde{U}^+(t,x,y)-U^-_{\theta\beta}(t,x,y)\right|\le 2\varepsilon,
\ee
and
\be\label{lemma4.2-3}
\begin{aligned}
\widetilde{U}^+(t,x,y)\ge\max\{&U_\alpha(x\cos\alpha+y\sin\alpha-c_{\alpha} t,x,y),U_\theta(x\cos\theta+y\sin\theta-c_{\theta} t,x,y),\\
&U_\beta(x\cos\beta+y\sin\beta-c_{\beta} t,x,y)\}, \hbox{ for $t$ negative enough and $(x,y)\in\R^2$}.
\end{aligned}
\ee
\end{lemma}

\begin{proof}
We only prove for the part $x\le (c_1+\hat{c}_1)t/2$. Take $0<\varepsilon_0\le \sigma/2$ and more restrictions on $\varepsilon_0$ will be given later. Change variables $X=x-c_1 t$ and $Y=y-c_2 t$. Then, 
$$\psi_1(t,X):=\varphi_1(X)+\rho\text{sech}(X)+\rho\text{sech}(X+(c_1-\hat{c}_1)t),$$
and
$$e(t,X)=(e_1(t,X),e_2(t,X))=\left(-\frac{(\psi_1)_X(\varrho t,\varrho X)}{\sqrt{(\psi_1)_X^2(\varrho t,\varrho X)+1}},\frac{1}{\sqrt{(\psi_1)_X^2(\varrho t,\varrho X)+1}}\right),$$
where $(\psi_1)_X$ is taking values at $(\varrho t,\varrho X)$. 
One can compute that
\begin{align*}
(\psi_1)_t(t,X)=&(c_1-\hat{c}_1)\text{sech}(X+(c_1-\hat{c}_1)t),\\
(\psi_1)_X(t,X)=&\varphi'_1(X)+ \text{sech}'(X)+ \text{sech}'(X+\varrho(c_1-\hat{c}_1)t),\\
(\psi_1)_{tX}(t,X)=&(c_1-\hat{c}_1)\text{sech}'(X+ (c_1-\hat{c}_1)t),\\
(\psi_1)_{XX}(t,X)=&\varphi''_1(X)+\text{sech}''(X)+\text{sech}''(X+(c_1-\hat{c}_1)t),\\
(\psi_1)_{XXX}(t,X)=&\varphi'''_1(X)+\text{sech}'''(X)+\text{sech}'''(X+(c_1-\hat{c}_1)t).
\end{align*}
\begin{align*}
e_t=&\left(-\frac{\varrho(\psi_1)_{tX}}{((\psi_1)_X^2+1)^{3/2}},-\frac{\varrho(\psi_1)_X(\psi_1)_{tX}}{((\psi_1)_X^2+1)^{3/2}}\right),\\
e_X=&\Big(-\frac{\varrho(\psi_1)_{XX}}{((\psi_1)_X^2+1)^{\frac{3}{2}}},-\frac{\varrho(\psi_1)_X (\psi_1)_{XX}}{((\psi_1)_X^2 +1)^{\frac{3}{2}}}\Big),
\end{align*}
and
\begin{align*}
e_{XX}=&\Big(-\frac{\varrho^2(\psi_1)_{XXX}}{((\psi_1)_X^2 +1)^{\frac{3}{2}}}+\frac{3\varrho^2(\psi_1)_X (\psi_1)_{XX}^2}{((\psi_1)_X^2 +1)^{\frac{5}{2}}},-\frac{\varrho^2(\psi_1)_{XX}^2}{((\psi_1)_X^2 +1)^{\frac{3}{2}}}-\frac{\varrho^2(\psi_1)_X (\psi_1)_{XXX}}{((\psi_1)_X^2 +1)^{\frac{3}{2}}}+\frac{3\varrho^2(\psi_1)_X^2 (\psi_1)_{XX}^2 }{((\psi_1)_X^2 +1)^{\frac{5}{2}}}\Big),
\end{align*}
where $(\psi_1)_X$, $(\psi_1)_{XX}$, $(\psi_1)_{XXX}$, $(\psi_1)_{tX}$ are taking values at $(\varrho t,\varrho X)$ in $e_t$, $e_X$, $e_{XX}$. 
Let
$$\widetilde{U}^+(t,X,Y)=U_{e(t,X)}\left(\xi_1(t,X,Y),X+c_1 t,Y+c_2 t\right)+\varepsilon \text{sech}(\varrho X)+\varepsilon \text{sech}(\varrho X-\varrho(c_1-\hat{c}_1)t),$$
where
$$\xi_1(t,X,Y)=\frac{Y-\psi_1(\varrho t,\varrho X)/\varrho}{\sqrt{(\psi_1)_x^2(\varrho t,\varrho X)+1}}.$$

We need to verify that
$$N\widetilde{U}^+:=\widetilde{U}^+_t-\Delta_{X,Y}\widetilde{U}^+-c_1 \widetilde{U}^+_X-c_2 \widetilde{U}^+_Y -f(X+c_1 t,Y+c_2 t,\widetilde{U}^+)\ge 0,$$
for $t$ negative enough and $(x,y)\in\R^2$.
By \eqref{Ue} and after some computation, one can get that
\begin{align*}
N\widetilde{U}^+=&\partial_{\xi} U_{e(t,X)} ((\xi_1)_t-c_1(\xi_1)_X-c_2(\xi_1)_Y+c_{e(t,X)}) + U'_{e(t,X)}\cdot e_t -U''_{e(t,X)}\cdot e_X\cdot e_X -U'_{e(t,X)} \cdot e_{XX} \\
&-2\partial_{\xi} U'_{e(t,X)}\cdot e_{X} (\xi_1)_X-2\partial_{X} U'_{e(t,X)}\cdot e_X -\partial_{\xi\xi} U_{e(t,X)} ((\xi_1)_X^2+(\xi_1)_Y^2-1) \\
&-2\partial_{\xi}\partial_X U_{e(t,X)}((\xi_1)_X-e_1(t,X))-2\partial_{\xi}\partial_Y U_{e(t,X)}((\xi_1)_Y-e_2(t,X)) -\partial_{\xi}U_{e(t,X)} \xi_{XX}\\
&-c_1U'_{e(t,X)}\cdot e_X -\varepsilon \varrho^2 \text{sech}''(\varrho X)- \varepsilon \varrho^2 \text{sech}''(\varrho X-\varrho(c_1-\hat{c}_1)t) -c_1\varepsilon\varrho \text{sech}'(\varrho X)\\
& -c_1\varepsilon \varrho'(\varrho X-\varrho(c_1-\hat{c}_1)t)+f(X+c_1 t,Y+c_2 t,U_{e(t,X)})-f(X+c_1 t,Y+c_2 t,\widetilde{U}^+),
\end{align*}
where $\partial_{\xi} U_{e(t,X)}$, $\partial_{\xi\xi} U_{e(t,X)}$, $\nabla_{X,Y}\partial_{\xi} U_{e(t,X)}$, $ U'_{e(t,X)}\cdot e_t$, $U''_{e(t,X)}\cdot e_X\cdot e_X$, $U'_{e(t,X)}\cdot e_{XX}$, $\partial_{\xi} U'_{e(t,X)}\cdot e_{X}$, $\partial_{X} U'_{e(t,X)}\cdot e_{X}$, $U'_{e(t,X)}\cdot e_{X}$, $U_{e(t,X)}$ are taking values at $(\xi_1(t,X,Y),X,Y)$ and $\widetilde{U}^+$, $(\xi_1)_t$, $(\xi_1)_X$, $(\xi_1)_Y$ are taking values at $(t,X,Y)$.
Similar as those formulas of \eqref{eq-xi}, one can also compute that
\be\label{fomulas}
\begin{aligned}
&(\xi_1)_t=-\frac{\varrho(\psi_1)_X(\psi_1)_{tX}}{(\psi_1)_X^2+1}\xi_1 -\frac{(\psi_1)_t}{\sqrt{(\psi_1)^2_X +1}},\\
&(\xi_1)_X=-\frac{\varrho(\psi_1)_X(\psi_1)_{XX}}{((\psi_1)_X^2+1)^{\frac{1}{2}}}\xi_1 -\frac{(\psi_1)_X}{\sqrt{(\psi_1)_X^2+1}},\\
&(\xi_1)_Y=\frac{1}{\sqrt{(\psi_1)_X^2+1}},\\
&(\xi_1)_{XX}=-\frac{\varrho^2(\psi_1)_{XX}^2 +\varrho^2(\psi_1)_X(\psi_1)_{XX}}{(\psi_1)_X^2+1}\xi_1 +\frac{3\varrho^2(\psi_1)_X^2(\psi_1)_{XX}^2}{((\psi_1)_X^2+1)^2}\xi_1 +\frac{\varrho((\psi_1)_X^2-1)(\psi_1)_{XX}}{((\psi_1)_X^2+1)^{\frac{3}{2}}},\\
&(\xi_1)^2_X+(\xi_1)^2_Y-1=\left(\Big(\frac{\varrho(\psi_1)_X(\psi_1)_{XX}}{((\psi_1)_X^2 +1)^{\frac{3}{2}}}\Big)^2\xi_1^2+2\Big(\frac{\varrho(\psi_1)_X (\psi_1)_{XX}}{((\psi_1)_X^2 +1)^{\frac{3}{2}}}\Big)\frac{(\psi_1)_X}{\sqrt{(\psi_1)_X^2 +1}}\xi_1\right),
\end{aligned}
\ee
where $(\psi_1)_X$, $(\psi_1)_t$ $(\psi_1)_{XX}$,  $(\psi_1)_{tX}$ are taking values at $(\varrho t,\varrho X)$. 
By Lemma~\ref{lemma-sech}, Lemma~\ref{ASY2}, Lemma~\ref{lemma-Ue'}, boundedness of $\|U'_e\|$, $\|U''_e\|$, $\|\partial_{\xi}U'_e\|$, $\|\partial_{x}U'_e\|$ and above formulas, there are constants $C_8>0$ and $C_9>0$ such that
\begin{align*}
|\partial_{\xi\xi}U_{e(t,X)}((\xi_1)^2_X+(\xi_1)^2_Y-1)| +2|\partial_{X}\partial_{\xi} U_{e(t,X)}((\xi_1)_X&-e_1(t,X))|+|\partial_{\xi} U_{e(t,X)}(\xi_1)_{XX}|\\
&\le C_8\varrho (\text{sech}(\varrho X)+\text{sech}(\varrho X-\varrho (c_1-\hat{c}_1)t)),
\end{align*}
and
\begin{align*}
|U'_{e(t,X)}\cdot e_t| +|U''_{e(t,X)}\cdot e_X\cdot e_X| &+|U'_{e(t,X)}\cdot e_{XX}|   +2|\partial_{\xi} U'_{e(t,X)}\cdot e_{X} (\xi_1)_X|+ 2|\partial_{x} U'_{e(t,X)}\cdot e_X|\\
&+ c_1| U'_{e(t,X)}\cdot e_X|\le C_9\varrho (\text{sech}(\varrho X)+\text{sech}(\varrho X-\varrho (c_1-\hat{c}_1)t)),
\end{align*}

Therefore, it follows that 
\begin{align*}
N\widetilde{U}^+\ge &\partial_{\xi} U_{e(t,X)} ((\xi_1)_t-c_1(\xi_1)_X-c_2(\xi_1)_Y+c_{e(t,X)}) -(C_8+C_9) \varrho (\text{sech}(\varrho X)\\
&+ \varrho\text{sech}(\varrho X +\varrho(c_1-\hat{c}_1)t)) -(1+c_1)\varepsilon \varrho^2 (\text{sech}(\varrho X)+ \varrho\text{sech}(\varrho X +\varrho(c_1-\hat{c}_1)t))\\
&+f(X+c_1 t,Y+c_2 t,U_{e(t,x)})-f(X+c_1 t,Y+c_2 t,\widetilde{U}^+)
\end{align*}

We claim that
\begin{claim}\label{claim4}
There exist positive constants $C_{10}$ and $C_{11}$ such that
\begin{align*}
c_1(\xi_1)_X+c_2(\xi_1)_Y&-(\xi_1)_t-c_{e(t,X)}\ge  -C_{10}\varrho( \text{\rm sech}(\varrho X)+\text{\rm sech}(\varrho X+\varrho(c_1-\hat{c}_1)t))|\xi_1|\\
&-C_{10} \varrho \text{\rm sech}(\varrho X+\varrho(c_1-\hat{c}_1)t))
+C_{11}  (\text{\rm sech}(\varrho X)+\text{\rm sech}(\varrho X+\varrho(c_1-\hat{c}_1)t)).
\end{align*}
\end{claim}
In order to not lengthen the proof, we postpone the proof of Claim~\ref{claim4} after the proof of this lemma.

For $\xi_1(t,X,Y)\ge C$ and $\xi_1(t,X,Y)\le -C$ where $C$ is defined by \eqref{eq-C}, it follows from \eqref{lambda} that
$$f(X+c_1 t,Y+c_2 t,U_{e(t,x)})-f(X+c_1 t,Y+c_2 t,\widetilde{U}^+)\ge \lambda\varepsilon (\text{sech}(\varrho x)+\text{sech}(\varrho x+\varrho(c_1-\hat{c}t)))$$
Then, by $\partial_{\xi} U_e<0$, Lemma~\ref{ASY2} and Claim~\ref{claim4}, it follows that
\begin{align*}
N\widetilde{U}^+\ge & -B_1 C_{10}\varrho(\text{sech}(\varrho X)+\text{sech}(\varrho X+\varrho(c_1-\hat{c}_1)t)) -B_2 C_{10} \varrho\text{sech}(\varrho X+\varrho(c_1-\hat{c}_1)t)) \\
&-\Big((C_8+C_9)\varrho+(1+c_1)\varepsilon \varrho^2\Big) (\text{sech}(\varrho X)+ \varrho\text{sech}(\varrho X +\varrho(c_1-\hat{c}_1)t))\\
&+\lambda\varepsilon(\text{sech}(\varrho X)+\text{sech}(\varrho X+\varrho(c_1-\hat{c}t)))\ge 0
\end{align*}
where $B_1=\sup_{e\in\mathbb{S}} \|\partial_{\xi} U_e \xi\|_{L^{\infty}}$ and $B_2=\sup_{e\in\mathbb{S}} \|\partial_{\xi} U_e\|_{L^{\infty}}$,
by taking $0<\varrho\le \varrho(\varepsilon)$ where $\varrho(\varepsilon)$ is small enough such that
\be\label{rho-e}
-B_1 C_{10}\varrho-B_2 C_{10} \varrho-\Big((C_8+C_9)\varrho+(1+c_1)\varepsilon \varrho^2\Big)+\lambda\varepsilon>0, \hbox{ for all $0<\varrho\le \varrho(\varepsilon)$}.
\ee
For $-C\le \xi_1(t,x,y)\le C$, there is $k>0$ such that $-\partial_{\xi}U_{e(t,x)}(\xi_1(t,x,y),x,y)\ge k$. Then, it follows from Claim~\ref{claim4} that
\begin{align*}
N\widetilde{U}^+\ge & kC_{11}  (\text{sech}(\varrho X)+\text{sech}(\varrho X+\varrho(c_1-\hat{c}_1)t)) -B_1 C_{10}\varrho(\text{sech}(\varrho X)+\text{sech}(\varrho X+\varrho(c_1-\hat{c}_1)t))\\
& -B_2 C_{10} \varrho\text{sech}(\varrho X+\varrho(c_1-\hat{c}_1)t)) -\Big((C_8+C_9)\varrho+(1+c_1)\varepsilon \varrho^2\Big) (\text{sech}(\varrho X)\\
&+ \varrho\text{sech}(\varrho X +\varrho(c_1-\hat{c}_1)t)) -M\varepsilon(\text{sech}(\varrho X)+\text{sech}(\varrho X+\varrho(c_1-\hat{c}t)))\ge 0
\end{align*}
where $M=\max_{(x,y,u)\in\R^2\times\R} |f_u(x,y,u)|$,
by \eqref{rho-e}, taking $0<\varepsilon\le \varepsilon_0$ and $\varepsilon_0=\max\{\sigma/2,kC_{11}/(\lambda+M)\}$.

By the comparison principle, $\widetilde{U}^+(t,x,y)$ is a supersolution of \eqref{eq1.1}.

By the definition of $\psi_1(x)$, $\psi_2(x)$ and Lemma~\ref{lemma-c1c2}, one has that
$$\xi_1(t,X,Y)\rightarrow X\cos\alpha+Y\sin\alpha=x\cos\alpha+y\sin\alpha-c_{\alpha}t \hbox{ as } X\rightarrow -\infty,$$
and
$$\xi_1(t,X,Y)\rightarrow X\cos\theta+Y\sin\theta=x\cos\theta+y\sin\theta-c_{\theta}t \hbox{ as } X\rightarrow +\infty.$$
Then, by similar arguments as in Step~2 of the proof of Lemma~\ref{U+}, one can get \eqref{lemma4.2-1} and \eqref{lemma4.2-2}. The inequality \eqref{lemma4.2-3} can be gotten by comparing $\widetilde{U}^+(t,x,y)$ with $U_{\alpha}(x\cos\alpha+y\sin\alpha -c_{\alpha}t,x,y)$, $U_{\beta}(x\cos\beta+y\sin\beta -c_{\beta}t,x,y)$, $U_{\theta}(x\cos\theta+y\sin\theta -c_{\theta}t,x,y)$ respectively for $t$ negative enough through similar arguments as in Step~3 of the proof of Lemma~\ref{U+}. This completes the proof.
\end{proof}
\vskip 0.3cm

We then prove Claim~\ref{claim4}.

\begin{proof}[Proof of Claim~\ref{claim4}]
From \eqref{fomulas}, one has that
\begin{align*}
c_1(\xi_1)_X+c_2(\xi_1)_Y-(\xi_1)_t-c_{e(t,X)}=&-c_1\frac{\varrho(\psi_1)_X(\psi_1)_{XX}}{((\psi_1)_X^2+1)^{\frac{1}{2}}}\xi_1- c_1 \frac{(\psi_1)_X}{\sqrt{(\psi_1)_X^2+1}}+ c_2 \frac{1}{\sqrt{(\psi_1)_X^2+1}}\\
&+\frac{\varrho(\psi_1)_X(\psi_1)_{tX}}{(\psi_1)_X^2+1}\xi_1 +\frac{(\psi_1)_t}{\sqrt{(\psi_1)^2_X +1}}-c_{e(t,X)}.
\end{align*}
Then, by Lemma~\ref{lemma-sech} and the definition of $\psi_1$, there is $C_{10}>0$ such that 
\be\label{claim5-1}
\Big|-c_1\frac{\varrho(\psi_1)_X(\psi_1)_{XX}}{((\psi_1)_X^2+1)^{\frac{1}{2}}}\xi_1+\frac{\varrho(\psi_1)_X(\psi_1)_{tX}}{(\psi_1)_X^2+1}\xi_1\Big|\le C_{10}\varrho(\text{sech}(\varrho X) +\text{sech}(\varrho X+\varrho(c_1-\hat{c}_1)t))|\xi_1|,
\ee
and
\be\label{claim5-2}
\left|\frac{(\psi_1)_t}{\sqrt{(\psi_1)^2_X +1}}\right|\le  C_{10}\varrho\text{sech}(\varrho X+\varrho(c_1-\hat{c}_1)t).
\ee
Let $\theta(t,X)=\arccos(e_1(t,X))$. Then, $e(t,X)=(\cos\theta(t,X),\sin\theta(t,X))$. By the definition of $\psi_1(t,X)$, one has $\alpha<\theta(t,X)<\theta$. It follows from Lemma~\ref{lemma-c1c2} that
\be\label{claim5-3}
\begin{aligned}
- c_1 \frac{(\psi_1)_X}{\sqrt{(\psi_1)_X^2+1}}+ c_2 \frac{1}{\sqrt{(\psi_1)_X^2+1}}-c_{e(t,X)}=&(c_1,c_2)(\cos\theta(t,X),\sin\theta(t,X))-c_{\theta(t,X)}\\
=&c_{\alpha\theta} e_{\alpha\theta}(\cos\theta(t,X),\sin\theta(t,X))-c_{\theta(t,X)}>0.
\end{aligned}
\ee
Notice that $c_e>0$ for all $e\in\mathbb{S}$. By Lemma~\ref{lemma-c1c2}, one has that
$$e_{\alpha\theta}\cdot (\cos\theta(t,X),\sin\theta(t,X))>0, \hbox{ for all $X\in\R$}.$$
Let 
$$h(s)=\frac{c_{s}}{e_{\alpha\theta}\cdot (\cos s,\sin s)}.$$
Notice that $h(\alpha)=c_{\alpha\beta}$. Also notice that $e_1(t,X)\rightarrow\cos\alpha$ as $X\rightarrow -\infty$ and $\theta(t,X)\rightarrow \alpha$ as $X\rightarrow -\infty$ for $X$ being very negative, one has that
\be\label{claim5-4}
\begin{aligned}
&c_{\alpha\theta} e_{\alpha\theta}(\cos\theta(t,X),\sin\theta(t,X))-c_{\theta(t,X)}\\
=& e_{\alpha\theta}\cdot (\cos\theta(t,X),\sin\theta(t,X)) (h(\alpha)-h(\theta(t,X)))\\
=& e_{\alpha\theta}\cdot (\cos\theta(t,X),\sin\theta(t,X)) (h'(\alpha)(\alpha-\theta(t,X))+o(|\alpha-\theta(t,X)|))
\end{aligned}
\ee
Remember that $h'(\alpha)<0$ by the assumptions of Theorem~\ref{th5}. 
Moreover, by the formulas in the proof of Lemma~\ref{lemma4.2}, there is $C_{11}>0$ such that
\be\label{claim5-5}
\begin{aligned}
\alpha-\theta(t,X)=&\int_{-\infty}^{X} \theta_X(t,s)ds=\int_{-\infty}^{X} \frac{\varrho (\psi_1)_{XX}(\varrho t,\varrho s)}{(\psi_1)^2_{X}(\varrho t,\varrho s)+1}ds\\
\ge& \frac{1}{\|(\psi_1)_X\|^2_{L^{\infty}}+1}((\psi_1)_{X}(\varrho t,\varrho X)+\cot\alpha)\ge C_{11} ( \text{sech}(\varrho X)+ \text{sech}(\varrho X+\varrho (c_1-\hat{c}t))).
\end{aligned}
\ee
By \eqref{claim5-1}-\eqref{claim5-5}, we have our conclusion.
\end{proof}
\vskip 0.3cm

Now, we turn to prove Theorem~\ref{th5}.

\begin{proof}[Proof of Theorem~\ref{th5}]
Let $u_n(t,x,y)$ be the solution of \eqref{eq1.1} for $t\ge -n$ with initial data
$$u_n(-n,x,y)=U^-_{\alpha\theta\beta}(-n,x,y),$$
where
\begin{align*}
U^-_{\alpha\theta\beta}(t,x,y)=\max\{U_\alpha(x\cos\alpha+y\sin\alpha-c_{\alpha} t,x,y),U_\theta(&x\cos\theta+ y\sin\theta-c_{\theta} t,x,y),\\
&U_\beta(x\cos\beta+y\sin\beta-c_{\beta} t,x,y)\}.
\end{align*}
By Lemma~\ref{lemma4.2}, it follows from the comparison principle that
\be\label{atble}
U^-_{\alpha\theta\beta}(t,x,y)\le u_n(t,x,y)\le \widetilde{U}^+(t,x,y), \hbox{ for $-n\le t\le T$ and $(x,y)\in\R^2$},
\ee
wherer $T$ is a negative constant such that Lemma~\ref{lemma4.2} holds for $-\infty<t\le T$. Since $U^-_{\alpha\theta\beta}(t,x,y)$ is a subsolution, the sequence $u_n(t,x,y)$ is increasing in $n$. Letting $n\rightarrow +\infty$ and by parabolic estimates, the sequence $u_n(t,x,y)$ converges to an entire solution $u(t,x,y)$ of \eqref{eq1.1}.

By \eqref{atble}, $u(t,x,y)$ satisfies
\be\label{atb}
U^-_{\alpha\theta\beta}(t,x,y)\le u(t,x,y)\le \widetilde{U}^+(t,x,y), \hbox{ for $t\le T$ and $(x,y)\in\R^2$}.
\ee
Moreover, by \eqref{lemma4.2-1}, \eqref{lemma4.2-2} and since $\varepsilon$ can be arbitrary small, one can get that $u(t,x,y)$ satisfies
\be\label{u-at}
\lim_{R\rightarrow +\infty}\sup_{x\le 0, ((x,y)-c_{\alpha\theta}e_{\alpha\theta} t)^2>R^2} \left|u(t,x,y)-U^-_{\alpha\theta}(t,x,y)\right|=0,
\ee
and
\be\label{u-bt}
\lim_{R\rightarrow +\infty}\sup_{x> 0, ((x,y)-c_{\beta\theta}e_{\beta\theta} t)^2>R^2} \left|u(t,x,y)-U^-_{\beta\theta}(t,x,y)\right|=0,
\ee
for $t$ negative enough.
Now, we consider the half plane $H:=\{(x,y)\in\R^2; x<0\}$. Take any sequence $\{t_n\}_{n\in\mathbb{N}}$ of $\R$ such that $t_n\rightarrow -\infty$ as $n\rightarrow +\infty$. Notice that for any $n$, there are $k^1_n$, $k^2_n\in\mathbb{Z}$ and $x'_n\in[0,L_1)$, $y'_n\in [0,L_2)$ such that $c_1 t_n=k^1_n L_1+x'_n$ and $c_2 t_n=k^2_n L_2+y'_n$. Moreover, up to extract subsequences of $c_1 t_n$ and $c_2 t_n$, there are $x'_*\in [0,L_1]$ and $y'_*\in [0,L_2]$ such that $x_n'\rightarrow x'_*$ and $y'_n\rightarrow y'_*$ as $n\rightarrow +\infty$. Let $v_n(t,x,y)=u(t+t_n,x+c_1 t_n,y+c_2 t_n)$ and $H_n=H-c_1 t_n$. Then, $H_n\rightarrow \R^2$ as $n\rightarrow +\infty$. Since $f(x,y,\cdot)$ is $L$-periodic in $(x,y)$, one has that $f(x+c_1 t_n,y+c_2 t_n,\cdot)\rightarrow f(x+x'_*,y+y'_*,\cdot)$ By parabolic estimates, $v_n(t,x,y)$, up to extract of a subsequence, converges to a solution $v_{\infty}(t,x,y)$ of 
\be\label{eq-v}
v_t-\Delta v=f(x+x'_*,y+y'_*,v),\quad\,(t,x,y)\in\R\times\R^2.
\ee
By the definitions of $c_1$ and $c_2$, one can easily check that 
\be\label{hatU}
U^-_{\alpha\theta}(t+t_n,x+x_n,y+y_n)\rightarrow \hat{U}^-_{\alpha\theta}(t,x,y), \hbox{ as $n\rightarrow +\infty$ uniformly in $\R\times\R^2$},
\ee
where
$$\hat{U}^-_{\alpha\theta}(t,x,y):=\max\{U_\alpha(x\cos\alpha+y\sin\alpha-c_{\alpha} t,x+x'_*,y+y'_*),U_\theta(x\cos\theta+y\sin\theta-c_{\theta} t,x+x'_*,y+y'_*)\}.$$ 
 By \eqref{u-at}, it follows that
$$\lim_{R\rightarrow +\infty}\sup_{((x,y)-c_{\alpha\theta}e_{\alpha\theta} t)^2>R^2} \left|v_{\infty}(t,x,y)-\hat{U}^-_{\alpha\theta}(t,x,y)\right|=0.$$
By uniqueness of the curved front, one then has that $v_{\infty}(t,x,y)\equiv \hat{V}_{\alpha\theta}(t,x,y)$ where $\hat{V}_{\alpha\theta}(t,x,y)$ is the curved front of \eqref{eq-v} satisfying
\be\label{hatV}
\lim_{R\rightarrow +\infty}\sup_{((x,y)-c_{\alpha\theta}e_{\alpha\theta} t)^2>R^2} \left|\hat{V}_{\alpha\theta}(t,x,y)-\hat{U}^-_{\alpha\theta}(t,x,y)\right|=0.
\ee
Thus, for any fixed $t$,
$$v_n(t,x,y)\rightarrow \hat{V}_{\alpha\theta}(t,x,y), \hbox{ as $n\rightarrow +\infty$ locally uniformly in $H_n$}.$$
By \eqref{u-at}, \eqref{hatU} and \eqref{hatV}, the above convergence is uniform in $\overline{H_n}$. Thus, for any fixed $t$,
$$u(t+t_n,x+c_1 t_n,y+ c_2 t_n)\rightarrow \hat{V}_{\alpha\theta}(t,x,y), \hbox{ as $n\rightarrow +\infty$ uniformly in $\overline{H_n}$}.$$
That implies
\be\label{uhatV}
u(t+t_n,x,y)\rightarrow \hat{V}_{\alpha\theta}(t,x-c_1 t_n,y-c_2 t_n), \hbox{ as $n\rightarrow +\infty$ uniformly in $\overline{H}$}.
\ee
By above arguments applied to $\hat{V}_{\alpha\theta}(t-t_n+t_0,x-c_1 t_n,y-c_2 t_n)$ for arbitrary $t_0\in\R$, one can get that
\begin{align*}
\hat{V}_{\alpha\theta}(t-t_n+t_0,x-c_1 t_n,y-c_2 t_n)\rightarrow V_{\alpha\theta}(t+t_0,x,y), &\hbox{ as $n\rightarrow +\infty$ locally uniformly for $t\in\R$}\\
&\hbox{ and uniformly for $(x,y)\in\R^2$}.
\end{align*}
Since $t_0$ is arbitrary, the above convergence is also uniform for $t\in\R$. Thus, by \eqref{uhatV}, one gets that
$$u(t,x,y)\rightarrow V_{\alpha\theta}(t,x,y), \hbox{ as $t\rightarrow -\infty$ uniformly in $\overline{H}$}.$$
Similarly, one can prove that $u(t,x,y)\rightarrow V_{\beta\theta}(t,x,y)$ as $t \rightarrow -\infty$ uniformly in $\R^2\setminus H$.

On the other hand, for fixed $T<0$ such that Lemma~\ref{lemma4.2} holds, one can easily check that
$$\lim_{R\rightarrow +\infty}\sup_{x^2+y^2>R^2} \left|U^-_{\alpha\theta\beta}(T,x,y)-U^-_{\alpha\beta}(T,x,y)\right|=0,$$
and
$$\lim_{R\rightarrow +\infty} \sup_{x^2+y^2>R^2} \left|\widetilde{U}^+(T,x,y)-U^-_{\alpha\beta}(T,x,y)\right|\le 2\varepsilon.$$
Since $\varepsilon$ can be arbitrary small and by \eqref{atb}, one has that 
$$\lim_{R\rightarrow +\infty} \sup_{x^2+y^2>R^2} \left|u(T,x,y)-U^-_{\alpha\beta}(T,x,y)\right|=0.$$
By stability of the curved front, that is, Theorem~\ref{th4}, one has that
$$u(t,x,y)\rightarrow V_{\alpha\beta}(t,x,y), \hbox{ as $t\rightarrow +\infty$ uniformly in $\R^2$}.$$
This completes the proof of Theorem~\ref{th5}.
\end{proof}
\vskip 0.3cm

Finally, we prove Corollary~\ref{cor3} which implies that Theorem~\ref{th5} is not empty.

\begin{proof}[Proof of Corollary~\ref{cor3}]
Assume without loss of generality that $e_*=(0,1)$. Since $c_{e_*}=\min_{e\in\mathbb{S}} \{c_e\}$ and $c'_e$ is bounded, there exist $\alpha_1\in (0,\pi/2)$ and $\beta_1\in (\pi/2,\pi)$ such that
$$\frac{d c_{\theta}}{d\theta}\Big|_{\theta=\alpha}=c'_{\alpha}\cdot (-\sin\alpha,\cos\alpha)\le 0 \hbox{ for $\alpha\in [\alpha_1,\frac{\pi}{2}]$},$$
 and 
 $$ \frac{d c_{\theta}}{d\theta}\Big|_{\theta=\beta}=c'_{\beta}\cdot (-\sin\beta,\cos\beta)\ge 0 \hbox{ for $\beta\in[\frac{\pi}{2},\beta_1]$}.$$
Let $g(\theta)=c_{\theta}/\sin\theta$. Then, 
$$g'(\theta)=\frac{c'_{\theta}\cdot (-\sin\theta,\cos\theta)}{\sin\theta}-\frac{c_{\theta}\cos\theta}{\sin^2\theta}.$$
One can make $\alpha_1$, $\beta_1$ close to $\pi/2$ such that
$$g'(\alpha)<0 \hbox{ for all $\alpha\in [\alpha_1,\frac{\pi}{2})$ and } g'(\beta)>0 \hbox{ for all $\beta\in [\frac{\pi}{2},\beta_1]$}.$$
Thus, $g(\theta)$ is decreasing from $g(\alpha)$ to $g(\pi/2)$ as $\theta$ varying from $\alpha_1$ to $\pi/2$, and is increasing from $g(\pi/2)$ to $g(\beta)$ as $\theta$ varying from $\pi/2$ to $\beta_1$. By continuity, one can pick $\alpha\in [\alpha_1,\pi/2)$ and $\beta\in (\pi/2,\beta_1]$ such that
$$
g(\alpha)=g(\beta),\ g'(\alpha)<0 \hbox{ and } g'(\beta)>0.
$$
Let $e_1=(\cos\alpha,\sin\alpha)$ and $e_2=(\cos\beta,\sin\beta)$. By Theorem~\ref{th1}, there is a curved front $V_{e_1 e_2}(t,x,y)$ of \eqref{eq1.1} satisfying \eqref{Ve1e2} with $e_0=e_*$. 

Now, by rotating the coordinate, we can assume that $e_*$ is denoted by $(\cos\theta_*,\sin\theta_*)$ where $\theta_*\in (0,\pi/2)$ is small enough. Let $e_1$ and $e_2$ be denoted by $(\cos\theta_1,\sin\theta_1)$ and $\cos\theta_2,\sin\theta_2$ respectively, where $\theta_1$ and $\theta_2$ are close to $\theta_*$. By Corollary~\ref{cor1} and since $\theta_*$ is small enough which means that $\theta_1$ is small enough, there is $\theta_3\in (\pi/2,\pi)$ such that 
$$\frac{c_{\theta_1}}{\sin\theta_1}=\frac{c_{\theta_3}}{\sin\theta_3}:=c_{\theta_1\theta_3},$$
and there is a curved front $V_{\theta_1\theta_3}$ of \eqref{eq1.1} satisfying \eqref{Vf} with $\alpha=\theta_1$, $\beta=\theta_3$ and $c_{\alpha\beta}=c_{\theta_1\theta_3}$. On the other hand, since $\theta_1$ is small enough, it implies that $\theta_3$ is close to $\pi$ enough. Then, since $\theta_2$ is also small enough, one has that $\theta_3-\theta_2$ is close to $\pi$ enough and hence, $(\cos\theta_2,\sin\theta_2)\cdot (\cos\theta_3,\sin\theta_3)=\cos(\theta_3-\theta_2)$ is close to $-1$ enough. By Corollary~\ref{cor2}, there is $e_{**}$ such that \eqref{ce1e2} holds for $e_1=(\cos\theta_2,\sin\theta_2)$, $e_2=(\cos\theta_3,\sin\theta_3)$, $e_0=e_{**}$ and there is a curved front $V_{\theta_2\theta_3}$ of \eqref{eq1.1} satisfying \eqref{Ve1e2}. 

Then, by Theorem~\ref{th5}, there is an entire solution $u(t,x,y)$ of \eqref{eq1.1} satisfying \eqref{th1.8-1} and \eqref{th1.8-2} with $\alpha=\theta_1$, $\theta=\theta_2$, $\beta=\theta_3$.
\end{proof}

\end{document}